\documentclass[12pt,a4paper]{article}
\usepackage{amsmath,amsfonts,amssymb,amsthm}
\usepackage[margin=2cm]{geometry}
\usepackage{url}
\usepackage{stmaryrd}
\usepackage{numprint}
\usepackage{verbatim}
\usepackage{enumitem}
\setlist[enumerate]{label=(\arabic*)}
\setlist[enumerate,2]{label=(\alph*)}
\setlist[enumerate,3]{label=(\roman*)}
\usepackage{appendix}
\usepackage{tabularx}
\usepackage{csquotes}
\usepackage{authblk}
\usepackage[dvipsnames]{xcolor}
\definecolor{mylinkcolor}{rgb}{0.5,0.0,0.0}
\definecolor{myurlcolor}{rgb}{0.0,.25,1}
\usepackage[colorlinks, urlcolor=myurlcolor, citecolor=myurlcolor, linkcolor=mylinkcolor,
pagebackref,
breaklinks=true
]{hyperref}
\usepackage{mathtools}
\usepackage[normalem]{ulem}

\usepackage{tikz} 
\usetikzlibrary{cd,arrows.meta,calc}
\usetikzlibrary{decorations.pathreplacing,calligraphy}

\usepackage{float}

\let\oldsection\section
\renewcommand{\section}{
\renewcommand{\theequation}{\thesection.\arabic{equation}}
\oldsection}
\let\oldsubsection\subsection
\renewcommand{\subsection}{
\renewcommand{\theequation}{\thesubsection.\arabic{equation}}
\oldsubsection}

\numberwithin{equation}{subsection}

\theoremstyle{plain}
\newtheorem{theorem}[equation]{Theorem}
\newtheorem{prop}[equation]{Proposition}
\newtheorem{cor}[equation]{Corollary}
\newtheorem{lemma}[equation]{Lemma}

\theoremstyle{definition}
\newtheorem{defi}[equation]{Definition}
\newtheorem{remark}[equation]{Remark}
\newtheorem{notation}[equation]{Notation}
\newtheorem{assumption}[equation]{Assumption}

\newtheorem*{question*}{Question}
\newtheorem{example}[equation]{Example}

\theoremstyle{remark}

\newcommand{\ol}{\overline}

\newcommand{\GLmod}[1]{\mathrm{GL}_2(\mathbb Z/#1\mathbb Z)}
\newcommand{\wt}{\widetilde}

\renewcommand{\P}{\mathbb P}
\newcommand{\C}{\mathbb C}

\newcommand{\Q}{\mathbb Q}
\newcommand{\Z}{\mathbb Z}
\newcommand{\N}{\mathbb N}
\newcommand{\F}{\mathbb F}
\newcommand{\G}{\mathbb G}
\newcommand{\HH}{\mathbb H}
\newcommand{\A}{\mathbb A}

\newcommand{\Legendre}[2]{\left( \frac{ #1 }{ #2 }\right)}

\newcommand{\pp}{\mathfrak{p}}

\newcommand{\Frob}{\operatorname{Frob}}
\newcommand{\Gm}{\G_{\mathrm{m}}}

\newcommand{\divisor}{\mathrm{div}}

\newcommand{\ns}{\mathrm{ns}}

\newcommand{\Div}{\operatorname{Div}}
\newcommand{\Eff}{\operatorname{Eff}}

\newcommand{\Lie}{\mathrm{Lie}}

\newcommand{\Norm}{\mathrm{Norm}}

\newcommand{\End}{\mathrm{End}}
\newcommand{\GL}{\mathrm{GL}}
\newcommand{\Pic}{\mathrm{Pic}}
\newcommand{\Gal}{\mathrm{Gal}}
\newcommand{\Aut}{\mathrm{Aut}}
\newcommand{\Spec}{\operatorname{Spec}}

\newcommand{\id}{\mathrm{id}}
\newcommand{\ord}{\mathrm{ord}}
\newcommand{\sm}{\mathrm{sm}}
\newcommand{\unr}{\mathrm{unr}}
\newcommand{\NS}{\mathrm{NS}}

\newcommand{\lcm}{\mathrm{lcm}}
\newcommand{\Tr}{\mathrm{Tr}}

\newcommand{\Id}{\mathrm{Id}}
\newcommand{\rank}{\operatorname{rank}}

\newcommand{\lto}{\longrightarrow}

\newcommand{\calL}{\mathcal{L}}
\newcommand{\calM}{\mathcal M}

\newcommand{\calO}{\mathcal{O}}
\newcommand{\calD}{\mathcal D}
\newcommand{\LMak}{\calL_{\operatorname{Mak}}}

\newcommand{\jbtilde}{\widetilde{j_b}}

\newcommand{\jac}{J}

\newcommand{\mum}{\mathcal{M}}
\newcommand{\Psaiguido}{\widetilde\Psi}
\newcommand{\schmorph}{h}

 \newcommand{\uniform}{w}

\newcommand{\Mx}{\mathcal{M}^\times}

\newcommand{\Lsymbol}{{\text{1}}}
\newcommand{\Rsymbol}{\text{2}}
\newcommand{\Ltimes}{\overset{\Lsymbol}{\otimes}}
\newcommand{\Rtimes}{\overset{\Rsymbol}{\otimes}}
\newcommand{\LTimes}{\bigotimes^\Lsymbol}
\newcommand{\RTimes}{\bigotimes^\Rsymbol}
\newcommand{\BLC}[3]{{}^{\top}#1 \, #2 \, #3}

\newcommand{\eqoverZ}{t^{\mathbb{Z}}_{\ol{D}_V + B, b}}

\newcommand{\NicMumfordSectGpLaws}{4}
\newcommand{\NicMumfordAlgCompare}{5}
\newcommand{\NicMumfordAlgtPQ}{21}
\newcommand{\NicMumfordMainFormula}{2.4}

\title{Equationless quadratic Chabauty for non-split Cartan modular curves} 
\author[1]{Sachi Hashimoto}
\author[2]{Guido Maria Lido}
\author[3]{Davide Lombardo}
\author[4]{Nicolas Mascot}
\author[5]{Pierre Parent}
\affil[1]{University of Colorado, Boulder}
\affil[2]{Università di Roma Tor Vergata}
\affil[3]{Università di Pisa}
\affil[4]{Trinity College Dublin}
\affil[5]{Universit\'e de Bordeaux}

\date{}
\begin{document}

\maketitle

\begin{flushright}
{\it Al contadino\footnote{Quadratic Chabauty} non far sapere \\ quanto \`e buono il formaggio\footnote{Makdisi's algorithms} con le pere\footnote{Modular regular models}...
\it }
\end{flushright}
\begin{abstract}
The aim of this article is to  describe an equationless method for determining the rational points on the non-split Cartan curve $X_{\ns}^+(N)$ of prime level $N \geq 13$. 
Instead of using a projective model for the modular curve, our method uses the moduli interpretation of the curve, namely we work directly with elliptic curves and Cartan level structures.
To accomplish this, we use the geometric version of the quadratic Chabauty method. 
We show that this can be combined with algorithms for divisor arithmetic developed by Makdisi and Mascot so as to apply to modular curves.
As an illustration, we rederive the set of rational points on the curve $X_{\ns}^+(13)$.
\end{abstract}
\tableofcontents

\section{Introduction} 
Let $N \geq 13$ be a prime number,  and let $X_\ns^+(N)$ be the modular curve  associated to the normalizer of a non-split Cartan subgroup of level $N$.  The aim of this article is to describe an \emph{equationless} method for the determination of the set $X_\ns^+(N)(\Q)$ of $\Q$-rational points on this curve, working directly with its moduli interpretation, without choosing an embedding of $X_{\ns}^+(N)$ in a projective space. 
As an illustration of our method, we will rederive the famous result from \cite{BDMTV}:
\begin{theorem}
The $\Q$-rational points of $X_{\ns}^+(13)$ are precisely the 7 known CM points. 
\end{theorem}

To help illustrate what it means to have an equationless method, we give the following example. 
Suppose one wants to find $X_{\ns}^+(13)(\F_p)$. 
Traditionally, with equations, one first represents the curve $X_{\ns}^+(13)$ with an equation, such as $ - x^{3} + x^{2} y^{2} + x y^{3} + 2 x y^{2}  + 2 x y  + 2 x  + y^{3}  - y =0 $, and then for each $x \in \F_p$,  solves for $y$.
In contrast, in an equationless approach, one instead first enumerates elliptic curves $E$ over $\F_p$ by iterating over $j \in \F_p$, and then enumerates Cartan level structures defined over $\F_p$ on these elliptic curves. 
In the first approach, the outcome is the data of points $(x,y)$ on the model, where in the second approach, by focusing on the moduli interpretation and without introducing an explicit model of the modular curve, one ends up with the data of pairs $(E, [\phi])$ of  $E/\F_p$  an elliptic curve and $\phi$ a Cartan structure on $E[13]$.

A motivation to study the rational points on the curves $X_\ns^+(N)$ is given by Serre's uniformity question \cite{serre1972proprietes}, which, after 
\cite{SerChebo, MazurIso, BiluParentCartan, BiluParentRebolledo},  
reduces to the following. 

\begin{question*}
Is there an integer $B$ such that, for all primes $N>B$, all the $\Q$-rational points on $X_\ns^+(N)$ are associated to CM elliptic curves?
\end{question*}

For all primes $N \ge 13$, the curve $X_\ns^+(N)$ has genus greater than $2$, and consequently, by Faltings's theorem, $X_\ns^+(N)(\Q)$ is finite. 
However, its determination is considered to be a challenge, even for a single value of $N$.

To put things in perspective, consider a curve $X/\Q$ of genus $g \geq 2$, let $J$ be its Jacobian, $r$ be the rank of $J(\Q)$, and $\rho$ be the rank of the N\'eron--Severi group of $J$, and suppose we know a point $b \in X(\Q)$. When applicable,  variants of the Chabauty method are commonly used to bound $X(\Q)$. These methods consist of embedding $X$ into $J$ via $x \mapsto [x-b]$, picking an auxiliary prime $p \in \N$, and attempting to construct obstructions to single out the $x \in X(\Q_p)$ for which $[x-b]$ lies in $J(\Q)$.

The classical Chabauty--Coleman method  \cite{Chab,ColemanDuke,McCallumPoonen}
applies when the $p$-adic closure of $J(\Q)$ has dimension less than $g$. This condition holds if $r < g$. 
However, for~$X = X_\ns^+(N)$ one knows that $r \ge g$, making it extremely likely that $J(\Q)$ is dense in an open of $J(\Q_p)$ and thereby precluding this method from working. 

Because of this limitation, Kim \cite{Kim1,Kim2} developed a non-Abelian version of the method. The quadratic\footnote{In \cite{dogra2025generalised}, Dogra renamed it to \emph{skimmed} quadratic Chabauty, so as to distinguish it from the \emph{full} quadratic Chabauty method in Kim's non-abelian program.} Chabauty method~\cite{QCIntegral, QC1, QC2} is based on Kim's program, and relaxes the condition $r<g$ to $r<g+\rho-1$. Dogra and Le Fourn \cite{DograLeFourninfamily} prove that this quadratic method 
is in principle applicable (possibly by working with a quotient of the Jacobian) to the family of curves $X_{\ns}^+(N)$.  It has indeed been successful in tackling several cases: \cite{BDMTV} settles the case $N=13$ as already mentioned, \cite{BDMTVmodular17} solves the case $N=17$, and the case $N=19$ has recently been announced. 

These computations used the \emph{cohomological} version of the quadratic Chabauty method, constructing obstructions from  $p$-adic heights. The current  implementation requires a plane model of $X$ such that, in each affine patch, the equation of $X$ is monic in one variable, has $p$-integral coefficients, and satisfies the technical assumptions \cite[Assumption~1]{Tuitman2}.

In contrast, in this article we use the \emph{geometric} version of the quadratic Chabauty method introduced in~\cite{Edil}.
In \cite{DRHS}, it is shown that the geometric and cohomological methods are essentially equivalent. See also recent work of Betts \cite{betts2025nonabeliananalogueconjecturemichael}
for an intriguing new proof of the equivalence in the rank $0$ case. 
In this geometric method, the obstruction that allows us to distinguish whether a divisor class $[x-b] \in J(\Q_p)$ comes from $J(\Q)$ 
is (related to) the liftability of well-chosen points in $J(\Z) \times J(\Z)$ to a certain $\mathbb{G}_m$ torsor, namely the Mumford torsor $\Mx$. More precisely, 
using the hypothesis that $\rho > 1$, for each $x\in X(\Q_p)$ one constructs a divisor class $[D_x]\in J(\Q_p)$ such that if $x \in X(\Q)$, then the pair $(x-b, D_x)\in J(\Z) \times J(\Z)$ can be lifted to a point in $\Mx(\Z)$, cf.~Figure~\ref{fig:artist_view} in Section~\ref{sec:review-geometric-quadratic-chabauty} below. 

The purpose of this work is thus to demonstrate that in the case of non-split Cartan modular curves, these divisor computations can be reduced to computations with modular forms and elliptic curves $E/\ol \Q$ or $E/\ol{\Q}_p$ equipped with a level structure $\phi$ on $E[N]$, thereby freeing us from the need for explicit models of $X_\ns^+(N)$. 
Indeed, such pairs $(E, [\phi])$ correspond to points on $X_\ns^+(N)$, and can thus be used to represent divisors. Furthermore, one can also represent divisors using subspaces of spaces of modular forms, using the techniques in \cite{Mak1,Mak2,NicMod}, which yields efficient algorithms to compute in $J$, and these algorithms were recently extended~\cite{NicMumford} to $\Mx$.

\subsubsection*{Relation to other recent work}

There has also been recent interest in equationless quadratic Chabauty methods in the non-geometric case. 
For example, Chen, Kedlaya, and Lau \cite{ChenKedlayaLau} gave an algorithm to compute Coleman integrals of holomorphic differentials on modular curves without equations.
In \cite{HashimotoGZ}, Hashimoto showed how to relate the global $p$-adic height to special values of $p$-adic $L$-functions using $p$-adic Gross--Zagier formulas, for certain quotients of $X_0(N)$.
On the other hand, in the case of $X_0^+(N)$, Rendell \cite{Rendell} considered the local $p$-adic height and gave an equationless algorithm to compute the Hodge filtration on a certain unipotent vector bundle, one of two key quantities needed for the local height.
See also Xu's recent PhD thesis~\cite{ChrisXu} where an equationless approach to the classical Chabauty--Coleman method for modular curves is proposed.

\subsubsection*{Structure of this article}

This article is structured as follows. We start by setting definitions and reviewing general notions that are used in the article in Section~\ref{sec:basicdef}. In Section \ref{sec:review-geometric-quadratic-chabauty} we recall how the geometric quadratic Chabauty method works on general curves. 
In Section \ref{sec: general Mumford} we still work on a general, possibly not modular curve, and we describe algorithms to compute in the Mumford torsor and give some formulas that describe both the $p$-adic closure $\ol{\Mx(\Z)}$ inside $\Mx(\Q_p)$ and the map $\jbtilde : X \longrightarrow \Mx$. We then go into greater detail,  
in order to apply these formulas in the case of non-split Cartan curves.  
In Section \ref{sec:some_modular_theory} we fix a regular model for $X_\ns^+(N)$ and study some properties that are used to deal with the bad reduction of the curve and to describe $\jbtilde$. In Section \ref{sect: model-free computations for Xnsp} we describe how we perform each step of the geometric quadratic Chabauty method on $X_\ns^+(N)$, in a modular, equation-free way.
Finally, in Section \ref{sec: summary} we give a summary of the algorithm, and in Section \ref{sec: example} we apply it to the case $N=13$ as an example.

\subsection*{Acknowledgments}

We thank Jonathan Love, \'Elie Studnia and Jan Vonk for 
sharing their work~\cite{LSV26} which is used at several places in this text. The proof of Proposition~\ref{prop:flatness}, in particular, comes from an idea they explained to us.  
We thank Eran Assaf and Edgar Costa for advice on coding and computations.
We thank Kamal Khuri-Makdisi for inspiring crucial parts of this work and  participating in many of our brainstorming sessions. We are also indebted to Bas Edixhoven for proposing many ideas that led to the present work.

Hashimoto is supported by the US National Science Foundation under Award No.\ DMS-2501658 and an AMS-Simons Travel Grant.
Lido is supported by the MIUR Excellence Department Projects MathMod@TOV awarded to the Department of Mathematics, University of Rome Tor Vergata,  and by the PRIN PNRR 2022 ``Mathematical Primitives for Post Quantum Digital Signatures''.
 Lombardo is supported by MUR grant PRIN-2022HPSNCR (funded by the European Union project Next Generation EU).  
 Lido and Lombardo are members of the INdAM group GNSAGA.
Parent is supported by AAP grants from Universit\'e de 
Bordeaux and 
ANR $j$-invariant.

\subsection*{AI Statement}
We did not use AI tools to aid in proving theorems nor to write this article.
The PARI/GP code to compute with modular Jacobians and with the Poincar\'e torsor was also composed without AI.
Several AI tools were used in the creation of the Magma code that produces representative points to generate $J(\Z)$ and produces deformations of points in each residue disc and acts by Hecke operators on these points.
We used ChatGPT-4o and ChatGPT-5 to help produce several auxiliary Magma functions in the code related to isogenies of elliptic curves and writing to files, and also used it to convert from the original codebase of Sage code to Magma code.
We used Cursor and Claude Opus 5 to refactor the (author written) Magma code, fix bugs in this code, and generate tests for correctness of the aforementioned code.
We (the authors) take full responsibility for the paper, the code and its output, and their mathematical correctness.

\subsection{Notation}

Here we collect notation that is used throughout the paper. This is only meant to help the reader navigate the paper more easily, so we do not always include full details; precise definitions will be given in the main body of the text.

\vspace{2mm}

Numbers.
\begin{itemize}[itemsep=2pt, topsep=1pt, parsep=1pt, partopsep=1pt]
\item $N$ the prime level of the non-split Cartan curve 
\item $p$ the Chabauty prime
\item $m$ the lcm of the exponents of $(J_{\F_\ell}/J^0_{\F_\ell})(\ol{\F}_\ell)$ as $\ell$ ranges over all (integer) primes
\item $g$ the genus of the curve $X$
\item $r$ the Mordell--Weil rank of $J_\Q(\Q)$
\end{itemize}

\vspace{2mm}
Curves, divisors, torsors, line bundles.
\begin{itemize}[itemsep=2pt, topsep=1pt, parsep=1pt, partopsep=1pt]
\item $X_\Q$ a smooth projective curve over $\Q$ of genus $\geq 2$ and $X/\Z$ a proper flat regular model with generic fiber $X_\Q$. Starting from Section \ref{sec:some_modular_theory}, we restrict to $X_\Q = X_\ns^+(N)$
\item $b \in X(\Q)$ a base point
\item $J_\Q$ the Jacobian of $X_\Q$, $J$ its N\'eron model over $\Z$, $J^0$ the connected component of the identity
\item $\mum$ the Mumford bundle of $J$ and $\Mx$ the Mumford torsor (see Section \ref{sec:mumford})
\item $U$ a simple open of $X$  (see Section \ref{subsubsec: simple opens})
\item $\Delta : S \to S\times_T S$ the diagonal of a scheme morphism $S \to T$
\item $\calL_y$  the fiber of a line bundle (or $\Gm$-torsor) $\calL$ over a scheme $Y$ at a point $y$ on $Y$

\item $Y(\Z_p)_y$ the residue disc of $y \in Y(\F_p)$ in $Y(\Z_p)$
\item $Y(\Z)_{y}$ the set of points in $Y(\Z)$ that reduce to $y$ modulo $p$ 
\end{itemize}

\vspace{2mm}

Maps.

\begin{itemize}[itemsep=2pt, topsep=1pt, parsep=1pt, partopsep=1pt]
\item $j_b$ the Abel--Jacobi embedding $X_\Q \to J_\Q$ by $b \in X(\Q)$
\item $h: J\to J^0$ 
a scheme morphism such that $j_b^*(\id_J, h)^* \mum$ is trivial on $X$ of the form  $h \coloneq m\cdot(f +c)$ for a choice of a trace 0 self-dual endomorphism $f$ of $J$
\item $\jbtilde : U \to \Mx$ a lift of $(\id_J, h)\circ j_b \colon U \to J \times J^0$
\item $\kappa: \Z_p^{2g} \to \Mx(\Z_p)_{\ol{t}}$ the parameterization of the closure of the $\Z$-points in the residue disc of some  $\ol t\in \Mx(\F_p)$
\end{itemize}

\vspace{2mm}
Fields, rings.
\begin{itemize}[itemsep=2pt, topsep=1pt, parsep=1pt, partopsep=1pt]
\item $\Q_q$ an unramified extension of $\Q_p$ with ring of integers $\Z_q$ and residue field~$\F_q$
\end{itemize}

\section{Basic definitions}\label{sec:basicdef}

Let $X_{\mathbb{Q}}$ be a smooth projective curve of genus at least $2$ and let $X / \operatorname{Spec} \Z$ be a proper flat regular model with generic fiber $X_{\mathbb{Q}}$. We fix a rational point $b \in X_{\Q}(\Q)$; by properness it extends uniquely to a section $b \in X(\Z)$.
Let $J_\Q$ be the Jacobian of $X_\Q$ and let $J$ be the N\'eron model of $J_\Q$ over the integers.
We write $J^0$ for the fiberwise connected component of $J$ containing $0$. Define $m \in \Z$ to be the least common multiple of the exponents of the geometric component groups of $J$, that is,
\begin{equation}\label{eq:m}
    m  \coloneq \underset{\ell}{\lcm} \left( \text{exponent of the component group } (J_{\F_\ell}/J^0_{\F_\ell})(\overline{\F}_\ell)\right),
\end{equation}
where $\ell$ varies over all rational primes, or equivalently over all primes $\ell$ of bad reduction for $J$.

\subsection{Parameters and integral analytic coordinates}\label{sec:parameters}

For any scheme (resp.~$\Z_p$-scheme) $Y$ and any $\ol y \in Y(\F_p)$ we define the \emph{residue disc above $\ol y$}  as
\[
Y(\Z)_{\ol y} \coloneq \{ P \in Y(\Z) : P \equiv \ol y \bmod p\}, \quad \text{resp.~} Y(\Z_p)_{\ol y} \coloneq \{ P \in Y(\Z_p) : P \equiv \ol y \bmod p\}.
\]

Let $Y/{\Z_p}$ be a scheme of finite type, flat over $\Z_p$ and of relative dimension $d$. Fix a smooth point $\ol y \in Y(\F_p)$ (which will often arise as the reduction of a $\Z_p$-point $y$). Since $\ol y$ is a smooth point, the maximal ideal of its local ring  $\calO_{Y, \ol y}$ is generated by $p$, together with $d$ elements $t_1, \dots, t_d$. We will call $p, t_1, \ldots, t_d$ a \emph{system of local parameters} or \emph{algebraic parameters} at $\ol y$. Note that we always assume that $p$ is part of our system of local parameters, so we will also say that $t_1, \ldots, t_d$ are algebraic parameters, omitting explicit mention of $p$.
The function $\tilde{t}\coloneq (\tilde{t}_1, \ldots, \tilde{t}_d) \coloneq \left( \frac{t_1}{p}, \ldots, \frac{t_d}{p} \right)$ 
gives a bijection 
\begin{equation}    
	\label{eq:alg_param}
	\tilde t \colon Y(\Z_p)_{\ol y} \lto (\Z_p)^d.
\end{equation}
In particular, when $Y$ is a curve and $\ol y \in Y(\F_p)$ is a smooth point, by \emph{local parameter at $\ol y$} we will mean a function $f$ such that $(p, f)$ generates the maximal ideal of $\mathcal{O}_{Y, \ol y}$.
	
	We will call \emph{integral analytic function at $\ol y$} any integral convergent power series in the $\tilde t_i$, i.e., an element of the ring
	\begin{equation*}
		\Z_p \langle \tilde{t}_1, \ldots, \tilde{t}_d \rangle  \quad \subset (\widehat \calO_{X,\ol x}) \otimes_{\Z_p} \Q_p \subset \Q_p[\![t_1, \ldots, t_d ]\!] \ .
	\end{equation*}
The ring $\Z_p \langle \tilde{t}_1, \ldots, \tilde{t}_d \rangle$ is independent of the choice of the $t_i$, as proven in \cite[Section 3]{Edil}. Each integral analytic function at $\ol y$ defines a continuous function $ Y(\Z_p)_{\ol y} \lto \Z_p$.

	Elements $\gamma_1, \ldots, \gamma_d\in \Z_p \langle \tilde{t}_1, \ldots, \tilde{t}_d \rangle$ will be called (a system of) \emph{analytic coordinates at $\ol y$} if the natural inclusion 
	\[
	\Z_p \langle \gamma_1, \ldots, \gamma_d \rangle  \to \Z_p \langle \tilde{t}_1, \ldots, \tilde{t}_d \rangle
	\]
	is bijective.
When this holds, the map $(\gamma_1, \ldots, \gamma_d) : Y(\Z_p)_{\ol y} \to \Z_p^d$ is a bijection. By a slight abuse of language, we will also call this map a system of analytic coordinates.

\begin{remark}
	Let $\gamma_1, \ldots, \gamma_d$ be elements of $\Z_p \langle \tilde{t}_1, \ldots, \tilde{t}_d \rangle$
	whose images in  
	$\mathbb{F}_p[\tilde{t}_1, \ldots, \tilde{t}_d]$ are polynomials of degree 1. 
	Then $\gamma_1, \ldots, \gamma_d$ are analytic coordinates at $\ol y$ if and only if the linear forms $\gamma_1 -\gamma_1(0) \pmod p, \ldots, \gamma_d  -\gamma_d(0)\pmod p$ are $\F_p$-linearly independent. 
\end{remark}

We recall the following property of systems of analytic coordinates induced by algebraic parameters (see also \cite[\S3]{Edil}).
\begin{lemma}\label{lem:algebraic is linear}
	Let $Y, Z$ be schemes of finite type, flat over $\Z_p$ and of relative dimensions $d_Y, d_Z$. Let $f: Y \to Z$ be a map of schemes. Let $\ol y \in Y(\F_p)$ be a smooth point and let $\ol z = f(\ol y)$. Assume that $\ol z$ is also a smooth point and let  $ t_1, \ldots, t_{d_Y}$ be algebraic parameters at $\ol y$ and $ u_1, \ldots, u_{d_Z}$ be algebraic parameters at $\ol z$, giving bijections $\tilde{t} :  Y(\Z_p)_{\ol y} \to \Z_p^{d_Y}$ and $\tilde{u} \colon Z(\Z_p)_{\ol z} \to \Z_p^{d_Z}$. The map
	\[
	\tilde{u} \circ f \circ \tilde{t}^{-1} : \Z_p^{d_Y} \to \Z_p^{d_Z}
	\]
	is given by integral convergent power series that reduce modulo $p$ to polynomials of degree at most $1$.
\end{lemma}
\begin{proof}
	The map of local rings $\calO_{Z, \ol z} \to \calO_{Y,\ol y}$ induces a map between the completions
	\[
	\Z_p[\![u_1, \ldots  , u_{d_Z}]\!] =  \widehat  \calO_{Z, \ol z} \to \widehat \calO_{Y,\ol y} = \Z_p[\![t_1, \ldots  , t_{d_Y}]\!]\ .
	\]
	In particular, there are power series $F_i \in \Z_p[\![t_1, \ldots  , t_{d_Y}]\!]$ such that, for each point $P\in Y(\Z_p)_{\ol y}$, we have   $u_i(f( P)) = F_i(t_1(P), \ldots  , t_{d_Y}(P))$.
	
	Then, the map $\tilde{u} \circ f \circ \tilde{t}^{-1} : \Z_p^{d_Y} \to \Z_p^{d_Z}
	$ is given by the power series $\tilde F_i \coloneq \frac 1p  F_i(p \tilde t_1, \ldots  , p \tilde t_{d_Y})$, hence the claim follows by direct computation (notice that $F_i(0)$, which is equal to $u_i(f(P_0))$ for a suitable $P_0 \in Y(\Z_p)_{\ol y}$, is a multiple of $p$ since $f(P_0)$ reduces to $\ol z$).
\end{proof}

A nice property of analytic coordinates is that they detect congruences modulo powers of~$p$.

\begin{lemma}\label{lem: coords mod pe}
	Let $Y$ be a smooth scheme over $\Z_p$, let $\ol y$ be a point in $Y(\F_p)$, and $f$ be an integral analytic function at $\ol y$. Then, for $e$ in $\Z_{\ge1}$ and for all points $y_1, y_2 \in Y(\Z_p)_{\ol y}$ such that $y_1 \equiv y_2 \pmod{p^e}$, we have $f(y_1) \equiv f(y_2) \pmod{p^{e-1}}$. 
	
	Moreover, if $F = (f_1, \ldots, f_d) : Y(\Z_p)_{\ol y} \to \Z_p^d$ are analytic coordinates at $\ol y$, the composition of $F$ with the reduction map $\Z_p^d \to (\Z/p^{e-1}\Z)^d$ gives a bijection 
	\[
	F \colon Y(\Z/p^e\Z)_{\ol y}\to (\Z/p^{e-1}\Z)^d .
	\]    
\end{lemma}
\begin{proof}
	In the notation of  \eqref{eq:alg_param}, the first part of the statement is true for $f = \tilde t_i = \frac{t_i}{p}$, since the function $t_i$ is algebraic, hence it descends to a map $t_i \colon Y(\Z/p^e\Z)_{\ol y} \to \Z/p^e\Z$. Then, it also holds for any integral analytic $f$, since we can write  $f$ as the sum of a polynomial in the $\tilde t_i$ and a convergent power series in the $\tilde t_i$  with coefficients divisible by $p^{e}$. 
	
	Now the second part. The first part implies that the composition we look at descends to a function  $Y(\Z/p^e\Z)_{\ol y}\to (\Z/p^{e-1}\Z)^d $. Moreover, such a composition is surjective (both $F$ and the reduction are surjective), hence bijective, since the source and the target are finite with the same cardinality (by smoothness of $Y$ and Hensel's lemma).
\end{proof}

Because of Lemma \ref{lem: coords mod pe}, when $y$ is a point in $Y(\Z/p^e\Z)$ and $F$ is a system of analytic coordinates at $\ol y$, we will be able to consider $F(y) \in (\Z/p^{e-1}\Z)^d$.

\subsection{Simple opens}\label{subsubsec: simple opens}

We say $U \subset X$ is a \emph{simple open} if it is an open subset such that in every fiber $X_{\F_\ell}$ for $\ell$ prime, $U_{\F_\ell}$ is an irreducible component of the smooth part $X^\sm$ of $X$, and additionally $U_{\F_\ell}(\F_\ell) \neq \emptyset$.
In particular, there are at most finitely many simple opens of $X$, and every $z \in X(\Z)$ is contained in a unique simple open.

When $U$ is a simple open, the map $\Pic(U) \to \Pic(X_\Q)$ is an isomorphism. In particular, given a line bundle $\calL$ over $X^\sm$, trivial over $X_\Q$, and a trivializing section $s: X_\Q \to \calL_\Q$, there exists $\lambda\in \Q^\times$, unique up to sign, such that $\lambda \cdot s$ extends to a trivializing section $U \to \calL$.

\subsection{Modular curves and the non-split Cartan subgroup}\label{subsect: modular curves}

Let $N$ be a positive integer and $H\leq \GLmod{N}$. Let $k$ be a field in which $N$ is invertible.
We give a brief overview of the definition of the modular curve $X_H$. For more details see e.g.~\cite{DiamondIm}, \cite[Section 2]{RSZB} and the references therein.
An \emph{elliptic curve $E/\overline{k}$ with $H$-level structure} is a pair $(E, [\phi])$, where $E/\overline{k}$ is an elliptic curve and $[\phi]$ is an equivalence class of isomorphisms
\begin{equation}\label{eq:level_structure_general}
\phi \colon E[N] \to (\Z/N\Z)^2,  
\end{equation}
where $\phi \sim \phi' $ if $\phi = h \circ \phi'$ for some $ h \in H$. 
Two pairs $(E,[\phi])$ and $(E',[\phi'])$ are equivalent if there exists an isomorphism $\alpha\colon E\to E'$ such that $[\phi]=[\phi'\circ\alpha_{|E[N]}]$.
The points of $Y_H(\overline{k})$ parameterize the equivalence
classes of such pairs.

There is a Galois action on the pairs $(E, [\phi])$. An automorphism $\sigma\in\Gal(\bar k/k)$ sends $E$ to $E^\sigma$ and the point $P \in E^{\sigma}[N]$ to $\sigma^{-1}(P) \in E[N]$. This induces an action on the points $Y_H(\overline{k})$: $\sigma$ sends $(E,[\phi])$ to $(E^\sigma,[\phi \circ \sigma^{-1}])$. The $k$-rational points of $Y_H$ are the fixed points under this action. In particular, each $k$-rational point of $ Y_H$ is represented by a pair $(E,[\phi])$ where $E/k$ is an elliptic curve over $k$ with $H$-level structure $\phi: E[N] \to (\Z/N\Z)^2$, and for all $\sigma \in \Gal(\bar{k}/k)$, there exist $\alpha \in  \Aut(E_{\bar{k}})$ and $h \in H$ such that 
\begin{equation} 
\label{eq:rationalptsYH}
\phi \circ \sigma^{-1} = h \circ \phi \circ \alpha_{|_{E[N]}}.
\end{equation}
By adjoining finitely many cusps, we obtain a compact modular curve $X_H$. 
The curve $X_H$ is defined over $\Z[1/N]$ and can be extended to $\Z$ in various ways, see e.g.~\cite{KatzMazur}. In this paper, we choose a regular model with normal crossings (see Section \ref{sec:some_modular_theory}).

Let now $N>2$ be a prime. A \emph{non-split Cartan subgroup} of $\GLmod{N} = \GL_2(\F_N)$ is the set of $\F_N$-rational points of a non-split maximal torus of $\GL_{2,\F_N}$. All such subgroups of $\GL_2(\F_N)$ are conjugate, and they are isomorphic to $\F_{N^2}^\times$. Less canonically, but more explicitly, we can describe a non-split Cartan subgroup as the image of an embedding $\F_{N^2}^\times \hookrightarrow \GL_2(\F_N)$ obtained 
by fixing an $\F_N$-basis $B$ of $\F_{N^2}$ and expressing in this basis the action of $\F_{N^2}^\times$ on $\F_{N^2}$ given by multiplication. For example, if we fix  a non-square $\varepsilon\in \F_N$ and we consider the basis $B = \{1,\sqrt \varepsilon\}$, we obtain the subgroup 
\[ 
C_{\rm ns }^{\varepsilon}(N)  = \left\{ \begin{pmatrix}a & b \\ c \varepsilon & d \end{pmatrix} \in \GL_2(\F_N) : a = d, c = b    \right\}. 
\]
(Note that we consider the elements of $\F_N^2$ as row vectors, and $C_{\rm ns }^{\varepsilon}(N)$ acts by multiplication on the right.)
We are interested in the normalizers of non-split Cartan subgroups of $\GL_2(\F_N)$, which, up to conjugacy, are isomorphic to 
\[ 
C_{\rm ns }^{\varepsilon,+}(N)  = C_{\rm ns }^{+}(N)  = \left\{ \begin{pmatrix}a & b \\ c \varepsilon & d \end{pmatrix} \in \GL_2(\F_N) : a = d, c = b  \text{ or } a = - d,  c = - b  \right\} . 
\]
In particular, after fixing $\varepsilon$, the modular curve $X_\ns^{\varepsilon, +}(N) = X_\ns^+(N)$ is the modular curve $X_H$ corresponding to $H = C_{\rm ns }^{\varepsilon,+}(N)$.

One way to describe non-cuspidal points is to follow \eqref{eq:level_structure_general}. We shall also use alternative descriptions of $C_{\rm ns }^{+}(N)$-structures on an elliptic curve $E$, either as subgroups of $\GL(E[N])$ or as unordered pairs of lines in $E[N]\otimes_{\F_N}\F_{N^2}$, for example in Proposition \ref{GaloisOrbitsSS} and its proof. 

\subsection{Line bundles and \texorpdfstring{$\Gm$}{Gm}-torsors}
\label{sec:basic_gm_torsors}
We will freely use the following fact: giving a line bundle $\mathcal{L}$ on a scheme $X$ is equivalent to giving a $\Gm$-torsor $\mathcal{L}^\times \to X$. Here $\mathcal{L}^\times$ denotes the complement of the zero section in the total space of $\mathcal{L}$, and the natural fiberwise action of $\Gm$ makes $\mathcal{L}^\times$ into a $\Gm$-torsor over $X$. Conversely, any $\Gm$-torsor over $X$ arises in this way from a unique line bundle (see e.g.~\cite[Section 11.6]{GortzWedhorn}). 
We will be particularly interested in the $\Gm$-torsor $\Mx$ obtained from the Mumford bundle $\mum$ by this construction, see the next section for details.

\subsection{The Mumford bundle}
\label{sec:mumford}

Let $j_b : X_\Q \to J_\Q$ be the Abel--Jacobi map with respect to the base point $b$. Then $j_b^*\colon J^\vee_\Q \to J_\Q$ is an isomorphism, and indeed it is the inverse of $-\lambda$, where $\lambda : J_\Q \to J_\Q^\vee$ is the canonical principal polarization of $J_\Q$ (see \cite[Lemma 6.9]{JacobianMilne}).

The \emph{Mumford bundle} is the line bundle $\mum$ on $J_\Q \times J_\Q$ characterized by the property that for each $x \in J_\Q$ the restriction of $\mum$ to $J_\Q \times \{x\}$ and to $\{x\} \times J_\Q$ represents the line bundle on $J_\Q$ given by $(j_b^*)^{-1}(x)$. 
The Mumford bundle comes equipped with trivializations over $J_\Q \times \{0\}$ and $\{0\} \times J_\Q$, that is, isomorphisms between the restrictions of $\mum$ to these subvarieties and the corresponding trivial line bundles. These trivializations are compatible above the point $(0, 0)$.
Note that $\mum$ is the pullback of the Poincaré bundle $\mathcal{P} \to J_\Q \times J_\Q^\vee$ under the isomorphism $J_\Q \times J_\Q \xrightarrow{\operatorname{id} \times (-\lambda)} J_\Q \times J_\Q^{\vee}$, because $-\lambda = (j_b^*)^{-1}$. 

We  briefly recall how to extend the construction of $\mum$ over $\Spec(\Z)$. 
Let $\mum_{\Q}^{\times}$ be the $\G_m$-torsor associated with $\mum$ over $J_{\Q}\times J_{\Q}$. It extends uniquely, compatibly with its biextension structure, to a $\G_m$-torsor $\Mx\to J\times J^0$. This is analogous to the extension of the Poincar\'e torsor over $J\times J^{\vee,0}$; see \cite[Expos\'e VIII, Theorem 7.1]{SGA7} and \cite[Section 2]{Edil}.
In particular, the integral points of $\Mx$ lying over $(x,y) \in (J \times J^0)(\Z)$ form a $\G_m(\Z)$-torsor, so they are unique up to sign. 

By the correspondence between $\Gm$-torsors and line bundles, we obtain that $\mum$ also extends to a line bundle over $J \times J^0$. The fibers of this extended bundle have the following explicit description that makes them amenable to computation. Let $S$ be a scheme such that $X_S$ is smooth, and let $(x,y) \in (J \times J^0)(S)$ be a pair of points on the Jacobian represented by the degree 0 Cartier divisors $(D, E)$. We can write a divisor as a difference of (relatively Cartier) effective divisors $D = D^+ - D^-$. 
Using the canonical isomorphism in \cite[Remark 6.3.12]{Edil} and the symmetry in the two factors \cite[Proposition 6.3.2 and Lemma 6.5.4]{Edil},
the fiber $(x,y)^*\mum $  is 
\begin{equation} \label{eq:Mumford_fiber}    
D^* \calO_X(E)\coloneq \mathrm{Norm}_{D^+/S}(\calO_X(E)) \otimes\mathrm{Norm}_{D^-/S}(\calO_X(E))^{-1}. 
\end{equation}
Similarly, if $x$ is represented by a degree $0$ divisor $D$ and $y$ by a line bundle $\calL$ on $X$ of multidegree $0$, we have 
\begin{equation} \label{eq:Mumford_fiber_bundles}    
(x,y)^*\mum  = D^* \calL\coloneq \mathrm{Norm}_{D^+/S}(\calL) \otimes\mathrm{Norm}_{D^-/S}(\calL)^{-1}. 
\end{equation}
Removing the zero section, we obtain a similar description for $(x,y)^* \Mx$.
We also denote the fiber simply by  $\Mx_{(x,y)}$ or $\Mx_{x,y}$.
Furthermore, the $\G_m$-biextension structure of $\Mx$ can be described as follows: the two partial group laws  $\Ltimes \colon \Mx_{(x_1,y)} \times \Mx_{(x_2,y)} \to \Mx_{(x_1+x_2,y)}$ and $\Rtimes : \Mx_{(x,y_1)} \times \Mx_{(x,y_2)} \to \Mx_{(x, y_1+y_2)}$ are induced by the canonical isomorphisms of line bundles
\begin{align*}& \Ltimes : D_1^* \calO_X(E) \otimes D_2^* 
\calO_X(E) \to (D_1+D_2)^* \calO_X(E),\\
& \Rtimes : D^* \calO_X(E_1) \otimes D^* \calO_X(E_2) \to D^* \calO_X(E_1+E_2)\end{align*}
where $E, E_1, E_2, D, D_1, D_2$ are degree 0 divisors on $X$.
For more details, see also \cite[Section 6.3]{Edil} or \cite[Section 3.2]{DRHS}.

\subsection{Bilinear combinations}\label{sec: bilinear combination}
The following notation will be useful to simplify some of our calculations in $\Mx$.
Given $a,b \in \N$, $v \in \Z^a$, $w \in \Z^b$, a ring $R$, points $x_1,\cdots,x_a \in J(R)$ and $y_1,\cdots,y_b \in   J^0(R)$, and an $a \times b$ matrix $\Pi=(\pi_{ij})$ with entries in $\Mx(R)$ such that $\pi_{i,j} \in \Mx_{x_i,y_j}(R)$ for all $i,j$, we define
\begin{equation} \label{eq: blc}
{}^{\top} v \, \Pi \, w = \LTimes_{i \le a} \left( \RTimes_{j \le b} \pi_{i,j}^{\Rtimes w_j} \right)^{\Ltimes v_i} = \RTimes_{j \le b} \left( \LTimes_{i \le a} \pi_{i,j}^{\Ltimes v_i} \right)^{\Rtimes w_j} \in \Mx_{\sum_i v_i x_i, \sum_j w_j y_j}(R).
\end{equation}
The equality of the two expressions follows from the compatibility of the two partial group laws in the biextension structure. 
Analogously, given a matrix $M \in M_{c \times a}(\Z)$ (respectively  $M \in M_{b \times d}(\Z)$), we define  a matrix of size $c \times b$  (resp. $a \times d$)  with entries in $\Mx(R)$ by 
\[
(M \,\Pi)_{ij} \coloneq \LTimes_{k\le a} \pi_{kj}^{\Ltimes M_{ik}}, 
\quad \text{resp. }
(\Pi \, M)_{ij} = \RTimes_{k\le b} \pi_{i, k}^{\Rtimes M_{k, j}} .
\]

\begin{remark}
This choice of notation arises from thinking additively of the group laws, i.e., if one writes $\Ltimes$ as $+_1$ these equations become
\[
(M \, \Pi)_{ij} = \sum_{k\le a}^1 M_{ik} \pi_{kj},
\quad
{}^{\top} v \Pi w = \sum^{1}_{i \leq a} v_i \cdot_1 \sum_{j \leq b}^2 w_j \cdot_2 \pi_{i, j}.
\]
\end{remark}
One checks easily that
\[ {}^{\top}(v+v') \Pi w = {}^{\top}v \Pi w \Ltimes {}^{\top} (v') \Pi w \text{ and } {}^{\top}v \Pi (w+w') = {}^{\top}v \Pi w \Rtimes {}^{\top} v \Pi w'. \]
Moreover, whenever $v, w$ are integer column vectors and $M$ is an integer matrix of compatible sizes, we have
\[
{}^{\top} (Mv) \Pi w = {}^{\top} v ({}^{\top}M\Pi) w \text{ and } {}^{\top}v \Pi (Mw) = {}^{\top}v (\Pi M) w.
\]
\subsection{Choice of \texorpdfstring{$p$}{p}-adic logarithm}\label{ref: p-adic log}

Let $p$ be an odd prime.
We denote by $\log\colon \Z_p^\times \to p\Z_p$ the branch of the $p$-adic logarithm which on $1+p\Z_p$ is given by the usual power series and sends roots of unity to $0$. Since $p>2$, it induces the analytic coordinate $\wt\log \coloneq \tfrac 1p \log \colon \Gm(\Z_p)_{\ol 1} \coloneq 1+p\Z_p \to \Z_p$. The map $\wt\log$ is an isomorphism of groups, with inverse $\wt \exp (z) \coloneq \exp(pz)$. 

\section{Overview of geometric quadratic Chabauty}\label{sec:review-geometric-quadratic-chabauty}

\begin{figure}[H]
    \centering

\begin{tikzpicture}[scale=.8]
\draw[red,thick,variable=\t,domain=-1:1,samples=50]
plot ({-3+1.5*\t},{(\t*\t*\t-\t)/5-2});
\draw(-3,-2.5) node{$\color{red} X(\Q_p)$};

\draw[thick,->] (-1,-2) -- (1,-2);
\draw(0,-2.5) node{$j_b$};

\draw[thick,->] (-1,-1.5) -- (0.5,1);
\draw(-0.5,0.7) node{$\wt{j_b}$};

\draw[thick,->] (3.5,-0.5) -- (3.5,-1);
\draw(-0.5,0.7) node{$\wt{j_b}$};

\draw[green!60!black,fill=green!10] (1,-3) -- (5,-3) -- (6,-1.5) -- (2,-1.5) -- (1,-3);
\draw[red,thick,variable=\t,domain=-1:1,samples=50]
plot ({3.5+1.5*\t},{(\t*\t*\t-\t)/5-2.2});
\draw(7,-2) node{\color{green!60!black} $J(\Q_p)$};
\draw(7,-2.7) node{\color{green!60!black} $J(\Z)$ dense};

\draw[blue] (1,0) -- (5,0) -- (6,1.5) -- (2,1.5) -- (1,0); 
\draw[blue,fill=blue!30] (1,2) -- (5,2) -- (6,3.5) -- (2,3.5) -- (1,2); 
\draw[blue] (1,4) -- (5,4) -- (6,5.5) -- (2,5.5) -- (1,4); 
\draw[blue] (1,0) -- (1,4);
\draw[blue] (5,0) -- (5,4);
\draw[blue] (6,1.5) -- (6,5.5);
\draw[blue] (2,1.5) -- (2,5.5);
\draw[red,thick,variable=\t,domain=-1:-0.856452866183,samples=50]
plot ({3.4+1.5*\t},{-2*\t*\t*\t*\t+(\t*\t*\t-\t)/3+3});
\draw[red,dashed,thick,variable=\t,domain=-0.856452866183:-0.5,samples=50]
plot ({3.4+1.5*\t},{-2*\t*\t*\t*\t+(\t*\t*\t-\t)/3+3});
\draw[red,thick,variable=\t,domain=-0.5:0.7,samples=50]
plot ({3.4+1.5*\t},{-2*\t*\t*\t*\t+(\t*\t*\t-\t)/3+3});
\draw[red,dashed,thick,variable=\t,domain=0.7:0.821521577185686,samples=50]
plot ({3.4+1.5*\t},{-2*\t*\t*\t*\t+(\t*\t*\t-\t)/3+3});
\draw[red,thick,variable=\t,domain=0.821521577185686:1,samples=50]
plot ({3.4+1.5*\t},{-2*\t*\t*\t*\t+(\t*\t*\t-\t)/3+3});
\draw [thick,decorate,decoration = {calligraphic brace,amplitude=5pt},pen colour=blue] (6.5,5.5) --  (6.5,0);
\fill[red] (2.65,3) circle (2.5pt);
\fill[red] (4.45,2.4008) circle (2.5pt);
\draw(7.8,2.75) node{\color{blue}$\Mx(\Q_p)$};
\draw(-1,4) node{\color{blue} $\ol{\Mx(\Z)}$};
\draw[gray] (0,4) -- (1.5,3); 
\end{tikzpicture}   

    \centering
\caption{Artistic impression of the method in a residue disc}
\label{fig:artist_view}
\end{figure}
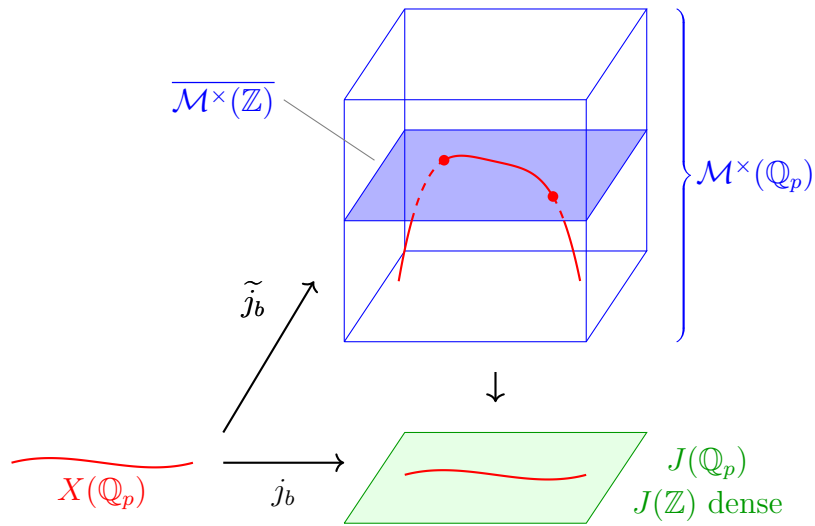

Let $X_{\Q}$ be a smooth projective geometrically connected curve of genus $g \geq 2$, $\jac_\Q$ its Jacobian, and let $r=\rank J_\Q(\Q)$. Let $J$ be the Néron model of $J_\Q$ over $\Z$. We briefly review the geometric quadratic Chabauty strategy to determine the set $X_\Q(\Q)$. 
When $r<g$, the classical Chabauty--Coleman method can be used to bound $X_\Q(\Q)$,
so we focus on the first case beyond the range of the classical method, namely $r=g$ and $\rho \coloneq \rank_{\Z}\NS(J_{\Q})>1$ (cf.\ Remark \ref{remark: GQC works} below for the case where $r>g$).

Let $X$ be a regular scheme, proper and flat over $\Z$, with generic fiber isomorphic to $X_{\Q}$.  We fix throughout an odd prime $p \geq 3$ of good reduction for $X$.

We assume that we know a point $b\in X(\Z)$ (see also Remark \ref{rmk:no_base_point}), and aim to find all points in $X(\Z)$, or at least a finite subset of $X(\Z_p)$ containing $X(\Z)$. 

We first split the problem into sub-problems, to simplify the geometry. Let $X^\sm$ be the smooth locus of $X/\Z$. Let $U_1, \ldots, U_{k}$ be the 
simple opens of $X^\sm$ (cf.\ Section \ref{subsubsec: simple opens}).
We have
\[
X_\Q(\Q) = X^\sm(\Z) = \sqcup_{i} U_i(\Z)\,.
\]
Hence, it is enough to compute a finite subset of $U(\Z_p)$ containing $U(\Z)$ for each $U=U_i$.

Let $J$ be the N\'eron model over $\Z$ of the Jacobian of $X_\Q$, let $J^0$ be the fiberwise connected component of the identity of  $J$ and let $j_b: X_\Q \to J_\Q$ be the Abel--Jacobi map sending $x \mapsto \calO_{X_{\Q}}(x - b)$. Let $\Mx \to J\times J^0$ be the Mumford torsor, cf.~Section \ref{sec:mumford}.

   To apply the method, we build a scheme morphism $\schmorph \colon J\to J^0$ 
   such that the bottom  map of
\begin{equation}    \label{eq:does jbtilde exist}
\begin{tikzcd}[column sep=small]
	& & & \Mx \arrow[d] \\
	U \arrow[rrru,"\wt{j_{b}} \ ?",dashed ] \arrow[r,"j_b"']
	& J \arrow{rr}[swap]{(\id_J, \, \schmorph)} & {} & J\times J^{0}
\end{tikzcd}
\end{equation}
can be lifted to a map $\wt{j_{b}}$.
Such a lift exists if and only if 
$j_b^*(\id_J, \schmorph)^* \Mx$ is a trivial torsor on $X_\Q$, or equivalently on $U$, as explained in Section \ref{subsubsec: simple opens}. 
Note that we do not require $\schmorph$ to be a homomorphism of group schemes. In fact, we will choose $\schmorph$ to be of the form
\begin{align}
     \label{def:h}
    \schmorph :  J  \to & J^0  \\
      x  \mapsto& m \cdot (f(x) + c),\notag
\end{align}
where $f : J \to J$ is a group homomorphism, $c$ is a fixed point in $J(\Z)$, and $m$ is defined as in \eqref{eq:m}, so that multiplication by $m$ kills the component group of (each fiber of) $J$. We further require that the endomorphism $f_\Q$ has trace $0$. 
Conversely, for each such $f$, there exists a choice of $c \in J(\Z)$ such that $(\id_J, \schmorph)^* \Mx$ is trivial on $X_\Q$. In practice, we will also take $f$ to be equal to its dual $f^\vee$ up to the canonical principal polarization  $\lambda: J \simeq J^\vee$.

Note that $U(\Z)$ is the disjoint union of the sets $U(\Z)_{\ol u}$ for $\ol u\in U(\F_p)$. We work residue disc by residue disc.
Since $U(\Z)_{\ol u} \subset U(\Z_p)_{\ol u}$ and $\jbtilde$ is injective (indeed $j_b$ is injective), we can identify $U(\Z)_{\ol u}$ with the subset $\jbtilde(U(\Z)_{\ol u})$ inside 
\begin{equation}
\wt{j_{b}}(U(\Z_p)_{\ol u})\cap \ol{\Mx(\Z)}_{\wt{j_b}(\ol u)}
\label{eq:THE_intersection}
\end{equation}
where $ \ol{ \Mx(\Z)}_{\wt{j_b}(\ol u)}$ is the closure of $\Mx(\Z)_{\wt{j_b}(\ol u)}$ inside $\Mx(\Z_p)_{\wt{j_b}(\ol u)}$ for the $p$-adic topology. 
 
The main goal of the geometric quadratic Chabauty method is to compute the  intersection \eqref{eq:THE_intersection},
which under our hypotheses is a finite set of $p$-adic points containing (the image of) $U(\Z)_{\ol u}$, see Remark \ref{remark: GQC works}.

If $\Mx(\Z)_{\wt{j_b}(\ol u)}$ is empty, we deduce that $U(\Z)_{\ol u}$ is also empty (see also Remark \ref{rmk kappa}). Otherwise, by the argument in \cite[Theorem 4.10]{Edil} (see also Section \ref{sec:kappa_general}), using that the rank $r$ of the Mordell--Weil group of $J$ is equal to $g$, one can identify $\ol{\Mx(\Z)}_{\wt{j_b}(\ol u)}$ with the image of an analytic map
\begin{equation} \label{eq:kappa_first_appearance}
\kappa \colon \Z_p^{2r} = \Z_p^{2g}  \to \Mx(\Z_p)_{\wt{j_b}(\ol u)},    
\end{equation}
which can be constructed explicitly from the biextension structure of $\Mx$.

Thus, in essence, the geometric quadratic Chabauty method consists in computing, for each simple open $U$ and each $\ol{u}\in U(\F_p)$, the intersection 
\begin{equation}\label{eq:inters}
    \kappa(\Z_p^{2g}) \cap \wt{j_b}(U(\Z_p)_{\ol u}).
\end{equation}
By choosing suitable coordinate systems around the point $\wt{j_b}(\ol u)$, we make this more explicit and show that, in concrete calculations, this intersection can be recovered by studying the zero locus of suitable polynomials modulo $p^n$. Letting
\[
\Psi : \Mx(\Z_p)_{\wt{j_b}(\ol u)} \to \Z_p^{2g+1}, \quad \beta : U(\Z_p)_{\ol u} \to \Z_p
\]
be bijections given by analytic coordinate systems (see Section \ref{subsec:analytic-coordinates} for details), we may express both $\kappa$ and $\wt{j_b}$ as convergent $p$-adic power series and we have the diagram

\begin{equation}\label{diag:gqc_coordinates}
	\begin{tikzcd}
	& U(\Z)_{\ol u} \arrow[rr,"\wt{j_b}"]
     \arrow[d]
    &&\ol{\Mx(\Z)}_{\wt{j_b}(\ol u)} \arrow[d] 
    & \Z_p^{2g}  \arrow[dd, "{(\kappa_1, \ldots, \kappa_{2g+1})}"] \arrow[l, twoheadrightarrow,"\kappa"] 
    \\
	& U(\Z_p)_{\ol u} \arrow[rr,"\wt{j_b}"] \arrow[ld,"\beta"]&& \Mx(\Z_p)_{\wt{j_b}(\ol u)} \arrow[rd," \Psi"] \\
\Z_p \arrow[rrrr,"{\theta = (\theta_1, \ldots, \theta_{2g+1})}"] &&&&
    \Z_p^{2g+1} 
	\end{tikzcd}  
\end{equation}
where $\theta$ is defined as $\Psi\circ \wt{j_b}\circ\beta^{-1}$.
 In our computations, we choose the coordinate system $\Psi$ so that the image of $(\kappa_1, \ldots, \kappa_{2g+1})$ is the set of points in $\Z_p^{2g+1}$ with last coordinate equal to $0$. Then, the intersection in \eqref{eq:inters} is in bijection with the zeros of the power series $\theta_{2g+1}$ in \eqref{diag:gqc_coordinates}. We approximate such zeros successively. We start by computing the polynomial $\theta_{2g+1} \pmod p$ and finding its zeros $\beta \in \F_p$. The number of such zeros, counted with multiplicities, gives an upper bound for the cardinality of \eqref{eq:inters}. This bound can then be refined by searching for lifts of these solutions to higher $p$-adic precision. 
In practice, we determine one $p$-adic digit at a time, so this process amounts to solving several explicit polynomial congruences over $\Z/p\Z$. 

The solutions of these congruences will not always correspond to rational points. However, when $\rank(\NS(J_\Q))>2$  we can try to rule out any extra solutions by choosing multiple independent endomorphisms $f$, and repeating the process. At the end, hopefully, no other candidate solutions are left, except for the ones corresponding to the known rational points.

\begin{remark}\label{remark: GQC works}
 Let $\rho$ be the rank of the N\'eron--Severi group of $J_\Q$. 
Assuming $r = g$ and $\rho > 1$, 
\[  \overline{\Mx(\Z)}_{\jbtilde(\ol{u})} \cap \jbtilde(U(\Z_p))_{\ol{u}} 
\] 
is finite. Indeed, a more general finiteness statement is true when $r < g+ \rho -1$, see  \cite[Section 9]{Edil}.
In the general case, the method works by replacing  $\Mx$ with the product of the $\Gm$-torsors over $J$ obtained by pulling back $\Mx$ by $(\id_J, h_i)$ using $\rho - 1$ independent endomorphisms $f_1, \ldots, f_{\rho-1}$ to construct $h_1, \dots, h_{\rho-1}$
as in \eqref{eq:does jbtilde exist}.

In this paper, we work under the assumption $r=g$. 
For the curve  $X_{\ns}^+(N)$, Chen \cite{Chen} shows that the Jacobian $J$ is isogenous to $J_0(N^2)^{+, \rm new}$. By results of Kolyvagin--Logachev (see \cite[\S 7 Appendix]{DograLeFourninfamily}), when the analytic rank of the associated newforms $f \in S_2(\Gamma_0(N^2))^{+, \rm new}$ are all $0$ or $1$, the algebraic rank of the Jacobian is equal to the analytic rank. This makes the assumption $r = g$ easier to verify in this case.
Indeed, the assumption $r=g$ is satisfied at least for all modular curves $X_{\ns}^+(N)$ where $N$ is a prime less than $100$. In general, if $A$ denotes the largest quotient abelian variety of ${\mathrm{Jac}} (X_{\ns}^+(N) )$ having rank equal to $\dim (A)$, then 
the ratio $\dim (A)/\dim ({\mathrm{Jac}}(X_{\ns}^+(N) ))$ is expected to be bounded from below by some absolute constant, and even tend to $1$ when $N$ goes to infinity. See for instance~\cite{KMVdK00} for partial results in the case of $J_0 (\ell)$, $\ell$ a prime. 
\end{remark}

\begin{remark}\label{rmk:no_base_point}
Suppose we do not know a point $b \in X(\Z)$. 
In this case, we can consider a divisor $D$ over $X_\Q$ of degree $d \neq 0$ and replace $j_b$ with the map $j_D \colon X_\Q \to J_\Q$ sending $x \mapsto [d\cdot x-D]$. Then, for each trace~$0$ endomorphism $f\colon J \to J$, we can still form a diagram as in \eqref{eq:does jbtilde exist}, for a suitable $h$ as in \eqref{def:h}. The rest of the machinery still applies, but $\jbtilde$ might not be injective. However, it still has finite fibers. 

The technical difficulty in this case is in rigidifying line bundles: specifically, in Lemma \ref{lem:canonical_isomorphisms_for_jbtilde}, we use that the line bundle  $(j_b, h \circ j_b)^*\calM$   is trivialized at  $b\times U$, and some modification is needed to extend  to the higher degree case. So, we assume for simplicity throughout that some $b\in X(\Z)$ is known.
\end{remark}

\begin{remark}\label{rmk:not_uniqueness_jbtilde_kappa} 
Note that, despite the fact that the notation $\jbtilde$ only involves the base point $b$, the map $\jbtilde$ does in fact depend on more data, namely the endomorphism $f$ and the simple open $U$. Conversely, given $b, f$ and $U$ there is a unique choice of $\jbtilde$ up to sign.
Similarly, there is a map $\kappa$ for each residue disc of $\Mx(\Z_p)$. See Section \ref{sec:kappa_general} for details.
\end{remark}

\begin{remark}
\label{fromTtoM}
Throughout this article, we rephrase the method in \cite{Edil} completely in terms of  the Mumford torsor, and avoid the torsor $T$ in loc.~cit. 
Because of this difference from \cite{Edil}, here the map $\kappa$ in \eqref{eq:kappa_first_appearance} has domain $\Z_p^{2r}$, instead of $\Z_p^r$ as in \cite{Edil}.
\end{remark}

\section{Computations in the Mumford torsor}
\label{sec: general Mumford}

\subsection{Makdisi representations of the Mumford torsor}
\label{sec:mumford_makdisi}

In this section, we review the algorithmic framework detailed in \cite{NicMumford}, in order to explain how one can compute algorithmically and without equations with the Mumford bundle of a modular curve. 
This is possible not only over a field, but also over a local ring such as the truncation $\Z_q/p^e\Z_q$, $e \in \N$, of the ring of integers $\Z_q$ of an unramified extension of $\Q_p$, where $p \ge 5$ is a prime that does not divide the level.

\subsubsection{Computing with Jacobians}\label{sect:MakJac}

Let $X$ be a smooth projective curve of genus $g$ over a local ring $S$.
When $S$ is a field, Makdisi introduced in~\cite{Mak1,Mak2} efficient algorithms to compute in the Jacobian $J$ of $X$. In a nutshell, Makdisi's approach consists in picking a line bundle $\LMak$ of large enough degree $\delta = \deg \LMak \gg_g 0$, so that in particular every point of $J$ is of the form $\LMak^{-1}(D)$ for some effective divisor $D \in \Eff^\delta(X)$; and in representing this point by the subspace
\[ W_D = H^0\big(X,\LMak^2(-D) \big) \subset V = H^0\big(X,\LMak^2\big) \]
of the global section space $V$ of $\LMak^2$ formed by the sections of $\LMak^2$ that vanish at $D$. In order to represent such subspaces $W_D$, Makdisi fixes local trivializations of $\LMak$ at sufficiently many points $P_1, \cdots, P_n$ on $X$, $n \gg_\delta 0$, which he uses to embed $V$ (and its subspaces) into $S^n$; the subspace $W_D$ can thus be encoded by a matrix of size $n \times \dim W_D$ and  with entries in $S$, whose columns correspond to a basis of $W_D$, and where the $i$th entry of each column contains the value at $P_i$ of the corresponding section of $\LMak^2(-D)$ under these trivializations. Makdisi then shows how to reduce computing in $J$ (e.g.\ the group law of $J$, testing whether two such matrices represent the same point of $J$, etc.) to linear algebra over $S$  (of size $O(g) \times O(g)$, since $\delta$ and $n$ can be taken to be $O(g)$).

Mascot then showed in~\cite{NicJac} that these algorithms can be extended to the case where $S$ is a local ring over which $X$ has good reduction.

\subsubsection{Extending to \texorpdfstring{$\Mx$}{Mx}}

Mascot also shows in~\cite{NicMumford} how to extend Makdisi's algorithms for Jacobians to compute in the Mumford torsor $\Mx$ of $J$. 
Mascot picks two nonzero sections $d_0$ and $e_0$ of $\LMak$ whose divisors $D_0 \coloneq (d_0) \in \Eff^\delta(X)$ and $E_0 \coloneq (e_0) \in \Eff^\delta(X)$ do not intersect (over the residue field of $S$). Then, for any choice of effective divisors $D,E$ of degree $\delta$ such that $D+D_0$ is disjoint from $E+E_0$, and for any scalar $\lambda \in S^\times$, Mascot represents the point (see Section \ref{sec:mumford})
\[
\lambda \cdot \Norm_{D-D_0}(1 ) \quad \in \left( \Norm_{D-D_0}\calO_X(E-E_0)\right)^\times (S) = \Mx_{D-D_0, E-E_0}(S)
\]
with the triple
\begin{equation}\label{eq:Mum_point_makdisi}
\left( W_D, W_E, \lambda \right) .
\end{equation}
If the cardinality of (the residue field of) $S$ is large enough, then all pairs of points $(x,y)$ in $J(S) \times J(S)$ can be represented as the classes $(D-D_0, E-E_0)$ for $D,E$ as above, so one is able to represent all points in $\Mx(S)$ in the form \eqref{eq:Mum_point_makdisi} (for this reason, when applying the method with a small prime $p$, we may have to embed $\Mx(\Z_p)$ in $\Mx(\Z_q)$ for $q$ a large enough power of $p$, and work inside $\Mx(\Z_q)$).

Mascot then shows how Makdisi's algorithms extend naturally to points of $\Mx$ in this representation. 
For example, Makdisi's algorithm for the group law on $J$ extends to algorithms for the left and right partial group laws on $\Mx$~\cite[Section \NicMumfordSectGpLaws]{NicMumford}, and Makdisi's equality test in $J$ extends to an algorithm to test whether two given points $x,y \in \Mx(S)$ lie over the same point of $J \times J$, and when they do, to compute the unique $\lambda \in S^\times$ such that $y = \lambda \cdot x$~\cite[Algorithm \NicMumfordAlgCompare]{NicMumford}.

Mascot's algorithms are based on the following fundamental formula (cf.~\cite[Lemma \NicMumfordMainFormula]{NicMumford}). Suppose that the divisors $D_1,D_2,E_1,E_2 \in \Div^0(X)$ are such that the $D_i$ are disjoint from the $E_j$, so that we have the two points
\[ \mu_1 = \Norm_{D_1}(1 ) \quad \in \left( \Norm_{D_1}\calO_X(E_1)\right)^\times (S) = \Mx_{D_1, E_1}(S)\]
and
\[ \mu_2 = \Norm_{D_2}(1 ) \quad \in \left( \Norm_{D_2}\calO_X(E_2)\right)^\times (S) = \Mx_{D_2, E_2}(S).\]
If $D_1 \sim D_2$ and $E_1 \sim E_2$, and we know rational functions $d$ and $e$ on $X$ such that $(d)=D_2-D_1$ and $(e)=E_2-E_1$, then we can recover the unique $\lambda \in \Gm$ such that $\mu_2 = \lambda \cdot \mu_1$ as
\[ \lambda = d(E_1) e(D_2) = d(E_2) e(D_1). \]

\begin{remark}
Since Makdisi's algorithms rely on linear algebra of size $O(g) \times O(g)$, their time complexity is $O(g^3)$ operations in $S$, and their space complexity is $O(g^2)$ elements of $S$. The algorithms for $\Mx$ presented in~\cite{NicMumford}, being extensions of Makdisi's algorithms, have the same complexity.
\end{remark}

\subsubsection{Converting naive forms into Makdisi form}
\label{sec:naive2Makdisi}

When we compute the $\jbtilde$ map in Sections \ref{sect:jbt:on:Zppts} and \ref{sec:explicit_jbtilde}, we will have to handle points of $\Mx$ of the form
\begin{equation}
\Norm_{P_1-P_2}\big(t_{Q_1-Q_2}^{-1}\big)  \in \Norm_{P_1-P_2}\big(\calO_X(Q_1-Q_2)\big)^\times(S)  = \Mx_{P_1-P_2, Q_1-Q_2}(S) , \label{eqn:naiveform}
\end{equation} 
where $P_1,P_2,Q_1,Q_2 \in X(S)$ are points on the curve and $t_{Q_1-Q_2}$ is a rational function on $X$ which is a local equation for $Q_1-Q_2$ around both $P_1$ and $P_2$. We call the 5-tuple \\$(P_1, P_2, Q_1, Q_2, t_{Q_1-Q_2})$ a \emph{naive form} of this point.

Mascot provides an algorithm~\cite[Algorithm \NicMumfordAlgtPQ]{NicMumford} which, given such a naive form, computes a \emph{Makdisi form} \eqref{eq:Mum_point_makdisi} of this point \eqref{eqn:naiveform}.
This algorithm works by finding a representation $D-D_0$ for $P_1-P_2$ and analogously $E-E_0$ for the $Q_1-Q_2$, together with rational functions expressing the linear equivalences, and evaluating these functions (multiplied by a suitable power of $t_{Q_1-Q_2}$) at certain points. A remarkable feature is that it manages to compute such values without losing $p$-adic accuracy even when some points collide mod $p$, which Mascot achieves by using linear algebra to re-express the relevant rational functions as ratios of sections of some power of $\LMak$ so as to solve indeterminacies of the form $0/0 \bmod p$.

More generally, we call the triple
\[ \left(\sum_i m_i P_i \in \Div^0(X),\sum_j n_j Q_j \in \Div^0(X),t \in S(X)\right),
\]
a \emph{naive form} of the point
 \begin{equation}
\Norm_{\sum_i m_i P_i}(t)  \in \Norm_{\sum_i m_i P_i}\left(\calO_X\bigg(\sum_j n_j Q_j \bigg)\right)^\times (S)  = \Mx_{\sum_i m_i P_i, \sum_j n_j Q_j}(S) .\label{eqn:naiveform_2}
\end{equation} 
Here, $S(X)$ denotes the ring of rational functions on $X_S$,
and $t$ is a rational function on $X_S$ that generates $\calO_{X_S}(- \sum_j n_j Q_j)$ in a neighborhood of all the points $P_i$.
We can find a Makdisi representation of 
\eqref{eqn:naiveform_2}
by using the above algorithm along with algorithms for the partial group laws.

\subsubsection{Specializing to modular curves}
\label{sec:makdisi_modular_P}
In order to specialize Makdisi's algorithms for Jacobians to the case where $X \coloneq X(\Gamma)$ is a modular curve of congruence subgroup level $\Gamma \geq \Gamma(N)$, we can take $\LMak$ to be the line bundle $\big(\Omega^1(\text{cusps})\big)^k$ for some $k \in \N$ large enough to ensure $\deg \LMak$ is large enough (cf.\ section~\ref{sect:MakJac}).
When $X(\Gamma)$ has at least three cusps, (which, in particular, is the case of $X_{\ns}^+(N)$ for prime $N \geq 7$), we can take $k=1$; the section space of $\LMak$ is then $M_2(\Gamma)$, and, in the notation of section~\ref{sect:MakJac}, points on the Jacobian of $X(\Gamma)$ are represented by subspaces $W_D$ of $V = M_4(\Gamma)$.

Makdisi also constructs in~\cite{MakMod} specific modular forms for $\Gamma(N)$, defined over $\Z[\tfrac{1}{6N}]$, which are ``moduli-friendly'' in the sense that they are easy to evaluate on triples
\begin{equation}
  (E:y^2=x^3+ax+b,\ P,\ Q)   \label{eqn:MakModPt}
\end{equation}
consisting of a short Weierstrass equation defining an elliptic curve $E$ over $S$ and of a basis $P,Q$ of the $N$-torsion of $E$. Such a triple represents a point of the modular curve $X(N)$, and the choice of a Weierstrass equation of $E$ implies a choice of differential $\omega = d x / 2y$ on $E$, which is used~\cite[Section 4.1]{NicMod} to evaluate these modular forms ``\`a la Katz'' on the point \eqref{eqn:MakModPt}. Furthermore, the algebra generated by these forms is very large, cf.\ Remark~\ref{rmk: generation Mak}.

Assuming
\begin{equation} p \nmid [\Gamma:\Gamma(N)], \label{eqn:p_ndiv_Gamma} \end{equation}
so that the trace map $M_2\big(\Gamma(N)\big) \longrightarrow M_2(\Gamma)$ is surjective mod $p$,
we express $M_2(\Gamma)$ in terms of Makdisi's moduli-friendly forms,
and thus construct a $p$-adic model of the Jacobian of $X(\Gamma)$ as in section~\ref{sect:MakJac}, without requiring any equations or $q$-expansion calculations; all that is required is looking at the $N$-torsion of a few elliptic curves in short Weierstrass form so as to obtain data such as~\eqref{eqn:MakModPt} representing the points $P_i$ at which we evaluate the sections of $\LMak$. See \cite{NicMod} for complete details in the case where $\Gamma_0(N) \geq \Gamma \geq \Gamma_1(N)$.

Incorporating Mascot's extension~\cite{NicMumford} to the Mumford torsor, this makes it possible to work algorithmically with the Mumford torsor of modular curves over local rings such as $\Z_q/p^e\Z_q$, $e \in \N$ and $q$ a power of a prime $p \nmid N$, in an entirely equationless and $q$-expansion-free way.

\begin{remark}\label{rk:fixed_Weil}
We restrict ourselves to triples \eqref{eqn:MakModPt} where the Weil pairing of $P$ and $Q$ is some $N$th root of $1$ which we have fixed once and for all, so as to work in a fixed geometric component of $X(N)$. In particular, $\F_q$ must contain the $N$th roots of unity. 
\end{remark}

\begin{remark}\label{rmk: generation Mak}
In \cite{MakMod}, Makdisi shows that in characteristic 0, the algebra generated by his moduli-friendly forms contains $M_k\big(\Gamma(N)\big)$ for all $k \geq 2$. It follows that this also holds in characteristic $p$ for $p$ large enough. If the traces down to level $\Gamma$ of Makdisi's forms do not span $M_2(\Gamma)$ in characteristic $p$ for our chosen $p$, then our code constructing a $p$-adic model of the modular Jacobian will loop forever; so our code terminating implicitly proves that this holds for our chosen $p$.
We have never encountered a case where our code does not terminate, so it would seem that this continues to hold in every characteristic $p \nmid 6N$, but we do not prove this.
\end{remark}

\subsection{Maps from the curve to the torsor}\label{sec:jbtilde}
 
In this section we give formulas to compute a map $\wt{j_b}: U \to \Mx$ as in \eqref{eq:does jbtilde exist}, first for points of $U$ with values in general schemes $S$ such that $X_S$ is smooth, and then more precisely for $\Z_p$-points.

\subsubsection{Description in terms of line bundles} \label{sec:correspondences_jbtilde}

Given a scheme morphism $h \colon J_\Q\to J_\Q$ such that $j_b^* (\mathrm{id}, h)^* \mum$ is trivial on $U_\Q$, as in \eqref{eq:does jbtilde exist}, we can consider the pullback $\calL_h$ of $\mum$ to $(X \times U)_\Q$ by $j_b \times (h \circ j_b)$.
Denote by $\Delta$ the diagonal morphism $U_{\Q} \to X_{\Q} \times U_{\Q}$. The line bundle
\[
\Delta^* \calL_h = \Delta^* (j_b \times (h \circ j_b))^* \mum =
 ( (\mathrm{id}, h) \circ (j_b, j_b) )^* \mum
  = j_b^* (\mathrm{id}, h)^* \mum
\]
is trivial on $U_\Q  =X_\Q$. Moreover, we have $(b_{\Q},\id)^*\calL_h = \calO_{U_\Q}$ and $(\id,b_\Q)^*\calL_h$ is degree 0 over $U_\Q$.
Conversely, as explained in  \cite[Propositions 7.2 and 7.5]{Edil}, there is a bijective correspondence between
\begin{align*}
\left\{
\begin{array}{c}
\text{invertible } \mathcal{O}\text{-modules } \mathcal{L} \text{ on } (X \times U)_\Q
 \\
\text{of degree } 0 \text{ with respect to } \\
\pi_U: X_\Q \times U_\Q \to U_\Q, \\
\text{ such that } \Delta^* \mathcal{L} \text{ is trivial},
\\
\text{and rigidified on } b_{\Q} \times U_\Q
\end{array}
\right\}
\quad \text{and} \quad
\left\{
\begin{array}{c}
\text{scheme morphisms } \\
h_\Q: J_{\mathbb{Q}} \to J_{\mathbb{Q}} \\
\text{such that } \\j_b^*\bigl((\mathrm{id}, h)^* \Mx\bigr) \text{ is trivial}
\end{array}
\right\}
\end{align*}

We now extend this description over $\Z$. 
There is a map
\begin{align*}
\left\{
\begin{array}{c}
(\calL, s_\Delta, s_b) \text{ such that }
\calL \to X \times U \text{ is}\\
\text{a line bundle of multidegree 0 over } U, \\
s_\Delta  \text{ is a trivializing section of } \Delta^* \calL,  \\
s_b \text{ is a trivializing section of }(b \times \id_U)^* \calL
\end{array}
\right\}
\quad \to \quad
\left\{
\begin{array}{c}
\text{morphisms } \jbtilde: U \to \Mx \\
\text{ that fit in \eqref{eq:does jbtilde exist}}
\end{array}\right\}.
\end{align*}
This is a bijection up to a sign ambiguity: if we multiply $s_\Delta$ and $s_b$ by $(-1)$ we obtain the same morphism $\jbtilde$. 
The triple $(\calL, s_\Delta, s_b)$ determines the morphism
\begin{equation} \label{eq:tildej_b_abstract}
\wt{j_b}\colon U  \lto \Mx \subset \calM \ , \quad 
	 \wt{j_b}(u)  = (u^* s_\Delta ) 
	\otimes (u^* s_b)^{-1} 
\end{equation}
in 
\begin{equation} \label{eq:general_isoms_for_j_b_tilde}
(u^*\Delta^* \calL) \otimes   (u^* \calL|_{b\times U})^{-1}  
= \calL|_{(u,u)} \otimes \calL|_{(b,u)}^{-1} \cong \calL|_{(u,u)} \cong \calM_{u-b,\calL|_{X\times u}}  \cong  (u-b)^*(\calL|_{X\times u}) ,
\end{equation}
where the isomorphism  $\calL|_{(u,u)} \otimes \calL|_{(b,u)}^{-1} \cong \calL|_{(u,u)}$ is the trivialization of $\calL_{(b,u)}$ given by the section $u^* s_b$ and the last isomorphism is the isomorphism in Section \ref{sec:mumford}. 
For the third isomorphism, we use that $\calL$ is a pullback of $\calM$ as follows. There is a unique isomorphism  $\calL \cong (j_b,h \circ j_b)^*\calM$ that sends $s_b$ to the pullback of the trivializing section of $\calM$ at $0 \times J^0$. This gives an isomorphism of $\calL_{(u,u)}$ with $\calM_{u-b, h(u-b)} = \calM_{u-b,  \calL|_{X\times u}}$. 

As in \cite[Section 4]{DRHS}, we can construct a line bundle $\calL$ starting from a trace zero endomorphism $f : J_{\Q} \to J_{\Q}$. 
 Let $T' \subset X_\Q \times X_\Q$ be a correspondence for $f$ whose support does not contain any vertical or horizontal fiber, nor the diagonal. 
 Then, taking $m$ as in \eqref{eq:m}, we define $T \subset X \times U$ to be the Zariski closure 
 \begin{equation}
 \label{eq:correspondence_T}
     T = \ol{m \cdot T'} ,
 \end{equation} 
 and define the  divisors 
\begin{equation}\label{eq:D_H_D_V}
    D_H\coloneq-(b\times\id_U)^*T,
\qquad
D_V\coloneq(b\times\id_U)^*T-\Delta^*T,
\end{equation}
which make sense as Cartier divisors on $U$, since the assumption on $T'$ implies that the support of $T$ does not contain $b \times U$ or the diagonal. To simplify the notation, we will sometimes write $(b\times\id_U)^*T$ as $ \left( T \cap (b\times U) \right)$. 

Then, denoting by $\ol{D}_V$ the Zariski closure of $D_V$ in $X$, the
 line bundle associated with $T + X \times D_H  + \ol{D}_V \times U$ is automatically trivial on the diagonal of $ U\times U$ and on $b\times U$. Moreover, by \cite[Proposition~11.5.2]{BL}, the same line bundle has degree $0$ when restricted to the fibers of the projection $\pi_U : X \times U \to U$. 
However, it does not necessarily have multidegree 0. Since in \eqref{eq:correspondence_T} we multiplied by $m$, there exists a divisor $B$, supported on the fibers   $X_{\F_\ell}$ where $\ell$ is a prime of bad reduction, such that  
\begin{align}\label{eq:L_is_divisorial}
\calL \coloneq \calO_{X\times U}(\calD), \text{ where }
\calD \coloneq T + X \times D_H + \ol{D}_V \times U + B \times U,
\end{align}
has multidegree 0. 
For each bad prime $\ell$, we require $B$ to be supported on the components of $X_{\F_\ell}$ not meeting the base point $b$. The intersection matrix on these components is non-singular, so the condition that $\calD$ have multidegree zero determines $B$ uniquely.

A general method to compute $B$ will be described in Section~\ref{sec:V}.
This $\calL$ thus satisfies all the desired properties and corresponds to a (unique up to sign) $\jbtilde: U \to \Mx$. 
In the rest of this section, 
we describe how to 
compute the morphism $\jbtilde$ from an $\calL$ as in \eqref{eq:L_is_divisorial} and a choice of trivializing sections $s_\Delta$ and $s_b$.

\begin{remark}\phantom{~}
\begin{itemize}
    \item The translation point $c\in J(\Q)$ in  \eqref{def:h} is represented by the divisor 
    \[
    T' \cap (X_\Q \times b)+ T' \cap (b\times X_\Q) - \Delta \cap T'.
    \]
    \item In Section \ref{sec:review-geometric-quadratic-chabauty} we worked with a map $J \to J^0$. Note that here we start with $f : J_\Q \to J_\Q$, represented by the correspondence $T'$, which we then multiply by $m$. This has the effect of killing the component group of $J$, so that the endomorphism represented by $m \cdot T'$ extends over $\Z$ to a map $J \to J^0$.
\end{itemize}

\end{remark}

\subsubsection{The trivializing sections \texorpdfstring{$s_\Delta$}{sDelta} and \texorpdfstring{$s_b$}{sb}}
\label{sec:sections s_b s_Delta}

Suppose that $\calL$ is described as  in \eqref{eq:L_is_divisorial}. 
By construction, the line bundles $\Delta^*\calL$ and $\calL|_{b\times U}$ are trivial and, since we work over $\Z$, they have trivializing sections $s_\Delta$ and $s_b$ which are unique up to $\pm 1$.
In this section we explain how to describe $s_\Delta$ and $s_b$ as rational functions on $U$. 
To do this, we first make isomorphisms $\psi_\Delta$ and $\psi_b$ from the abstract line bundles $\Delta^*\calL$ and $(b, \id)^*\calL$ to $\calO_U(\Delta^* \calD)$ and $\calO_U(T|_{b \times U} + D_H)$, respectively. 
These isomorphisms allow us to describe $s_\Delta$ and $s_b$ in terms of sections $\sigma_\Delta$ and $\sigma_b$ of $\calO_U(\Delta^*\calD)$ and $\calO_U(T|_{b \times U} + D_H)$, which has the advantage that the sections here can be described explicitly as rational functions (in fact, rational numbers).
Once we have constructed the isomorphisms $\psi_\Delta, \psi_b$, we set 
\begin{equation}\label{eq:s_b_as_function}
	s_b = \psi_b^{-1}(\sigma_b)  , \quad s_\Delta = \psi_\Delta^{-1}(\sigma_\Delta)  .
\end{equation}
Given the choices of $D_V$ and $D_H$ in  \eqref{eq:D_H_D_V}, we have $T|_{b \times U} + D_H =0$ and $\Delta^*\calD = B \cap U$, hence we can choose 
\begin{equation}\label{eq:sigma_Delta_sigma_b}
	\sigma_b = 1 \ , \quad \sigma_\Delta = \prod_{\substack{\ell \text{ prime } \\ \ell \mid N}} \ell^{-V_\ell} \ ,
\end{equation}
where $V_\ell$ is the multiplicity of the (unique) component of $U_{\F_\ell}$ in $B$ and $N$ is the product of the primes of bad reduction of $X$ (recall that $V_\ell=0$ when $\ell$ is a prime of good reduction for $X$). Indeed, as a rational function on $U$, the product $\prod_{\ell\mid N}\ell^{-V_\ell}$ has divisor $-B|_U$, and therefore defines a nowhere-vanishing section of $\calO_U(B|_U)=\calO_U(\Delta^*\calD)$. 

We first describe the isomorphism $\psi_\Delta$.
Since we imposed that the support of $T$ does not contain the diagonal over $\Q$, the pullback $\Delta^* \calD$ is a (Cartier) divisor on $U$, and the line bundle $\Delta^* \calL$ is canonically isomorphic to $\calO_U(\Delta^* \calD)$. 
In particular, we have the desired isomorphism
\begin{equation}    \label{eq:psi_Delta}
\begin{array}{cccc}
\psi_\Delta \colon & \Delta^*\calL &\overset{\sim}{\lto} & \calO_U(\Delta^*\calD)  \\
& \Delta^* s &\longmapsto  & s|_\Delta.
\end{array}
\end{equation}

On the other hand, the isomorphism $\psi_b$ is not canonical. Fix a local equation $\eqoverZ$ for $\ol{D}_V + B$ in a neighborhood of $b$ inside $X$.
The choice of $\eqoverZ$ gives an isomorphism
\begin{equation}\label{eq:psi_b} 
\begin{array}{cccc}
	\psi_b \colon & \calL|_{b\times U} & \overset{\sim} {\lto}  & \calO_U((T + X {\times} D_H) \cap (b{\times} U) ) ,
\end{array}
\end{equation}
that is easy to describe on the sections of $\calL|_{b\times U}$ that are restrictions of sections of $\calL$:  
for every section $s \in \calL(U')$ (over each open $U' \subset X \times U$), we have 
\begin{equation}\label{eq:description-psib}
     \psi_b ((b \times \id_U)^* s)  = (s\cdot \pi_1^*\eqoverZ)|_{b\times U}  .
\end{equation}

\begin{remark}
Recall that there is a global sign ambiguity in the choice of $s_b$. If we consider $s_b$ as fixed, the choice  $\sigma_b=1$ determines a choice of $\eqoverZ$. Conversely, we can arbitrarily choose $\sigma_b=1$ and $\eqoverZ$ (which defines $\psi_b$), and then define $s_b=\psi_b^{-1}(1)$.
\end{remark}

These constructions already allow us to describe $\jbtilde$ in some simple cases.
\begin{prop}
        Let $N$ be the product of the primes of bad reduction for $X$.
        Suppose $b_\Q$ does not lie in the support of $(D_V)_\Q$. Let $S$ be a $\Z[1/N]$-scheme and $u \in U(S)$ be a point such that $(u,u)$ and $(b,u)$ do not meet $T$. Then
        \[
\jbtilde(u) = \sigma_\Delta \cdot u^*1 \otimes b^*1 \quad \in \Mx_{u-b, T \cap (X{\times} u) + D_V}.
\] 
\end{prop}
\begin{proof}
Equation \eqref{eq:tildej_b_abstract} and the definition of $\sigma_\Delta, \sigma_b$ give
\[
\begin{aligned}
\jbtilde(u) & = u^*(s_\Delta) \otimes u^*(s_b)^{-1} \\
&= u^*(\psi_\Delta^{-1}(\sigma_\Delta)) \otimes u^*(\psi_b^{-1}(\sigma_b))^{-1} \\
& =\sigma_\Delta \cdot u^*(\psi_\Delta^{-1}(1)) \otimes u^*(\psi_b^{-1}(1))^{-1}\\
&= \sigma_\Delta \cdot (u,u)^*1 \otimes u^*(\psi_b^{-1}(1))^{-1} \\
&= \sigma_\Delta \cdot (u,u)^*1 \otimes (b,u)^*1 .
\end{aligned}
\]
The second equality holds by \eqref{eq:s_b_as_function}  and the third because $\sigma_b =1$ and $\sigma_\Delta \in \Q$. 
We now consider the fourth equality. Since $(u,u), (b,u)$ do not meet $T$, the point $u$ is not in the support of $D_H$ or $D_V$ (see \eqref{eq:D_H_D_V}).  Since $S$ is a $\Z[1/N]$-scheme, the vertical divisor $B$ disappears after base change to $S$. Therefore $(u,u)$ and  $(b,u)$ do not meet $\calD$ (see \eqref{eq:L_is_divisorial}). Thus over some open set of $X\times U$ containing $(u,u)$, the section $1$ trivializes $ \calL = \calO_{X \times U}(\calD)$. The map $\psi_\Delta$ sends $\Delta^*1$ to a section $1 \in \calO_U(\Delta^* \calD)$, hence $\psi_\Delta^{-1}(1)$ is $\Delta^*1$. In particular, we get
\[
u^*(\Delta^*1) = (\Delta \circ u)^*(1) = (u,u)^*(1),
\]
as desired.
For the final equality, a similar argument works. Since $b_\Q$ does not lie in the support of $(D_V)_\Q$ and $(b,u), (u,u)$ do not meet $T$, the section $1$ trivializes $\calL$ over an open set. The map $\psi_b$ sends $(b, \id)^*1$ to a section $1 \in \calO_U(\calD|_{b \times U})$ and so $\psi^{-1}_b(1) = (b,\id)^*1$ and hence $u^*(\psi^{-1}_b(1)) = u^*(b,\id)^*(1) = (b,u)^*(1)$.
\end{proof}

We end this section with a technical lemma that allows us to identify the line bundles in  \eqref{eq:general_isoms_for_j_b_tilde}. 
This lemma justifies the form and definition of $\jbtilde$ given in \eqref{eq:tildej_b_abstract} which can now be deduced from \cite[Proposition 7.8]{Edil}.
\begin{lemma}\label{lem:canonical_isomorphisms_for_jbtilde}
    If $u$ is an $S$-point on $X$ and $X_S$ is smooth over $S$, then the composition of isomorphisms   \eqref{eq:general_isoms_for_j_b_tilde} is the identity in 
      \[
      (u^*\Delta^* \calL) \otimes   (u^* \calL|_{b\times U})^{-1}  
= \calL|_{(u,u)} \otimes \calL|_{(b,u)}^{-1} = (u-b)^*(\calL|_{X\times u}) .
      \]
\end{lemma}
\begin{proof}
We start by looking at the last two isomorphisms in \eqref{eq:general_isoms_for_j_b_tilde}, since they involve the Mumford torsor. The second-to-last isomorphism uses the identification between $\calL$ and the pullback of $(j_b, h \circ j_b)^*\calM$, which uses 
the trivializing section $s_b$ as follows (see \cite{Edil}, above Proposition 7.2). 
The map $(j_b, h \circ j_b)$ sends $b\times U$ into $0 \times J^0$ and  $\calM$ is trivialized along $0 \times J^0$ through a section $1_{0 \times J^0}$ and we use the unique isomorphism  $\calL \cong (j_b, h \circ j_b)^*\calM$ sending $s_b$ to  $(j_b, h \circ j_b)^*1_{0 \times J^0}$.
Pulling back this isomorphism by $\id_X \times u$, we obtain an isomorphism 
\begin{equation} \label{eq:technical_iso1}
    \calL|_{X\times u} \cong j_b^* \calM_{ J \times \{\calL|_{X\times u}\}}, 
\end{equation}
which is the only one sending $u^* s_b$ to $1_{0\times J^0}|_{h(u-b)}$. 
In particular, the second-to-last isomorphism in \eqref{eq:general_isoms_for_j_b_tilde} is defined as the pullback of the above isomorphism at $u$.

Now the last isomorphism in \eqref{eq:general_isoms_for_j_b_tilde}. 
By \cite[Equation 6.3.3]{Edil} we have a canonical isomorphism 
\begin{equation} \label{eq:technical_iso2}
j_b^* \calM_{ J \times \calL|_{X\times u}} \cong \calL|_{X\times u} 
\end{equation}
sending the rigidification given by $j_b^* 1_{0 \times J^0}$ to the one given by $u^*s_b$. 
We can also use the rigidification of $\calL|_{X\times u}$ at $b$ given by $u^*s_b$ to define the isomorphism
\begin{equation} \label{eq:technical_iso3}
\calL|_{X\times u} \cong \calL|_{X\times u}  \otimes b^*\calL|_{X\times u}^{-1} , \quad s \to s\otimes (u^*s_b)^{-1} .
\end{equation}

Then, the isomorphism in $\calM_{u-b,\calL|_{X\times u}}  \cong  (u-b)^*(\calL|_{X\times u})$ in Section \ref{sec:mumford} is the pullback at $u$ of the composition of  \eqref{eq:technical_iso2} and \eqref{eq:technical_iso3}. 
In particular, the composition of the last two isomorphisms in \eqref{eq:general_isoms_for_j_b_tilde} is the pullback at $u$ of the composition of \eqref{eq:technical_iso1}, \eqref{eq:technical_iso2} and \eqref{eq:technical_iso3}.
Notice that the composition of the first two is the identity of $\calL|_{X\times u}$, since it sends the rigidification $u^* s_b$ to itself. In particular, the composition of the last two isomorphisms in \eqref{eq:general_isoms_for_j_b_tilde} is the pullback at $u$ of \eqref{eq:technical_iso3}. The claim follows, since the other isomorphism in \eqref{eq:general_isoms_for_j_b_tilde} is by definition the inverse of the pullback at $u$ of \eqref{eq:technical_iso3}.
\end{proof}

\subsubsection{The map \texorpdfstring{$\tilde{j_b}$}{jbtilde} over \texorpdfstring{$\Z[1/N]$}{Z[1/N]}}
We can now prove a general formula for $\jbtilde$ on $S$-points of $U$ when $X_S$ is smooth, removing the assumption that $(u,u)$ and $(b,u)$ do not meet $T$, and the assumption that $b_\Q$ does not meet $(D_V)_\Q$.
Throughout the section we assume the following notation.

\begin{notation}\label{notation_for_jbtilde}
    We keep the notation and assumptions of \eqref{eq:L_is_divisorial}, \eqref{eq:s_b_as_function}, \eqref{eq:sigma_Delta_sigma_b}, and \eqref{eq:psi_b}.
Let $S$ be a local regular $\Z[1/N]$-scheme, so that in particular $X_S$ and $U_S$ are smooth over $S$. For $u \in U(S)$,  we define the divisor 
    \[
    D(u)\coloneq (\id_X \times_\Z u)^*(\calD-X\times D_{H})
      = T_S  \cap (X_S {\times} u) + D_{V,S}
    \]
 on $X \times_\Z u \cong X_S$.
\end{notation}

\begin{lemma}\label{lemma:preliminaries-jbtilde} 

Assume Notation \ref{notation_for_jbtilde}. The following hold:
    \begin{enumerate}
        \item For any smooth $S$-scheme $Z$, any (Weil, or equivalently Cartier) divisor $E$ on $Z$, and any $z \in Z(S)$, there exists a local equation $t_{E, z}$ for $E$ around $z$ (that is, a rational function on $Z$ such that there exists an open neighborhood $V$ of $z$ in $Z$ with $\divisor(t_{E,z})|_{V} = E|_{V}$).
        \item\label{item:technical_isomorphic_bundles} The line bundles $\calL|_{X\times u}$ and $\calO_X(D(u))$ are isomorphic. Any isomorphism $\calL|_{X\times u} \cong \calO_X(D(u))$ is induced by the choice of a local equation $t_{{D_H},u}$ for $D_H$ around $u$.
        \item        The description of the fibers of the Mumford bundle in Section \ref{sec:mumford} gives isomorphisms
    \begin{equation}\label{eq:iso_divisors_general_tildejb}
	 (u-b)^* \calL|_{X \times u} \simeq \calM_{u-b,\calL|_{X\times u}} = \calM_{u-b,\calO_X(D(u))} \simeq (u-b)^* \calO_X(D(u)).
	\end{equation}
    Note that the middle equality makes sense since $\calL|_{X \times u}$ and $\calO_X(D(u))$ are the same line bundle up to isomorphism by \ref{item:technical_isomorphic_bundles}, hence they give the same point of $J$.
    For any isomorphism $\psi_u : \calL|_{X \times u} \xrightarrow{\sim} \calO_X(D(u))$, the above composition can be described on sections by
\begin{equation}\label{eq:first_part_iso_divisors_general_tildejb}
(u-b)^* \calL|_{X\times u}  \overset{\sim}{\lto} (u-b)^*  \calO_X(D(u))\ , \quad 
u^*s_1 \otimes  b^* s_2^{-1} \longmapsto u^* \psi_u(s_1) \otimes 
  b^* \psi_u(s_2)^{-1}.
\end{equation}
    \end{enumerate}
\end{lemma}

\begin{proof}
Since $S$ is local, $S = \Spec R$ with $(R, \mathfrak{m})$ a local ring. Let $\eta$ be the topological point of $\Spec R$ corresponding to the maximal ideal $\mathfrak{m}$.

\begin{enumerate}
    \item Recall that a smooth scheme over a regular base is regular, so $Z$ is regular.
    Every divisor on a regular scheme is a Cartier divisor, so it admits a local equation in a Zariski neighborhood around the point $z(\eta)$.
    \item Choose a local equation $t_{D_H, u}$ for $D_H$ around $u$ (note that $X \times u \cong X_S$ is smooth). This induces an isomorphism
\begin{equation*} 
	\psi_u \colon \calL|_{X\times u}  \overset{\sim}{\lto} \calO_X(D(u)).
\end{equation*}
We describe this isomorphism on sections that are restrictions of sections on open subsets of $X \times U$.
Identify $\calL$ with $\calO_{X\times U }(\calD)$ and let $V$ be an open of $X \times U$. We may then view each section $s \in \calL(V)$ as a rational function on $V$ with $\divisor(s) \geq -\calD$. We similarly view $\pi_2^*t_{D_H, u}$ as a rational function on $X \times U$; one checks that $\pi_2^*t_{D_H, u}$ is a rational function with divisor $X \times D_H$, so $s \cdot \pi_2^*t_{D_H, u}$ is a section of $\mathcal{O}_{X \times U}( \mathcal{D} - X \times D_H )$. We now simply restrict this rational function to $X \times u$. We write this symbolically as
\begin{equation}\label{eq:description-psi-u}
\psi_u \colon (\id_X \times u)^*s \longmapsto (s\cdot \pi_2^* t_{D_H,u} )|_{X\times u}.
\end{equation}
\item This is a consequence of \cite[Proposition 6.3.2 and Remark 6.3.12]{Edil}. \qedhere
\end{enumerate}
\end{proof}

\begin{prop}\label{prop:general_formula_j_btilde}
Assume Notation \ref{notation_for_jbtilde}. Choose rational functions $t_{T, (b,u)}, t_{T, (u,u)}$ on $(X \times U)_S$ and $t_{D(u), u}, t_{D(u), b}$ on $X_S$ with the property that $t_{E,z}$ is a local equation for the divisor $E$ around the point $z$, as in Lemma \ref{lemma:preliminaries-jbtilde}.
 Under the isomorphisms \eqref{eq:iso_divisors_general_tildejb}, for every $u \in U(S)$ we have 
	\begin{equation} \label{eq:eq_jb_tilde_general}
	\begin{aligned}
	&
	\wt{j_b}(u) = \mu \cdot u^*(t_{D(u), u}^{-1}) \otimes b^*(t_{D(u), b}) \quad \in (u-b)^*\calO_X(D(u)), \text{ where }
	\\ & \quad 
	\mu = 
		\frac{\sigma_\Delta \cdot t_{T, (u,u)}|_{\Delta} \cdot  t_{D(u), u} }{\sigma_b \cdot t_{T, (u,u)}|_{X\times u} \cdot t_{T, (b,u)}|_{b\times U}} (u) 
		\cdot 
		\frac{t_{T, (b,u)}|_{X\times u} \cdot \eqoverZ}{t_{D(u), b}} (b)  \quad \in \calO_S(S)^\times \ .
	\end{aligned}
	\end{equation}
\end{prop}

\begin{proof}

By \eqref{eq:tildej_b_abstract}, we have $\jbtilde(u) = (u^*s_\Delta) \otimes (u^*s_b)^{-1}$. 
Represent $s_\Delta$ as $\Delta^*(\tilde{s}_{u,u})$ for some section $\tilde{s}_{u,u}$ of $\calL$ over an open neighborhood of $(u,u)$ in $X \times U$.
Now observe that
\[
u^*s_\Delta = u^* \Delta^*(\tilde{s}_{u,u}) = (\Delta \circ u)^*(\tilde{s}_{u,u}) = (u,u)^*(\tilde{s}_{u,u})=\left( (\id, u) \circ u \right)^*(\tilde{s}_{u,u}) = u^* (\id,u)^*(\tilde{s}_{u,u}),
\]
so that in particular
\[
(u^*\psi_u)(u^*s_\Delta) = (u^*\psi_u)(u^* (\id,u)^*(\tilde{s}_{u,u})) = u^*\left( \psi_u( (\id,u)^*(\tilde{s}_{u,u}) ) \right).
\]
By our description of $\psi_u$ in \eqref{eq:description-psi-u}, we may compute $\psi_u( (\id,u)^*(\tilde{s}_{u,u}) )$ by the formula
\[
\psi_u( (\id,u)^*(\tilde{s}_{u,u}) ) = (\tilde{s}_{u,u} \cdot \pi_2^*t_{D_H, u})|_{X \times u}.
\]
Similarly, choose a section $\tilde{s}_{b,u}$ around $(b,u)$ such that $(b \times \id)^*(\tilde{s}_{b,u}) = s_b$ and obtain
\[ u^* s_b = u^* (b \times \id)^* \tilde{s}_{b,u} = (b,u)^*\tilde{s}_{b,u} = b^* (\id, u)^* \tilde{s}_{b,u},   \]
where $b^*(\id, u)^* \tilde{s}_{b,u}$ is a section of $b^* (\id, u)^*\calL = b^* \calL|_{X \times u}$. We identify this line bundle to $b^* \calO_{X}(D(u))$ via $b^*\psi_u$ and compute
\[
(b^*\psi_u)\left(b^* (\id, u)^* \tilde{s}_{b,u} \right) = b^*\left( \psi_u( (\id, u)^* \tilde{s}_{b,u} ) \right).
\]
Using again our description of $\psi_u$ from Lemma \ref{lemma:preliminaries-jbtilde}, this is
\[
b^*\left( \psi_u( (\id, u)^* \tilde{s}_{b,u}) \right) = b^*\left( (\tilde{s}_{b,u} \cdot \pi_2^*t_{D_H, u})|_{X \times u} \right).
\]
We have thus proved
\begin{equation} \label{eq:tecnic_zeroth_tildejb}
	\wt{j_b}(u)
	= u^* \left((\wt s_{u,u}\cdot \pi_2^* t_{D_H,u} )|_{X\times u}\right)	\otimes  b^* \left((\wt s_{b,u}\cdot \pi_2^* t_{D_H,u} )|_{X\times u}\right)^{-1} \, .
\end{equation}

We now observe that 
$t_{D(u),u}  \cdot (\wt s_{u,u}\cdot \pi_2^* t_{D_H,u})|_{X\times u}$ is regular and non-vanishing at $u$ as a section of $\calO_X(D(u))$. To see this, we  first note that by definition the divisor of $\tilde{s}_{u,u}$, which is locally a generating section of $\calL \cong \calO_{X \times U}(\calD)$, coincides (locally) with $-\calD$. We have already observed that $\divisor(\pi_2^* t_{D_H, u})=X \times D_H$, hence $\divisor\left((\tilde{s}_{u,u} \cdot \pi_2^* t_{D_H, u})|_{X \times u}\right)$ coincides locally with $-D(u)$. Multiplying by $t_{D(u), u}$ cancels this out, giving a regular and non-vanishing section at $u$. As a consequence, $u^*\left(t_{D(u),u}  \cdot (\wt s_{u,u}\cdot \pi_2^* t_{D_H,u})|_{X\times u} \right)$ is simply the evaluation $\left(t_{D(u),u}  \cdot (\wt s_{u,u}\cdot \pi_2^* t_{D_H,u})|_{X\times u} \right)(u)$. 
The same argument shows that $ t_{D(u),b} \cdot (\wt s_{b,u}\cdot \pi_2^* t_{D_H,u})|_{X\times u}$ is regular and non-vanishing at $b$, so $b^*$ is evaluation at $b$ for this section.

Multiplying and dividing by $t_{D(u), b}$ in \eqref{eq:tecnic_zeroth_tildejb} we then obtain
\begin{equation} \label{eq:tecnic_first_tildejb}
	\wt{j_b}(u)
 = \frac{	\left( t_{D(u),u}  \cdot (\wt s_{u,u}\cdot \pi_2^* t_{D_H,u})|_{X\times u} \right)(u) }  
{	\left( t_{D(u),b} \cdot (\wt s_{b,u}\cdot \pi_2^* t_{D_H,u})|_{X\times u} \right) (b) }
\cdot u^* ( t^{-1}_{D(u),u})  \otimes b^* (t_{D(u),b}) 
\ ,
\end{equation}
where the first factor is an element of $\calO_S(S)^\times$. 
We are left to compare this first factor with the value $\mu$ in the statement. 

Choose a local equation $t_{D_V, u}$ for $D_V$ around $u$ and set $t_{\calD, (u,u)} \coloneq t_{T, (u,u)} \cdot \pi_2^* t_{D_H, u} \cdot \pi_1^* t_{D_V, u}$. The rational function $t_{\calD, (u,u)}$ has the same poles and zeros as $\wt s_{u,u}^{-1}$ in a neighborhood of $(u,u)$, i.e.~$\tilde s_{u,u} \,  t_{\calD, (u,u)}$ is regular and invertible in $(u,u)$. 
We deduce the following equalities in $\calO_S(S)^\times$, where we only evaluate sections that are regular and invertible at the evaluation point:
\begin{equation*}
	\begin{aligned}
 \left( t_{D(u),u}  \cdot (\wt s_{u,u}\cdot \pi_2^* t_{D_H,u})|_{X\times u} \right)(u) &= \frac{t_{D(u),u}}{(t_{\calD, (u,u)}/\pi_2^* t_{D_H,u})|_{X\times u}} (u) \cdot (\wt s_{u,u} \cdot  t_{\calD, (u,u)} )(u,u)
\\ 
& =  \frac{t_{D(u),u}}{t_{T, (u,u)}|_{X\times u} \cdot t_{D_V, u}} (u)  \cdot
 (\wt s_{u,u} \cdot  t_{\calD, (u,u)})|_{\Delta}(u) 
 \\ & = 
 \frac{t_{D(u),u}}{t_{T, (u,u)}|_{X\times u} \cdot t_{D_V, u}} (u)  \cdot  \left( \sigma_\Delta \cdot t_{T, (u,u)}|_\Delta \cdot  t_{D_V,u} \cdot t_{D_H,u}  \right)(u) \\
 & =    \frac{t_{D(u),u} \cdot \sigma_\Delta  \cdot t_{D_H,u} \cdot t_{T, (u,u)}|_\Delta }{t_{T, (u,u)}|_{X\times u} } (u)  \ .  
	\end{aligned}
\end{equation*}
Note that the third equality uses the definition $ \tilde{s}_{u,u}|_{\Delta} = \sigma_\Delta$.

Analogously, using that the rational function $t_{\calD, (b,u)} \coloneq t_{T, (b,u)} \cdot \pi_2^* t_{D_H, u} \cdot \pi_1^* \eqoverZ$  
has the same poles and zeros of $\wt s_{b,u}^{-1}$ in a neighborhood of $(b,u)$, we obtain
\begin{equation*}
	\begin{aligned}
	 	\big( t_{D(u),b} \cdot (\tilde s_{b,u}\cdot & \pi_2^* t_{D_H,u})|_{X\times u} \big) (b)	  = \frac{t_{D(u),b}}{(t_{\calD, (b,u)}/\pi_2^* t_{D_H,u})|_{X\times u}} (b) \cdot (\tilde s_{b,u} \cdot  t_{\calD, (b,u)} )(b,u)
	\\ 
	& =  \frac{t_{D(u),b}}{t_{T, (b,u)}|_{X\times u} \cdot \eqoverZ } (b)  \cdot
	(\tilde s_{b,u} \cdot  t_{\calD, (b,u)})|_{b\times U}(u) 
	\\ 
	& =  \frac{t_{D(u),b}}{t_{T, (b,u)}|_{X\times u} \cdot \eqoverZ } (b)  \cdot
	\left. \left( (\tilde s_{b,u} \, \pi_1^*\eqoverZ ) \cdot \frac{ t_{\calD, (b,u)}}{\pi_1^*\eqoverZ} \right)\right|_{b\times U}(u) 
	\\
	&  =  \frac{t_{D(u),b}}{t_{T, (b,u)}|_{X\times u} \cdot \eqoverZ } (b)  \cdot
	 \left( \sigma_b \cdot t_{T, (b,u)}|_{b\times U} \cdot t_{D_H,u} \right)(u) \ .
	\end{aligned}
\end{equation*}
For the fourth equality, note that we have
\[
\sigma_b = \psi_b(s_b) = \psi_b( (b \times \id)^* \tilde{s}_{b, u} ) = \left( \tilde{s}_{b,u} \cdot \pi_1^*  \eqoverZ
\right)|_{b \times U}
\]
by the description \eqref{eq:description-psib} of $\psi_b$.

Substituting these expressions into \eqref{eq:tecnic_first_tildejb} we get the desired expression for $\wt{j_b}(u)$.
\end{proof}

\subsubsection{The map \texorpdfstring{$\tilde{j_b}$}{jbtilde} on \texorpdfstring{$\Z_p$}{Zp}-points}\label{sect:jbt:on:Zppts}

We now specialize to $S$ being $\Spec(\Z_p)$, where $p$ is a prime of good reduction for $X$. Throughout this section, we use the following notation.
\begin{notation}\label{not: Ti}
    Recall the correspondence $T$ from \eqref{eq:correspondence_T}. We write $T = \sum_i n_i T_i$ as a sum of irreducible divisors $T_i \subset X \times U$.
\end{notation}

\begin{defi}\label{def:multiplicative_local_eqs}
    Let  $X$ be a scheme over $\ol{\Z}_p$. Let $\mathcal{B}=\{x_1, \ldots, x_n\}$ be a finite set of points in $X(\ol \Z_p)$ and $\mathcal{D}$ be a finitely generated subgroup of $\operatorname{Div}(X)$, the group of Cartier divisors of $X$. Fix, for every divisor $D \in \mathcal{D}$ and every point $x \in \mathcal{B}$, a local equation $t_{D,x}$ for $D$ around $x$. We say that the equations $t_{D,x}$ are chosen \emph{multiplicatively} if we have $t_{D_1 \pm D_2,x} = t_{D_1,x} \cdot t_{D_2,x}^{\pm 1}$ for any choice of $D_1,D_2 \in \mathcal{D}$ and $x \in \mathcal{B}$. 
\end{defi}

\begin{example}\label{example: multiplicative equations}
For each point $x \in \mathcal{B}$ choose a local parameter $t_{x}$ at its reduction $\ol x$. Given a divisor $D \in \mathcal D$ 
on $X$, we write
$
D = n_1 Q_1 + \ldots + n_s Q_s + D',
$
where the $Q_i$ are points in $X(\calO_{\ol{\Q}_p})$ reducing to $\ol x$ and the reduction of $D'$ does not contain $\ol x$. 
The function
\begin{equation}
    t_{D,x} = \prod_{i=1}^s (t_{\ol x}- t_{\ol x}(Q_i))^{n_i}
\end{equation}
is a local equation for $D$ around $x$. These equations are chosen multiplicatively. In our computations, we follow a variant of this strategy, as explained in Remark \ref{rmk:multiplicative_local_eqs_almost}.
\end{example}

The formula that we use to compute $\jbtilde$ in practice is the following.
\begin{prop}\label{prop:intermediate__formula_j_btilde}
Assume Notations \ref{notation_for_jbtilde} and \ref{not: Ti}. Recall the local equation $\eqoverZ$ for $\overline{D}_V+B$ around $b$ (this is defined over $\Z$ and was introduced right before \eqref{eq:psi_b}).
Also fix local equations $t_{\bullet, \bullet}$ over $\Z_p$, chosen multiplicatively as in Definition \ref{def:multiplicative_local_eqs} (for ${\cal B}=\{b, u\}$).
 If $u \in U(\Z_p)$ is a point such that $(u_{\Q_p}, u_{\Q_p})$ and $(b_{\Q_p}, u_{\Q_p})$ are not in any $T_i$, we have
\begin{equation}\label{eq:tildej_b_quasi_easy_Q_p}
 \begin{aligned}
 \wt{j_b}(u) 
    &  =  \sigma_\Delta \cdot \frac{\eqoverZ}{t_{D_V,b}}(b)  \cdot 
    \prod_i \left( \frac{t_{T_i \cap (X \times u ),u}(u)}{t_{\Delta^*T_i,u}(u)} \right)^{n_i} \!
    \cdot 
     \prod_i \left( \frac{t_{T_i \cap (b\times U),u}(u)} { t_{T_i \cap (X \times u),b}(b)}  \right)^{n_i}
    \cdot
    u^* ( t_{D(u),u}^{-1})  \otimes b^* (t_{D(u),b})   .
    \end{aligned}
\end{equation}
Moreover, each of the ratios in the above formula is an invertible element of $\Z_p$.
\end{prop}

\begin{remark}
Note that in this statement, all the schemes are silently base changed to $\Z_p$. We point out in particular that, while $b$ may not be a point of $U(\Z)$, the base change $b_{\Z_p}$ is a point of $U(\Z_p)$, since $p$ is a prime of good reduction. Similarly, one checks that $T_{\Z_p}$ and the $(T_i)_{\Z_p}$ are independent of the choice of simple open $U$.
\end{remark}

\begin{proof}
We start from the invertibility of the ratios appearing in 
 \eqref{eq:tildej_b_quasi_easy_Q_p}. 
Since $X$ is smooth over $S=\Spec \Z_p$, the divisors $(\overline{D}_V)_{\Z_p}$ and $(D_V)_{\Z_p}$ are the same, and $b$ does not meet $B$. Thus, since ${\eqoverZ}$ and ${t_{D_V,b}}$ are local equations for the same divisor around $b\in X(\Z_p)$, their ratio is invertible around $b$ and therefore $\frac{\eqoverZ}{t_{D_V,b}}(b)\in \Z_p^\times$.

Next, we consider the ratio $\frac{t_{T_i \cap (X \times u ),u}(u)}{t_{\Delta^*T_i,u}(u)}$. Choose a local equation $t_{T_i,(u,u)}$ for $T_i$ around $(u, u)$. Since $(u_{\Q_p},u_{\Q_p})$ does not lie in $T_i$, we have

 \begin{equation}     \label{eq:technical_non_zero}
 t_{T_i, (u,u)}(u,u) \neq0, \quad t_{\Delta^*T_i,u}(u) \neq 0, \quad t_{T_i, (u,u)}|_{X\times u}(u)
 \neq 0 \quad \text{in } \Z_p \ .
  \end{equation}
  Moreover, the rational functions $\frac{t_{T_i, (u,u)}|_\Delta}{t_{\Delta^*T_i,u}}$, respectively $\frac{t_{T_i, (u,u)}|_{X\times u}}{t_{T_i  \cap (X \times u ),u}}$, are regular and invertible around $u$, since both the numerator and denominator are local equations for $\Delta^*T_i$, respectively $T_i  \cap (X \times u )$, around $u$. In particular we have
  \[
  \frac{t_{T_i, (u,u)}|_{X\times u}}{t_{T_i  \cap (X \times u ),u}}(u) \ , \quad  \frac{t_{\Delta^*T_i,u}}{t_{T_i, (u,u)}|_\Delta}(u)  
  \quad \in \Z_p^\times \ . 
  \]
  Together with \eqref{eq:technical_non_zero}, this implies that also $t_{T_i \cap (X \times u ),u}(u)$ and $t_{\Delta^*T_i,u}(u)$ are non-zero $p$-adic numbers and that 
 \begin{equation}\label{eq:technical_invertible}
     \frac{t_{T_i \cap (X \times u ),u}(u)}{t_{\Delta^*T_i,u}(u)} = \frac{t_{T_i \cap (X \times u ),u}(u)}{t_{T_i, (u,u)}(u,u)} \cdot \frac{t_{T_i, (u,u)}(u,u)}{t_{\Delta^*T_i,u}(u)} 
     = \frac{t_{T_i  \cap (X \times u ),u}}{t_{T_i, (u,u)}|_{X\times u}}(u) \cdot  \frac{t_{T_i, (u,u)}|_\Delta}{t_{\Delta^*T_i,u}}(u)  \in \Z_p^\times\ .
 \end{equation}
 For the ratio $\frac{t_{T_i \cap (b\times U),u}(u)}{ t_{T_i \cap (X \times u),b}(b)}$ the argument is analogous, using a local equation $t_{T_i, (b,u)}$ for $T_i$ around $(b, u)$. 

It remains to derive \eqref{eq:tildej_b_quasi_easy_Q_p} and it is enough to do it over $\Q_p$.
Since $(u_{\Q_p},u_{\Q_p})$ and $(b_{\Q_p},u_{\Q_p})$ do not lie on any $T_i$, we may take $t_{T,(u,u)}=1$ and $t_{T,(b,u)}=1$ in Proposition~\ref{prop:general_formula_j_btilde}.
Recalling also that $\sigma_b=1$ in \eqref{eq:sigma_Delta_sigma_b}, we thus have
\[
\widetilde{j_b}(u) = \sigma_\Delta\cdot\frac{\eqoverZ}{t_{D(u),b}}(b) \cdot t_{D(u),u}(u) \cdot u^*(t_{D(u),u}^{-1})\otimes b^*(t_{D(u),b}).
\]
Using the multiplicativity of the local equations and the equalities $D(u)=T\cap(X\times u)+D_{V}$ and $D_{V}=T\cap(b\times U)-\Delta^*T$, we find 
\[
t_{D(u),u} = \prod_i\left( \frac{t_{T_i\cap(X\times u),u}\,t_{T_i\cap(b\times U),u}} {t_{\Delta^*T_i,u}} \right)^{n_i} \quad \text{and} \quad t_{D(u),b} = t_{D_V,b}\cdot \prod_i t_{T_i\cap(X\times u),b}^{n_i}.
\]
Substituting these identities gives exactly \eqref{eq:tildej_b_quasi_easy_Q_p}.
\end{proof}

For practical computations, we will use the following variants of the above proposition.  In particular, the formulas in the next lemma allow us to remove the scalar $\sigma_\Delta\cdot \frac{\eqoverZ}{t_{D_V,b}}(b)$ from \eqref{eq:tildej_b_quasi_easy_Q_p} by expressing it in terms of more easily computable quantities.

\begin{lemma} \label{lem:practical_formula_j_btilde}
Fix a base point $b \in X(\Z)$. Let $U$ be any simple open of $X$ and $U_b$ be the unique simple open that contains $b$. The following hold for any multiplicative choice of local equations $t_{\bullet, \bullet}$.
\begin{enumerate}
    \item Suppose that $U = U_b$ and that $(b_\Q,b_\Q)$ does not lie on any $T_i$. We have
\begin{equation}\label{eq:NOT_computing_sigma_Delta}
 \begin{aligned}
  \sigma_\Delta  \cdot \frac{\eqoverZ}{t_{D_V,b}}(b) = \pm \prod_i \left( \frac{t_{\Delta^*T_i,b}(b)}{t_{T_i \cap (b\times U),b}(b)} \right)^{n_i} .
    \end{aligned}
\end{equation}    

\item 
For any $U$, still assuming that $(b_\Q,b_\Q)$ does not lie on any $T_i$, we have
\begin{equation}\label{eq:NOT_computing_sigma_Delta_2}
 \begin{aligned}
  \sigma_\Delta  \cdot \frac{\eqoverZ}{t_{D_V,b}}(b) =   \pm  \prod_{q\mid N} q^{-V_{U,q} + V_{U_b,q}}
      \prod_i \left( \frac{t_{\Delta^*T_i,b}(b)}{t_{T_i \cap (b \times U ),b}(b)} \right)^{n_i},
    \end{aligned}
\end{equation}    
where the $V_{U, q}$ and $V_{U_b, q}$ are as in \eqref{eq:sigma_Delta_sigma_b} for the simple opens $U$ and $U_b$, respectively.
\end{enumerate}
\end{lemma}
\begin{proof}

In the first case, we can apply  \eqref{eq:tildej_b_quasi_easy_Q_p} with $u=b$. In this case, 
\[
u^*(t^{-1}_{D(u), u}) \otimes b^* (t_{D(u), b}) = b^*(t^{-1}_{D(b), b}) \otimes b^* (t_{D(b)}, b)
\]
is equal to $1$ in the canonically trivial $\Gm$-torsor
$
\Mx_{0,D(b)} = (b-b)^*\calO_X(D(b)) = \Gm .
$
The same holds, up to sign, for the integral point $\jbtilde(b)$, since there are only two integral points in $\Mx_{0,D(b)}$ and the canonical point $1$ is one of them. Hence, the numerical constant in \eqref{eq:tildej_b_quasi_easy_Q_p} is $\pm 1$. Observing that $t_{T_i \cap (X \times u), u}(u) = t_{T_i \cap (X \times u), b}(b)$ since $u=b$, the result follows.

For the second part of the statement, it is enough to track how the map $\jbtilde$ changes when we change the simple open. 
If we fix $b$ and $f$ and vary $U$, by \cite[Proposition 7.8]{Edil}, the map $\tilde j_{b,U}$,  restricted to $X_{\Z[1/N]}$, changes by a constant in $\Q^\times$.  In particular, it is enough to prove  
\begin{equation}\label{eq:technical_C_U}
\widetilde j_{b,U} = \pm \left( \prod_{q\mid N} q^{-V_{U,q} + V_{U_b,q}}  \right) \cdot \widetilde j_{b,U_b} \quad \text{ over } X_{\Z[1/N]}.
\end{equation}
For this, we describe $\tilde j_{b,U}$ in terms of a line bundle $\calL = \calL_{U}  = \calO_{X\times U}(\calD_U)$ (see  \eqref{eq:L_is_divisorial}) and trivializing sections $s_b = s_{b,U}$ and $s_\Delta = s_{\Delta,U}$.
The divisor $\calD_U$ of \eqref{eq:L_is_divisorial} depends on $U$ only through $B \times U$; in particular, its restriction to $\Z[1/N]$ is independent of $U$. Define $B_U = B|_{U}$.

As a consequence, the restrictions to $X_{\Z[1/N]}$ of $s_{b,U_b}$ and $s_{b,U}$ (respectively of $s_{\Delta,U_b}$ and $s_{\Delta,U}$) are sections of the same line bundle.
Recall from \eqref{eq:s_b_as_function} that we have $s_{\Delta,U} = \psi_{\Delta,U}^{-1}(\sigma_{\Delta,U})$ and $s_{b,U} = \psi_{b,U}^{-1}(\sigma_{b,U})$.
The map $\psi_\Delta = \psi_{\Delta,U}$ is independent of $U$ over $\Z[1/N]$, as clear from its definition in \eqref{eq:psi_Delta}. 

We show that the map $\psi_{b, U}$ also does not vary with $U$. This is not immediately clear, since the map depends on a local equation $t_{\ol D_V +B_U, b}^{\Z}$ for $\ol D_V + B_U$ in a neighborhood of $b \in X(\Z)$. However, by construction $B_U$ is disjoint from $b$ (see the discussion surrounding \eqref{eq:L_is_divisorial}), hence we can suppose that $t^{\mathbb{Z}}_{\ol D_V + B_U,b} = t^{\mathbb{Z}}_{\ol D_V,b}$ for a suitable local equation $t^{\mathbb{Z}}_{\ol D_V, b}$ for $\ol D_V$ around $b$, which does not depend on $U$. Thus $\psi_{b, U} = \psi_b$ does not depend on $U$. Finally, from the formulas  in \eqref{eq:sigma_Delta_sigma_b} for $\sigma_{b,U}$ and $\sigma_{\Delta,U}$, we see that $\sigma_{b,U} = \sigma_{b,U_b}=  1$ and that $\sigma_{\Delta,U}= \left( \prod_{q\mid N} q^{-V_{U,q} + V_{U_b,q}} \right) \cdot  \sigma_{\Delta,U_b}$. All together, we have
\[
s_{b,U} = s_{b,U_b} , \quad  s_{\Delta,U} = \left( \prod_{q\mid N} q^{-V_{U,q} + V_{U_b,q}} \right) \cdot s_{\Delta,U_b} .
\]
Substituting in the  \eqref{eq:tildej_b_abstract} for $\tilde j_{b,U}$, we obtain \eqref{eq:technical_C_U}.
\end{proof}

\begin{remark}
   Since $\Z^\times=\{\pm 1\}$, there are exactly two ways to lift $(\id_J,h) \circ j_b : U \rightarrow J \rightarrow J \times J^0$ to $\jbtilde : U \rightarrow \Mx$, which differ by a sign. When $U=U_b$ is the simple open that contains $b$, we pick the lift such that $\jbtilde(b)$ corresponds to $+1$ under the canonical trivialization $\Mx_{0,0} \cong \Gm$. This amounts to choosing the plus sign in \eqref{eq:NOT_computing_sigma_Delta}. Then, for any other simple open $U$, we fix $\jbtilde$ by picking the plus sign in       \eqref{eq:NOT_computing_sigma_Delta_2}. This amounts to asking that, when passing from $U_b$ to $U$, the map $\jbtilde$ changes by a \emph{positive} element of $\Q^\times$.
\end{remark}

\begin{remark}\label{rmk:multiplicative_local_eqs_almost}
    For Proposition \ref{prop:intermediate__formula_j_btilde} and Lemma \ref{lem:practical_formula_j_btilde} it is enough that the multiplicativity condition in Definition \ref{def:multiplicative_local_eqs}  holds for $\mathcal{B}=\{u, b\}$ and $\cal{D}$ the group generated by 
    $T_{i} \cap (X  \times u )$, $T_{i} \cap (b  \times U )$ and $\Delta^*T_{i}$.        
    
    In our computations, we work over $\Z_q$ for $q$ big enough  so that all the divisors $T_i \cap (b \times U)$, $T_i \cap (X \times u)$, and $\Delta^* T_i$ are sums $\sum_j n_j Q_j$ of $\Z_q$-points. Then,
    we choose local equations of the form  $t_{\sum n_j Q_j, u} = t_{\sum n_j Q_j, b} = \prod_j (r-r(Q_j))^{n_j}$,
    where $r$ is a function on $X$ defined as a ratio $r = f_1/f_2$ of sections of $\LMak$ (so modular forms of weight 2 in our case, see Section \ref{sec:mumford_makdisi}) and which satisfies the following properties:
    \begin{itemize}
        \item $r$ is defined at $u$ (in the sense that $f_2(u) \not \equiv 0 \bmod p$) and such that $r-r(u)$ is a local parameter at $u$, and similarly at $b$ (see Remark \ref{rmk: local parameter} below);
        \item Whenever $x \in \{b,u\}$ and $y$ is a point in 
        \[\bigcup_i \big(T_i \cap (b \times U)\big) \cup \big(T_i \cap (X \times u)\big) \cup \Delta^* T_i, \]
        $r$ is defined at $y$ (in the sense that $f_2(y) \not \equiv 0 \bmod p$) and we have that $r(y) \not \equiv r(x) \bmod p$ unless $y \equiv x \bmod p$.
    \end{itemize}
    Note that for fixed $r$, these properties only depend on the reduction mod $p$ of $u$ and of $b$; since $b$ is fixed, it is enough to find such an $r$ for each residue disc of $U$. In practice, finding such $r$ is easy, by picking random sections $f_1$ and $f_2$ defined over $\Z_q$ (as opposed to limiting ourselves to sections defined over $\Z_p$) and checking if all the conditions above are satisfied.
\end{remark}

\begin{remark}\label{rmk: local parameter}
    Fix a point $x \in U(\Z_p)$ and choose a point $x' \in U(\Z_p)$ such that $x'\equiv x \pmod{p}$ and $x'\not\equiv x \pmod{p^2}$. A function $r$ vanishing at $x$ is a local parameter at $x$ if and only if $r(x) \not\equiv r(x') \pmod{p^2}$. In particular, one can find a local parameter at $x$ by computing random ratios $f_1/f_2$ of sections of the line bundle $\LMak$ from Section \ref{sec:mumford_makdisi} and using this criterion to check whether $r \coloneq \frac{f_1}{f_2} - (f_1/f_2)(x)$ is a local parameter.
\end{remark}

\subsection{Analytic coordinates on the Mumford torsor}
\label{subsec:analytic-coordinates}

To make the geometric quadratic Chabauty method explicit, we use analytic
coordinates on the Mumford bundle and on the curve.
In this section, we introduce a system of analytic coordinates for each residue disc of $\Mx(\Z_p)$.
We first discuss analytic coordinates on $J(\Z_p)$ in a residue disc.

\begin{notation}\label{notation points on J}
For the rest of the section, fix $x_1, \ldots, x_g \in J(\Z_p)_{\ol 0} = \ker(J(\Z_p) \to J(\F_p))$ whose classes in $J(\Z/p^2\Z)$ form a basis of the $\F_p$-vector space $J(\Z/p^2\Z)_0$. We also fix $y_1, \ldots, y_g$  in $J(\Z_p)_{\ol 0}$ with the same property.
\end{notation}

\begin{remark}
Since $p \geq 3$, the condition that $x_1, \ldots, x_g$ form a basis of $J(\Z/p^2\Z)_0$ implies that $J(\Z_p)_0$ is free over $\Z_p$ with basis $x_1, \ldots, x_g$ by \cite[Lemma 5.1.1]{Edil}. Also note that $J(\Z_p)_0 = J^0(\Z_p)_0$ since $p$ is a prime of good reduction.
\end{remark}

\begin{lemma}\label{lemma:analytic-coordinates-J}     Let $x \in J(\Z_p)$ with reduction $\overline{x} \in J(\F_p)$.
    The map
    \[
    \begin{array}{ccc}
    \Z^g & \to & J(\Z_p)_{\ol x} \\
    (\lambda_i) & \mapsto & \displaystyle x + \sum_{i=1}^g \lambda_ix_i
    \end{array}
    \]
    extends by continuity to a bijection $\Z_p^g \to J(\Z_p)_{\ol x}$. The inverse of this bijection is a system of analytic coordinates (in the sense of Section~\ref{sec:parameters}) on the residue disc $J(\Z_p)_{\overline{x}}$.
\end{lemma}
\begin{proof}
    Translation by $-x$ identifies the residue disc $J(\Z_p)_{\overline{x}}$ with $J(\Z_p)_{\overline 0}$, so it suffices to prove the claim with $\ol x=\ol 0$. This follows from \cite[Lemma 5.1.1]{Edil}.
\end{proof}

For a given point $\ol t \in \Mx(\F_p)$, we now show how to parameterize $\Mx(\Z_p)_{\ol t}$.
We fix a lift  $t \in \Mx(\Z_p)$ of $\ol t$ and we denote by $(x,y) \in J(\Z_p) \times J^0(\Z_p)$ the projection of $t$.  
Fix lifts 
\begin{align*}
& P_{i, j} \in \Mx(\Z_p)&\text{ of }& (x_i, y_j) \in J(\Z_p) \times J^0(\Z_p), \\ 
&R_{i, y} \in \Mx(\Z_p) &\text{ of }& (x_i, y) \in J(\Z_p) \times J^0(\Z_p),\\ 
&S_{x,j} \in \Mx(\Z_p)&\text{ of }& (x, y_j) \in J(\Z_p) \times J^0(\Z_p). 
\end{align*}
Let $P$ be the $g\times g$ matrix with entries $P_{i,j}$, let $R$ be the column
vector with entries $R_{i,y}$, and let $S$ be the row vector with entries $S_{x,j}$.
For $\lambda, \mu \in \Z^g$, using the notation from Section \ref{sec: bilinear combination} we let
\begin{equation}\label{eq: A B C D}
\begin{aligned}
A(\mu) & = S \mu, \quad B(\lambda) = {}^\top \lambda R, \quad C(\lambda, \mu) = \BLC{\lambda}{P}{\mu} \\
D(\lambda,\mu) & = \left(C(\lambda,\mu) \Rtimes B(\lambda)\right) \LTimes \left(A(\mu) \Rtimes t \right)\,.
\end{aligned}
\end{equation}

For $a\in J$, we write $1_a$ for the neutral section of the group
$ (\Mx|_{\{a\}\times J^0}, \Rtimes )$. Similarly, for $a\in J^0$ we write $1_a$ for the neutral section of the group $(\Mx|_{J\times\{a\}}, \Ltimes)$. Moreover, we denote by $1$ the neutral section of $\Mx|_{J \times\{0\}}$ with respect to $\Ltimes$, or equivalently of $\Mx|_{\{0\}\times J^0}$  with respect to $\Rtimes$.

\begin{remark}
    Up to multiplying $P_{i,j}$ by a suitable $(p-1)$th root of unity, we can suppose that $P_{i, j}$ reduces to $1 \in \Mx_{0, 0}(\F_p)$. Similarly, we can suppose that the $R_{i,y}$ reduce to $1_{\bar y}\in\Mx_{0,\bar y}(\F_p)$, and that the $S_{x,j}$ reduce to $1_{\bar x}\in\Mx_{\bar x,0}(\F_p)$.
\end{remark}

\begin{lemma}\label{lem:biextension coordinates}
Let $t$ be a point in $\Mx(\Z_p)$, projecting to $(x,y) \in J(\Z_p) \times J^0(\Z_p)$. Assume that the $P_{i,j}$ reduce to $1\in\Mx_{0,0}(\F_p)$, that the $R_{i,y}$ reduce to $1_{\bar y}\in\Mx_{0,\bar y}(\F_p)$, and that the $S_{x,j}$ reduce to $1_{\bar x}\in\Mx_{\bar x,0}(\F_p)$. The map $D$ extends to a bijection
\begin{equation}\label{eq:Psi_coords}
\begin{array}{cccc}
\Theta : & \Z_p^g \times \Z_p^g \times \Z_p & \to & \Mx(\Z_p)_{\ol t} \\
& \left( \lambda, \mu, \nu\right) & \mapsto & \wt\exp(\nu) \cdot D\left(\lambda, \mu \right).
\end{array}
\end{equation}
The inverse of $\Theta$ gives a system of analytic coordinates
\begin{equation}\label{eq: def psi}
    \Psi : \Mx(\Z_p)_{\ol t} \to \Z_p^g \times \Z_p^g \times \Z_p.
\end{equation}
\end{lemma}

\begin{proof}[Proof of Lemma \ref{lem:biextension coordinates}] Note first that, by the assumption about the reduction of the points $P_{i,j}, R_{i, y}$ and $S_{x,j}$, the $\Z_p$-points $A(\mu)$, $B(\lambda)$, and $C(\lambda,
\mu)$ of $\Mx$ reduce to the neutral elements $1_{\ol x}, 1_{\ol y}$ and $1$, respectively, for all $\lambda, \mu\in\Z^g$. Hence $D(\lambda, \mu)$ reduces to $\bar t$ for all $\lambda, \mu \in \Z^{g}$.
Arguing as in \cite[Proof of Lemma 5.2.5]{Edil}, we see that the map $D$ in \eqref{eq: A B C D} extends to a continuous function $\Z_p^g \times \Z_p^g \to \Mx(\Z_p)_{\ol t}$. For each point of the base residue disc $J(\Z_p)_{\bar x}\times J^0(\Z_p)_{\bar y}$, the corresponding fiber of $\Mx(\Z_p)_{\bar t}$ is a torsor under $1+p\Z_p=\wt\exp(\Z_p)$. This torsor structure defines the map $\Theta$ in \eqref{eq:Psi_coords}. 
We now show that $\Theta$ is a bijection and that it satisfies the other properties in the statement. To do this, we compare it with a system of analytic coordinates in which the biextension laws are the trivial ones.

We follow the notation of \cite[Chapter~2, p.~73]{guidothesis}, with the following
minor changes: in \emph{op.~cit.}, the point $t$ is denoted $\wt t$; we identify the Poincar\'e torsor $P^\times$ of \cite{guidothesis} with the Mumford torsor $\Mx$; and we specialize to the case $\rho-1=1$, so that the $\G_m^{\rho-1}$-bundle $P^{\times, \rho-1}$ in \cite{guidothesis} becomes our $\Mx$.
We have the following maps.
\begin{itemize}
    \item Since $p>2$, by \cite[Proposition~2.3.5]{guidothesis} there is an isomorphism of biextensions
\newcommand{\triv}{\mathrm{triv}}
 \begin{equation*}
        \Psaiguido \colon (\Mx(\Z_p)_{\ol 1}, \Ltimes,\Rtimes) \lto (\Z_p^g \times \Z_p^{g} \times \Z_p, +_1, +_2) \,,
    \end{equation*}
which is also a system of analytic coordinates. Here, the biextension on the right is the trivial one, namely, it is given by the partial group laws
\[
\begin{aligned}
(a_1,b,c_1)+_1(a_2,b,c_2)=(a_1+a_2,b,c_1+c_2)\phantom{.} \\
(a,b_1,c_1)+_2(a,b_2,c_2)=(a,b_1+b_2,c_1+c_2).
\end{aligned}
\]
    \item 
    The formal logarithms give group isomorphisms
    \[
    \begin{aligned}
    \Psi_1  \colon \left( \Mx|_{\jac \times \{y\}}(\Z_p)_{ {1_{\ol y}}}, \Ltimes \right) \lto (\Z_p^{g} \times \Z_p, +) \\
    \Psi_2 \colon \left( \Mx|_{ \{x\} \times \jac}(\Z_p)_{{1_{\ol x}}}, \Rtimes \right) \lto (\Z_p^{g} \times \Z_p, +)
    \end{aligned}
    \]
    which are also systems of analytic coordinates.
    \item The previous maps combine to give a bijection (see \cite[p.~74]{guidothesis})
    \[
\begin{aligned}
\Phi \colon \Z_p^g \times \Z_p^g \times \Z_p \! \lto & \;\Mx(\Z_p)_{\ol t} \\
(\alpha, \beta,\gamma) \longmapsto & \left(\Psaiguido^{-1}(\alpha, \beta,\gamma) \Rtimes \Psi^{-1}_1 (\alpha,0) \right) \Ltimes \left( \Psi^{-1}_2(\beta,0)\Rtimes t \right) \\
= & \exp(p\gamma) {\cdot} \left(\left(\Psaiguido^{-1}(\alpha,\beta,0) \Rtimes \Psi^{-1}_1 (\alpha,0) \right) \Ltimes \left( \Psi^{-1}_2(\beta,0) \Rtimes t \right)\right) \!,
\end{aligned}
\]
whose inverse is a system of analytic coordinates at $\ol t$. 
\end{itemize}
Write $\Psaiguido(P_{i,j})=(\alpha_i, \beta_j, \gamma_{i,j})$, $\Psi_1(R_{i,y})=(\alpha_i, \gamma_{i,y})$ and $\Psi_2(S_{x,j})=(\beta_j, \gamma_{x,j})$, and set
\[
\rho(\lambda, \mu) \coloneq \sum_{i,j}\lambda_i\mu_j\gamma_{i,j} + \sum_i\lambda_i\gamma_{i,y} + \sum_j\mu_j\gamma_{x,j}.
\]
We note that for $\xi \in \Gm(\Z_p)_{\ol 1}$ and $z\in \Mx(\Z_p)_{\ol 1}$ we have $\Psaiguido( \xi\cdot z) =  \Psaiguido(z) + (0,0, \tfrac 1p \log(\xi))$, and similarly $\Psi_i(\xi \cdot z) = \Psi_i(z) + (0, \tfrac 1p \log(\xi))$ for $i=1$ and $z \in \Mx|_{J \times \{y\}}(\Z_p)_{1_{\ol y}}$ (resp.~for $i=2$ and $z \in \Mx|_{\{x\} \times J}(\Z_p)_{1_{\ol x}}$). Using this and the fact that the (partial) group laws become additive in the coordinates $\Psaiguido$, $\Psi_1$, and $\Psi_2$, we easily obtain, for $\lambda, \mu\in\Z_p^g$,
\begin{equation*}
\begin{gathered}
    \Psi_1(B(\lambda)) = \left( \sum_i \lambda_i \alpha_i, \sum_i \lambda_i \gamma_{i,y}\right) \ , \quad \Psi_2(A( \mu)) = \left( \sum_j \mu_j \beta_j, \sum_j \mu_j \gamma_{x,j}\right)\\
    \Psaiguido(C(\lambda, \mu)) = \left( \sum_i \lambda_i \alpha_i, \sum_{j} \mu_j \beta_j, \sum_{i,j} \lambda_i \mu_j \gamma_{i,j} \right), \quad D(\lambda, \mu)= \Phi\left(\sum_i\lambda_i\alpha_i, \sum_j\mu_j\beta_j,  \rho(\lambda, \mu)\right).
\end{gathered}
\end{equation*}
The last of these equalities implies
\begin{equation}\label{eq:Phi-inverse-Theta}
    \Phi^{-1}\circ\Theta(\lambda, \mu, \nu)= \left(\sum_i\lambda_i\alpha_i, \sum_j\mu_j\beta_j, \nu + \rho(\lambda, \mu)\right).
\end{equation}
By the assumption on the generators $x_i, y_j$ in Notation \ref{notation points on J}, the matrices with columns $\alpha_i$ and $\beta_j$ lie in $\GL_g(\Z_p)$: indeed, the $\alpha_i, \beta_j$ are the logarithms of $x_i, y_j$ divided by $p$, and by assumption the $x_i, y_j$ form a basis of the free $\Z_p$-module $J(\Z_p)_{\ol 0}$. Hence the last displayed map is an analytic bijection of $\Z_p^{2g+1}$, with analytic inverse. Since $\Phi$ is a system of analytic coordinates on $\Mx(\Z_p)_{\bar t}$,
this proves that $\Theta$ is a bijection and that $\Psi=\Theta^{-1}$ is a system of analytic coordinates.
\end{proof}

\begin{prop}\label{prop:Psi-algebraic-coordinates-quadratic}
    Keep the notation of Lemma \ref{lem:biextension coordinates}. If $p, w_1, \ldots, w_{2g+1}$ is an algebraic system of parameters at $\ol t$, inducing analytic coordinates $\wt w = (\wt w_1, \ldots, \wt w_{2g+1})\colon  \Mx(\Z_p)_{\overline{t}} \to \Z_p^{2g+1}$ as in \eqref{eq:alg_param}, then the compositions
$\Psi\circ \wt w^{-1}$ and $\wt w\circ \Psi^{-1}$ are given by integral convergent power series whose reductions modulo~$p$ are at most quadratic. 
\end{prop}
\begin{proof}
    We fix a regular system of parameters at $t$ as follows. Let  $p, \xi_1, \ldots, \xi_g$ be a regular system of parameters at $\ol x$, and similarly let $p, \eta_1, \ldots, \eta_g$ be parameters at $\ol y$. We pull back these parameters to $\Mx$ and complete them to a regular system of parameters $p, \xi_1, \ldots, \xi_g, \eta_1, \ldots, \eta_g, v$ at $\ol t$. Dividing by $p$ as usual (see \eqref{eq:alg_param}), we obtain analytic coordinates
\[
 (\tilde{\xi}, \tilde{\eta}, \tilde v) : \Mx(\Z_p)_{\ol t} \to \Z_p^{2g+1}.
\]

By \cite[pp. 74-75]{guidothesis}, the power series $(\tilde{\xi}, \tilde{\eta}, \tilde  v) \circ \Phi$ and $((\tilde{\xi}, \tilde{\eta}, \tilde  v) \circ \Phi)^{-1}$ reduce modulo $p$ to linear polynomials. Taking into account \eqref{eq:Phi-inverse-Theta}, it follows that
\[
 (\tilde{\xi}, \tilde{\eta}, \tilde  v) \circ \Psi^{-1} =   (\tilde{\xi}, \tilde{\eta}, \tilde  v) \circ \Phi \circ \left(\Psi \circ \Phi\right)^{-1}
\]
reduces modulo $p$ to the composition of a linear map and an at most quadratic one. The statement that also the inverse $\Psi \circ (\tilde{\xi}, \tilde{\eta}, \tilde  v)^{-1}$ is at most quadratic is proven similarly. This shows the claim for a system of parameters  $p, \xi_1, \ldots, \xi_g, \eta_1, \ldots, \eta_g, v$ of the form described above. 
For any other system of parameters $p, w_1, \ldots, w_{2g+1}$, note that both the change of variables $(\tilde{\xi}, \tilde{\eta}, \tilde  v) \circ \tilde{w}^{-1}$ and its inverse are  given by integral convergent power series that reduce to polynomials of degree $1$ by Lemma \ref{lem:algebraic is linear}. Combined with the special case above, this implies the statement also for $p, w_1, \ldots, w_{2g+1}$.
\end{proof}

\begin{cor}\label{cor:jbtilde quadratic}
Fix a point $u \in U(\F_p)$, let $\uniform$ be a uniformizer at $u$ and denote $\tilde \uniform \coloneq\frac{\uniform}{p}$, which induces a bijection $U(\Z_p)_u \to \Z_p$ as in  Section~\ref{sec:parameters}. Let $\Psi$ be the map defined in \eqref{eq: def psi}.
The composition $\Psi \circ \jbtilde \circ \tilde{\uniform}^{-1} : \Z_p \to \Z_p^g \times \Z_p^g \times \Z_p$ is given by integral convergent power series, and its reduction modulo $p$ is given by (at most) quadratic polynomials.
\end{cor}

\begin{proof}
Let $   \ol t=\jbtilde(u) \in \Mx(\F_p)$. Consider a system of algebraic parameters $p, v_1, \ldots, v_{2g+1}$ at $\ol t$, and $(\tilde v_1, \ldots, \tilde v_{2g+1})\colon \Mx(\Z_p)_{\overline{t}} \to 
\Z_p^{2g+1}$ the corresponding coordinates. By Lemma \ref{lem:algebraic is linear},
$\tilde{v} \circ \tilde{j}_b \circ \tilde \uniform^{-1}$ is given by convergent power series whose reductions modulo $p$ are polynomials of degree at most $1$.  By Proposition \ref{prop:Psi-algebraic-coordinates-quadratic}, the composition $\Psi \circ \tilde{v}^{-1}$ is given by convergent power series whose reductions modulo $p$ are at most quadratic polynomials. Composing the two maps gives the claim.
\end{proof}

\subsection{Integral points on the Mumford torsor}\label{sec:kappa_general}
Consider again a point $\ol t \in \Mx(\F_p)$.
In this section, we assume that the projection 
$(\ol x, \ol y)$ $\in J(\F_p) \times J^0(\F_p)$ of $\ol t$ can be lifted to an integral point  $(x,y) \in J(\Z) \times J^0(\Z)$.  
In particular, there exists a point $t^\Z \in \Mx(\Z)$ lying over $(x,y)$, though not necessarily in the same residue disc as $t$ (Here we use the fact that $\operatorname{Pic}(\mathbb{Z})=0$: as the fiber of $\Mx$ above $(x,y)$ is a $\mathbb{G}_m$-torsor over $\operatorname{Spec} \mathbb{Z}$, it is trivial, and thus contains an integral point). The goal of this section is to describe $\ol{\Mx(\Z)_{\ol t}}$ inside of $\Mx(\Z_p)_{\ol t}$.

We make the further assumption that the points $x_i$ and $y_i$ from Notation \ref{notation points on J}  have the property that $x_1, \dots, x_g \in J(\Z)_{\ol{0}}$ and  $y_1, \dots, y_g \in J^0(\Z)_{\ol{0}}$.
In particular, we assume that the map 
\begin{equation}\label{eq: J0Z0 to JZp2 onto}
    J^0(\Z)_{\ol 0} \to J(\Z/p^2\Z)_{\ol 0} 
\end{equation}
is surjective. 

For each $i,j \in \{1, \ldots, g\}$, choose lifts (unique up to sign)
\begin{align*}
& P_{i, j}^{\Z} \in \Mx(\Z)&\text{ of }&& (x_i, y_j) \in J(\Z) \times J^0(\Z), \\ 
&R_{i, y}^{\Z} \in \Mx(\Z) &\text{ of }&& (x_i, y) \in J(\Z) \times J^0(\Z),\\ 
&S_{x,j}^{\Z} \in \Mx(\Z)&\text{ of }&& (x, y_j) \in J(\Z) \times J^0(\Z). 
\end{align*}

There exist unique choices of $(p-1)$th roots of unity $\xi , \xi_{i,j}, \xi_{i,y}$ and $\xi_{x,j}$ in $\Z_p^\times$ such that the points  
\begin{equation}\label{eq: choices kappa}
t \coloneq\xi  t^\Z, \quad  P_{i,j}\coloneq \xi_{i,j} \cdot P_{i,j}^\Z ,  \quad R_{i,y} \coloneq \xi_{i,y} \cdot R_{i,y}^\Z, \quad S_{x,j}\coloneq \xi_{x,j} \cdot S_{x,j}^\Z,
\end{equation}
reduce modulo~$p$ respectively to the points $\ol t$, $\ol 1$, $1_{\ol y}$, and $1_{\ol x}$ in  $\Mx(\F_p)$.

 With these data, we can now define a map $\kappa$ to parametrize $\overline{\Mx(\Z)_{\ol t}}$.

\begin{defi}\label{definition kappa}
    Given a point $\ol t \in \Mx(\F_p)$ whose projection  $(\ol x, \ol y) \in J(\F_p) \times J(\F_p)$ can be lifted to an integral point on $J \times J^0$, we denote by
    \begin{equation}\label{eq: kappa new}
        \kappa \colon \Z_p^{2g} \to \Mx(\Z_p)_{\ol t}
    \end{equation}    
    the continuous extension of the map $D \colon \Z^{2g} \to \Mx(\Z_p)_{\ol t}$ defined in \eqref{eq: A B C D}, with the hypothesis that the points $t,P_{i,j}, R_{i,y}$ and $S_{x,j}$ in the definition of $D$ are chosen as in  \eqref{eq: choices kappa}.
\end{defi}
Note that, due to the presence of roots of unity in \eqref{eq: choices kappa}, the points in $\kappa(\Z^{2g})$ are not necessarily integral points. We can also define variants of the maps in \eqref{eq: A B C D}, which have the good property that their image is contained in $\Mx(\Z)$, at the cost that their image might not be contained in a single residue disc. In particular, for $\lambda, \mu \in \Z^g$ we set

\begin{equation}\label{eq: AZ BZ CZ DZ}
\begin{aligned}
A^\Z(\mu) & = S^\Z \mu, \quad B^\Z(\lambda) = {}^\top \lambda R^\Z, \quad C^\Z(\lambda, \mu) = \BLC{\lambda}{P^\Z}{\mu} \\
D^\Z(\lambda,\mu) & = \left(C^\Z(\lambda,\mu) \Rtimes B^\Z(\lambda)\right) \LTimes \left(A^\Z(\mu) \Rtimes t^\Z \right) ,
\end{aligned}    
\end{equation}
where $P^\Z$ is the $g\times g$ matrix whose entries are the $P_{i,j}^\Z$, and $R^\Z$, respectively $S^\Z$, is the column (resp.~row) vector whose entries are the $R_{i,y}^\Z$, respectively $S_{x,j}^\Z$.
These are used in the proofs of Theorem \ref{thm:Mx(Z) bar} and Lemma \ref{lem:onlycareaboutlastcoord}.

The next theorem, describing the closure of integral points in the Mumford torsor, is analogous to the computation of the map $\kappa$ in \cite{Edil}.

\begin{theorem}\label{thm:Mx(Z) bar} Assume that \eqref{eq: J0Z0 to JZp2 onto} is surjective and that $\rank  J(\Q) = g$. 
Let $\kappa$ be as in \eqref{eq: kappa new}, let $\Psi$ be the map given in Lemma \ref{lem:biextension coordinates}, and assume that the points appearing in the lemma are chosen as in \eqref{eq: choices kappa}. Then, if $\Mx(\Z)_{\ol t}$ is non-empty, we have
\[ 
\ol {\Mx(\Z)_{\ol t}} = \kappa(\Z_p^{2g}) = \Psi^{-1} \left( \Z_p^g \times \Z_p^{g} \times \{0\} \right). 
\]
\end{theorem}
\begin{proof}
    We can reduce to the case where $x_1, \ldots, x_g$ is a $\Z$-basis of $J(\Z)_{\ol 0}$ and $y_1, \ldots, y_g$ a $\Z$-basis of $J^0(\Z)_{\ol 0}$ since, otherwise, the map $\kappa$ differs by a change of variable $(\lambda, \mu) \mapsto (L \lambda, M \mu)$ for certain $g\times g$ matrices $M,L$ which are invertible over $\Z_p$ by the assumption in Notation \ref{notation points on J} about the $x_i, y_j$.

    We have $\Mx(\Z)_{\ol t} = \pm D^\Z(\Z^g \times \Z^g) \cap \Mx(\Z_p)_{\ol t}$, which is a subset of $D(\Z^g \times \Z^g)$. By continuity of $\kappa$ and the fact that it coincides with $D$ on $\Z^g \times \Z^g$ we obtain the containment
    \[
    \overline{\Mx(\Z)_{\ol t}} \subseteq \overline{D(\Z^g \times \Z^g)}=\kappa(\Z_p^g \times \Z_p^g).
    \]

    Moreover, if $t = t^\Z$, then  $\Mx(\Z)_{\ol t}$  contains  $D^\Z\left((p-1)\cdot (\Z^g \times \Z^g)\right) = D\left((p-1)\cdot (\Z^g \times \Z^g)\right)$ since for $\lambda, \mu \in \Z^g$ the points $D(\lambda, \mu)$ and $D^\Z(\lambda, \mu)$ differ by a $(p-1)$th root of unity. 
    More generally, if $\Mx(\Z)_{\ol t}$ is non-empty, then it contains a point of the form $\pm  D^\Z(\lambda_0, \mu_0) =  D(\lambda_0, \mu_0)$ for some $\lambda_0, \mu_0 \in \Z^g \times \Z^g$, and consequently contains the image, under $D$, of $(\lambda_0, \mu_0) + (p-1) \cdot \Z^{2g}$. Since $(\lambda_0, \mu_0) + (p-1) \cdot \Z^{2g}$ is dense in $\Z_p^{2g}$, we deduce that 
    \[
    \kappa(\Z_p^g \times \Z_p^g) = \ol{D\left( (\lambda_0, \mu_0) + (p-1) \cdot \Z^{2g}\right)} \subset  \overline{\Mx(\Z)_{\ol t}}. \qedhere
    \]
\end{proof}

\begin{remark}\label{rmk kappa}
    In \cite{Edil} one works with a certain $\Gm$-torsor $T$  over $J$, obtained as a suitable pullback of $\Mx$. One then describes a map $\kappa \colon \Z_p^g \to T(\Z_p)$ whose image is precisely the $p$-adic closure of $T(\Z)$ in a given residue disc. In our current notation, the map $\kappa$ in \emph{op.~cit.} coincides with the function 
    \[
    \lambda \mapsto \kappa( (p-1) \lambda, (p-1) M \lambda ),
    \]
    where $M$ is a suitable matrix and $\kappa$ is defined in \eqref{eq: kappa new}.

    Here, we work directly on $\Mx$, hence our map $\kappa$  has domain $\Z_p^{2g}$. Moreover, here we do not assume $\ol t$ lifts to a point in $\Mx(\Z)$, a condition we might not be able to check in practice, even when we know a lift $(x,y) \in J(\Z) \times J^0(\Z)$ of  $(\ol x, \ol y)$. In particular, there might be cases where we define a $\kappa$ map for the residue disc of $\ol t$, even though $\Mx(\Z)_{\ol t}$ is empty. In any case, when $\Mx(\Z)_{\ol t}$ is non-empty, Theorem \ref{thm:Mx(Z) bar} guarantees that its closure is still parameterized by~$\kappa$.  

    Finally, multiplication by $p-1$ in \cite{Edil} is replaced by our current assumption that the points $P_{i,j}, R_{x,j}, S_{i,y}$ reduce to the respective identities modulo $p$.
\end{remark}
\begin{remark}
The point of view used here is different from that in \cite{Edil}. In \cite{Edil}, the coordinates are chosen so that the curve is the image of a map that is affine modulo $p$, while the closure of the integral points is the image of an analytic map that is (at most) quadratic modulo $p$.
Here, the coordinates are chosen in the opposite way: the closure of the integral points in $\Mx(\Z_p)_{\ol t}$ is given by a linear equation in the last coordinate, while the image of $\jbtilde$ is described by a power series that is at most quadratic modulo $p$.
\end{remark}

The next lemma is the main tool that will allow us to compute with integral points of $\Mx(\Z)$ by embedding them in $\Mx(\Z_p)$ and working up to finite $p$-adic precision.
\begin{lemma}\label{lem:onlycareaboutlastcoord}
   Assume that \eqref{eq: J0Z0 to JZp2 onto} is surjective and that $\rank J(\Q) = g$. 
    Let $e$ be a positive integer. Let $x, x' \in J(\Z)$ and $y, y' \in J^0(\Z)$ be such that $x \equiv x'$ and $y \equiv y'$ modulo~$p^e$. Let $z \in \Mx(\Z)$, respectively $z' \in \Mx(\Z)$, be a lift of $(x, y)$, resp.~$(x',y')$. There exists a  $(p-1)$th root of unity $\xi \in \Z_p^\times$ such that  
    \[
    z \equiv  \xi z' \pmod{p^{e}}.
    \]  
\end{lemma}
\begin{proof}

    By Lemmas \ref{lemma:analytic-coordinates-J} and \ref{lem: coords mod pe}, the map $(\lambda, \mu) \mapsto (x+\sum \lambda_i x_i, y + \sum \mu_i y_i)$ induces a bijection 
    \begin{equation}\label{eq:F J technical}
    (\Z/p^{e-1}\Z)^g \times (\Z/p^{e-1}\Z)^g \to J(\Z/p^{e}\Z)_{\ol x} \times J(\Z/p^{e}\Z)_{\ol y}. 
    \end{equation}
    Let $x_1, \ldots, x_g$ be a basis of $J(\Z)_{\ol 0}$ and $y_1, \ldots, y_g$ a basis of $J^0(\Z)_{\ol 0}$, so that we can write  $x' = x + \sum \lambda_i' x_i$  and $y' = y + \sum \mu_i' y_i$ for suitable $\lambda', \mu' \in \Z^g$. By the injectivity of \eqref{eq:F J technical}, we know that all coordinates of $\lambda'$ and $\mu'$ are multiples of $p^{e-1}$.

    Now construct the maps $D$ and $D^\Z$ as in \eqref{eq: A B C D} and \eqref{eq: AZ BZ CZ DZ}, with $t=t^\Z=z$. 
    In particular, we have 
    \[
    z = D^\Z(0,0) = D(0,0) , \quad z'  = \epsilon D^\Z(\lambda',\mu') = \xi D(\lambda',\mu'),
    \]
    for a suitable root of unity $\xi$ and a suitable sign $\epsilon \in \{\pm1\}$. Now the statement follows from the fact that $D$ is given by integral convergent power series (indeed by polynomials with coefficients in $\Z_p$) in the analytic coordinates given by $\Psi$ (see Lemma \ref{lem:biextension coordinates}), hence, using again Lemma \ref{lem: coords mod pe}, we have $D(\lambda',\mu') \equiv D(0,0)$ modulo~$p^e$.  
\end{proof}

\begin{remark}
The $(p-1)$th root of unity ambiguity on the $\Gm$-coordinate will not matter, since we will apply the $p$-adic logarithm (cf.\ Section \ref{ref: p-adic log}) in the coordinate system $\Psi$ defined above (see also Definition \ref{defi:psi-prime}).
\end{remark}

\section{Non-split Cartan curves over the integers}
\label{sec:some_modular_theory}

Let $X\coloneq X_{\ns}^+ (N)$ be the  non-split Cartan modular curve of prime 
level $N$. Denote by $J$ the Néron model over $\Z$ of its Jacobian. In this section we prove (or recall) various properties of $J$ and of a specific regular model of $X$.

\subsection{Regular model}\label{subsec:NewRegularModel}
\label{NewRegularModel}

We start by recalling the shape of the special fiber of a specific regular model $X_{\ns}^+ (N)^{\mathrm{reg}}$ for $X_{\ns}^+ (N)$ over 
$\Z_N^{\unr}$, see Figure~\ref{figureNSgrossier+}.
An explicit description of $X_{\ns}^+ (N)^{\mathrm{reg}}$ is given in \cite[Theorem 3.8]{EdPar_regular} (where the prime level $N$ is called $p$), to which we shall refer in what follows.

{\em Note, however, that 
we will use a slightly different model here} (when $N\equiv -1\bmod 4$), which is regular with 
normal crossings but {\em not} minimal with these properties. Indeed, it follows from the proof of \cite[Theorem 3.8]{EdPar_regular} that a first (regular with normal crossings) model we obtained when $N\equiv -1\bmod 4$ possesses, above the 
supersingular points with $j\equiv 1728$, a chain of {\em two} projective lines, one with multiplicity $(N+1)/2$ (call it 
$G_0$), which intersects the central $j$-line $A$ (having multiplicity $(N-1)/2$), followed by a projective line with 
multiplicity $1$ (call it $H_0$). As the genus-$0$ component $G_0$ has self-intersection $-1$, the proof of Theorem~3.8 of~\cite{EdPar_regular} contracts it to obtain the minimality property for the SNC model described there. In the 
present work, however, we need to keep track of the specialization locus of all CM-points occurring in our divisors
$D_\ell$ (cf.~below), and some of them do reduce to $G_0$, whence the need not to contract it in our present model.  See Figure~\ref{figureNSgrossier+} here and compare it with Figure~6 of~\cite{EdPar_regular}.

\begin{figure}[ht]
\centering
	\begin{picture}(320,471)(-30,-90)

		\qbezier(-50,280)(5,290)(45,278)
		\qbezier(-50,281)(5,291)(45,279)
		\qbezier(-50,282)(5,292)(45,280)

		\put(-69,275){\scriptsize$D_2$}
		\put(52,275){\scriptsize$N+1$}

		\qbezier(-32,293)(-29,283)(-31,270)
		\qbezier(-31,293)(-28,283)(-30,270)
		\put(-46,294){\scriptsize$\frac{N+1}{2}$}
		\put(-28,272){\scriptsize$E_2$}

		\qbezier(20,301)(18,283)(21,272)
		\put(21,302){\scriptsize$1$}
		\put(24,269){\scriptsize$F_2$}

		\qbezier(-49,306)(5,321)(46,309)
		\qbezier(-49,307)(5,322)(46,310)
		\qbezier(-49,308)(5,323)(46,311)

		\put(-69,305){\scriptsize$D_1$}
		\put(52,305){\scriptsize$N+1$}

		\qbezier(-27,320)(-24,308)(-26,295)
		\qbezier(-26,320)(-23,308)(-25,295)
		\put(-29,323){\scriptsize$\frac{N+1}{2}$}
		\put(-22,300){\scriptsize$E_1$}

		\qbezier(30,322)(28,316)(31,295)
		\put(30,323){\scriptsize$1$}
		\put(34,296){\scriptsize$F_1$}

		\qbezier(-48,217)(5,233)(59,218)
		\qbezier(-48,216)(5,232)(59,217)
		\qbezier(-48,215)(5,231)(59,216)

		\put(-63,215){\scriptsize$D_k$}
		\put(64,217){\scriptsize$N+1$}

		\qbezier(-25,245)(-23,224)(-26,206)
		\qbezier(-24,245)(-22,224)(-25,206)
		\put(-27,248){\scriptsize$\frac{N+1}{2}$}
		\put(-22,208){\scriptsize$E_k$}

		\qbezier(27,245)(24,222)(26,212)
		\put(28,250){\scriptsize$1$}
		\put(28,208){\scriptsize$F_k$}


		\qbezier(0,373)(-5,270)(-2,200) 
		\qbezier(1,373)(-4,270)(-1,200)
		\qbezier(-1,373)(-6,270)(-3,200)

		\put(-9,382){\scriptsize$\frac{N-1}{2}$}
		\put(4,370){\scriptsize$A$}


		\put(20,340){\small{$j\equiv 0$}}
		\put(17,342){\vector(-1,1){13}}
		\qbezier(-25,358)(5,357)(29,359)
		\qbezier(-25,357)(5,356)(29,358)
		\put(34,353){\scriptsize$\frac{N-1}{3}$}
		\put(-40,356){\scriptsize${\cal D}_0$}

		\qbezier(-17,370)(-14,368)(-16,345)
		\qbezier(-16,370)(-13,368)(-15,345)
		\put(-32,373){\scriptsize$\frac{N-1}{6}$}
		\put(-29,344){\scriptsize${\cal E}_0$}


		\put(-45,+270){$\cdot$}
		\put(-45,253){$\cdot$}
		\put(-45,234){$\cdot$}

		\put(42,266){$\cdot$}
		\put(43,247){$\cdot$}
		\put(44,230){$\cdot$}

		
		\put(-20,175){$N=12k+1$}







		\qbezier(188,338)(245,341)(290,333)
		\qbezier(188,337)(245,340)(290,332)
		\put(164,334){\scriptsize$\frac{N+1}{3}$}
		\put(297,332){\scriptsize$D_0$}

		\qbezier(210,352)(212,343)(209,328)
		\qbezier(211,352)(213,343)(210,328)
		\put(196,355){\scriptsize$\frac{N+1}{6}$}
		\put(215,350){\scriptsize$E_0$}
		\qbezier(256,353)(252,342)(253,327)
		\put(250,352){\scriptsize$1$}
		\put(257,349){\scriptsize$F_0$}

		\put(214,318){\scriptsize{$j\equiv 0$}}
		\put(226,325){\vector(1,1){10}}
		

		\qbezier(182,306)(245,316)(295,301)
		\qbezier(182,305)(245,315)(295,300)
		\qbezier(182,304)(245,314)(295,299)
		\put(305,297){\scriptsize$D_1$}
		\put(157,302){\scriptsize$N+1$}

		\qbezier(180,279)(245,286)(296,270)
		\qbezier(180,278)(245,285)(295,269)
		\qbezier(180,277)(245,284)(295,268)
		\put(305,268){\scriptsize$D_2$}
		\put(157,274){\scriptsize$N+1$}

		\qbezier(188,226)(245,236)(295,226)
		\qbezier(188,225)(245,235)(295,225)
		\qbezier(188,224)(245,234)(295,224)
		\put(305,220){\scriptsize$D_k$}
		\put(164,221){\scriptsize$N+1$}

		\put(221,271){$\cdot$}
		\put(222,253){$\cdot$}
		\put(224,234){$\cdot$}


		\qbezier(216,264)(216,273)(212,292)
		\qbezier(217,264)(217,273)(213,292)
		\put(208,294){\scriptsize$\frac{N+1}{2}$}
		\put(201,287){\scriptsize$E_2$}

		\qbezier(200,293)(201,305)(196,322)
		\qbezier(199,293)(200,305)(195,322) 
		\put(191,322){\scriptsize$\frac{N+1}{2}$}
		\put(185,315){\scriptsize$E_1$}
		
		\qbezier(208,207)(211,225)(208,248)
		\qbezier(207,207)(210,225)(207,248)
		\put(203,251){\scriptsize$\frac{N+1}{2}$}
		\put(195,242){\scriptsize$E_k$}


		\qbezier(268,267)(262,275)(263,292)
		\put(259,294){\scriptsize$1$}
		\put(264,286){\scriptsize$F_2$}
		
		\qbezier(272,293)(266,305)(268,322)
		\put(263,322){\scriptsize$1$}
		\put(270,319){\scriptsize$F_1$}
		
		\qbezier(275,207)(263,225)(268,248)
		\put(264,249){\scriptsize$1$}
		\put(269,244){\scriptsize$F_k$}
		

		\put(232,375){\scriptsize$\frac{N-1}{2}$}
		\put(244,365){\scriptsize$A$}
		
		\qbezier(241,370)(240,270)(234,200)
		\qbezier(240,370)(239,270)(233,200)
		\qbezier(239,370)(238,270)(232,200)
		
		%
		

		\put(215,175){$N=12k+5$}







		\qbezier(-50,10)(5,20)(45,8)
		\qbezier(-50,11)(5,21)(45,9)
		\qbezier(-50,12)(5,22)(45,10)

		\put(-73,5){\scriptsize$D_2$}
		\put(54,5){\scriptsize$N+1$}

		\qbezier(-35,23)(-32,13)(-34,0)
		\qbezier(-36,23)(-33,13)(-35,0)
		\put(-49,23){\scriptsize$\frac{N+1}{2}$}
		\put(-31,0){\scriptsize$E_2$}
		
		\put(24,-4){\scriptsize$F_2$}

		\qbezier(20,31)(18,13)(21,2)
		\put(21,32){\scriptsize$1$}

		\qbezier(-49,36)(5,51)(46,39)
		\qbezier(-49,37)(5,52)(46,40)
		\qbezier(-49,38)(5,53)(46,41)

		\put(-73,35){\scriptsize$D_1$}
		\put(54,35){\scriptsize$N+1$}

		\qbezier(-27,50)(-24,38)(-26,25)
		\qbezier(-26,50)(-23,38)(-25,25)
		\put(-29,53){\scriptsize$\frac{N+1}{2}$}	
		\put(-21,28){\scriptsize$E_1$}
		
		\put(34,26){\scriptsize$F_1$}

		\qbezier(30,52)(28,46)(31,25)
		\put(30,53){\scriptsize$1$}

		\qbezier(-48,-47)(5,-37)(59,-52)
		\qbezier(-48,-46)(5,-36)(59,-51)
		\qbezier(-48,-45)(5,-35)(59,-50)

		\put(-68,-52){\scriptsize$D_k$}
		\put(65,-53){\scriptsize$N+1$}

		\qbezier(-24,-25)(-22,-46)(-25,-57)
		\qbezier(-25,-25)(-23,-46)(-26,-57)
		\put(-26,-23){\scriptsize$\frac{N+1}{2}$}
		\put(-21,-57){\scriptsize$E_k$}
		
		\put(30,-59){\scriptsize$F_k$}

		\qbezier(27,-25)(24,-48)(26,-58)
		\put(28,-20){\scriptsize$1$}


		\qbezier(0,123)(-5,0)(15,-70) 
		\qbezier(1,123)(-4,0)(16,-70)
		\qbezier(-1,123)(-6,0)(14,-70)

		\put(-9,131){\scriptsize$\frac{N-1}{2}$}
		\put(5,118){\scriptsize$A$}

		
		\put(16,104){\small{$j\equiv 0$}}
		\put(15,105){\vector(-2,-1){12}}
		\qbezier(-28,98)(5,96)(31,99)
		\qbezier(-28,97)(5,95)(31,98)
		\put(36,92){\scriptsize$\frac{N-1}{3}$}	
		\put(-45,97){\scriptsize${\cal D}_0$}

		\qbezier(-17,105)(-14,93)(-16,80)
		\qbezier(-16,105)(-13,93)(-15,80)
		\put(-32,113){\scriptsize$\frac{N-1}{6}$}
		\put(-28,84){\scriptsize${\cal E}_0$}


		\qbezier(-40,71)(0,79)(40,75)
		\qbezier(-40,70)(0,78)(40,74)
		\put(43,74){\scriptsize$\frac{N+1}{2}$}
		
		\put(20,60){\small{$j\equiv 1728$}}
		\put(16,63){\vector(-3,2){12}}
		\put(-57,66){\scriptsize$G_0$}

		\put(10,85){\scriptsize$1$}
		\qbezier(17,90)(15,86)(18,67)
		\put(19,86){\scriptsize$H_0$}


		\put(-45,-0){$\cdot$}
		\put(-45,-17){$\cdot$}
		\put(-45,-36){$\cdot$}

		\put(42,-4){$\cdot$}
		\put(43,-23){$\cdot$}
		\put(44,-40){$\cdot$}

		
		\put(-20,-95){$N= 12k+7$}



		
		\qbezier(190,99)(245,105)(289,98)
		\qbezier(190,98)(245,104)(289,97)
		\put(170,97){\scriptsize$\frac{N+1}{2}$}
		
		\put(260,119){\small{$j\equiv 1728$}}
		\put(258,120){\vector(-1,-1){15}}
		\put(292,94){\scriptsize$G_0$}

		\qbezier(280,108)(276,97)(276,82)
		\put(282,108){\scriptsize$1$}
		\put(280,80){\scriptsize$H_0$}


		\qbezier(188,68)(245,71)(290,63)
		\qbezier(188,67)(245,70)(290,62)
		
		\put(164,64){\scriptsize$\frac{N+1}{3}$}
		
		\put(213,48){\scriptsize{$j\equiv 0$}}
		\put(225,54){\vector(1,1){10}}
		
		\qbezier(210,82)(212,73)(209,58)
		\qbezier(211,82)(213,73)(210,58)
		\put(208,85){\scriptsize$\frac{N+1}{6}$}
		
		\put(198,75){\scriptsize$E_0$}
		
		\put(296,61){\scriptsize$D_0$}
		
		\qbezier(256,83)(252,72)(253,57)
		\put(259,84){\scriptsize$1$}
		\put(255,55){\scriptsize$F_0$}


		\qbezier(182,36)(245,46)(295,31)
		\qbezier(182,35)(245,45)(295,30)
		\qbezier(182,34)(245,44)(295,29)
		\put(305,27){\scriptsize$D_1$}
		\put(157,32){\scriptsize$N+1$}

		\qbezier(180,9)(245,16)(296,0)
		\qbezier(180,8)(245,15)(295,-1)
		\qbezier(180,7)(245,14)(295,-2)
		\put(305,-2){\scriptsize$D_2$}
		\put(156,4){\scriptsize$N+1$}

		\qbezier(188,-44)(245,-34)(295,-44)
		\qbezier(188,-45)(245,-35)(295,-45)
		\qbezier(188,-46)(245,-36)(295,-46)
		\put(305,-50){\scriptsize$D_k$}
		\put(164,-49){\scriptsize$N+1$}

		\put(221,1){$\cdot$}
		\put(222,-17){$\cdot$}
		\put(224,-36){$\cdot$}


		\qbezier(216,-6)(216,13)(212,22)
		\qbezier(215,-6)(215,13)(211,22)
		\put(207,24){\scriptsize$\frac{N+1}{2}$}
		\put(202,-5){\scriptsize$E_2$}
		\put(271,-8){\scriptsize$F_2$}

		\qbezier(202,23)(203,35)(198,52)
		\qbezier(201,23)(202,35)(197,52)
		\put(184,52){\scriptsize$\frac{N+1}{2}$}
		\put(188,24){\scriptsize$E_1$}
		\put(275,20){\scriptsize$F_1$}

		\qbezier(208,-63)(211,-45)(208,-22)
		\qbezier(209,-63)(212,-45)(209,-22)
		\put(205,-21){\scriptsize$\frac{N+1}{2}$}
		\put(212,-65){\scriptsize$E_k$}
		\put(279,-68){\scriptsize$F_k$}


		\qbezier(268,-3)(262,5)(263,22)
		\put(257,22){\scriptsize$1$}
		
		\qbezier(272,23)(266,35)(268,52)
		\put(271,50){\scriptsize$1$}
		
		\qbezier(275,-63)(263,-45)(268,-22)
		\put(271,-22){\scriptsize$1$}


		\put(232,132){\scriptsize$\frac{N-1}{2}$}
		\put(225,117){\scriptsize$A$}

		\qbezier(241,122)(236,0)(246,-70)
		\qbezier(240,122)(235,0)(245,-70)
		\qbezier(239,123)(234,0)(244,-70)
		
		%
		

		\put(215,-95){$N=12k+11$}


	\end{picture}
\caption{Special fiber above $\overline{\F}_N$ of our regular model $X^+_{\ns} (N)^{\mathrm{reg}}$, cf.~Section \ref{subsec:NewRegularModel}}   
\label{figureNSgrossier+}
\end{figure}
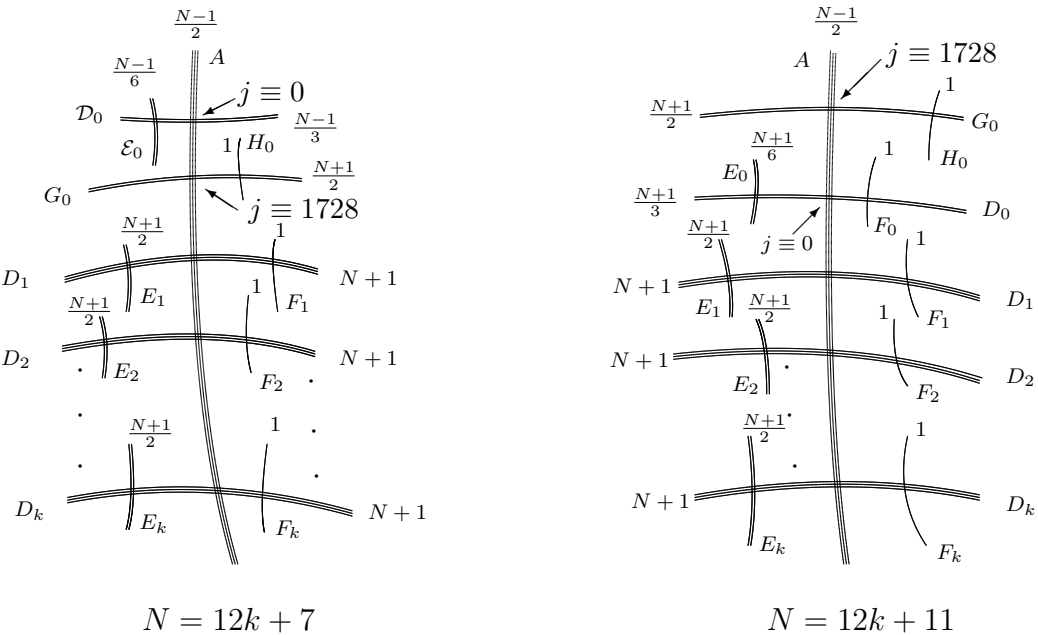

By \cite[Proposition 3.2]{EdPar_regular}, the exponent of the component group of $J_{\ol{\F}_N}$ for $N\ge 17$ is
\begin{equation}
    m \coloneq \exp\left((J_{\F_N}/J^0_{\F_N})(\ol{\F}_N)\right)  = N-1.
\end{equation}
The other cases of positive genus are $N=11$ and $13$, for which we have $m=2$, resp.~$m=1$.

\begin{remark}\label{rmk: intersection matrix}
Figure~\ref{figureNSgrossier+}, together with the multiplicities indicated there, determines the intersection matrix of the irreducible components of the special fiber of $X_{\ns}^+(N)^{\rm reg}$ at $N$. Indeed, the figure shows which pairs of distinct components meet. Since the model is regular with normal crossings, two distinct reduced components have intersection number equal to the number of their intersection points (which, for our model, is either $0$ or $1$ for each pair of components).

The self-intersections are then determined by the fact that every component has
intersection number $0$ with the full special fiber. More precisely, if $X_{\F_N}=\sum_C m_C C$, then, for every component $C$, we have $0=[C,X_{\F_N}]
=m_C[C,C]+\sum_{C'\neq C}m_{C'}[C,C']$ and therefore $[C,C]=-\frac{1}{m_C}\sum_{C'\neq C}m_{C'}[C,C']$. 

For example, the component $F_i$ has multiplicity $1$ and meets only $D_i$, transversely in one point, while $D_i$ has multiplicity $N+1$. Hence $0=[F_i,X_{\F_N}]=[F_i,F_i]+(N+1)[F_i,D_i]$, and therefore $[F_i,F_i]=-(N+1)$.
\end{remark}

\subsection{Galois orbits and specialization of CM points}

The main objects of this section are divisors of CM points over $\overline{\Q}$ or $N$-adic fields, Galois actions on them, and the way they spread out and specialize over the ring of integers
$\Z_N^{\unr}$ of the maximal unramified extension $\Q_N^{\unr}$ of $\Q_N$.

\begin{prop}
\label{GaloisOrbitsSS}
Let $E$ be an elliptic curve over $\Z_N^{\unr}$, having supersingular good reduction at $N$. Recall that points in $X_\ns^+(N) (\Q_N^{\unr} )$ corresponding to $E$ are in bijection with the possible classes of normalizers of non-split Cartan subgroups
\[
\mathcal C_{\ns}^+(N,E)
\coloneq\frac{\left\{ C^+\subset \Aut(E[N]) : C^+ \text{ is the normalizer of a non-split Cartan subgroup} \right\}
}{\Aut(E_{\overline{\Q}_N})/\{\pm 1\}}.
\]
The following holds. 
\begin{itemize}
\item If $j(E)\neq 0, 1728$, the set ${\cal C}_{\ns}^+ (N,E)$ has cardinality $\frac{N^2-N}{2}$. The action of $\Gal (\overline{\Q}_N /\Q^{\unr}_N )$ on this set has $(N-3)/2$ orbits of size $(N+1)$, one orbit of size $(N+1)/2$, and one fixed point. 

  {\em We say that these orbits are of type $\mathbf{D}, \mathbf{E}, \mathbf{F}$, respectively.}

\item If $j(E)= 0$, the set ${\cal C}_{\ns}^+ (N,E)$ has cardinality $\frac{N^2-N+4}{6}$.
The action of $\Gal (\overline{\Q}_N /\Q^{\unr}_N )$ on this set
has $(N-3)/2$ orbits of size $(N+1)/3$, one orbit of size $(N+1)/6$, and one fixed point.

   {\em We say that these orbits are of type $\mathbf{D}_0, \mathbf{E}_0$, and $\mathbf{F}_0$, respectively.}

\item If $j(E)=1728$, the set ${\cal C}_{\ns}^+ (N,E)$ has cardinality $(N^2 +3)/4$. The action of $\Gal (\overline{\Q}_N /\Q^{\unr}_N )$ on this set
has $(N-1)/2$ orbits of size $(N+1)/2$ and one fixed point. 

{\em We say that these orbits are of type $\mathbf{G}_0$ and $\mathbf{H}_0$, respectively.}

\end{itemize}

  For each $C \in \mathcal{C}_{\ns}^+(N, E)$, the points in the Galois orbit of $(E, C) \in X_{\ns}^+(N)(\Q_N^{\unr})$ specialize to irreducible components of $X_{\ns}^+ (N)^{\mathrm{reg}}_{/\Z^{\unr}_N}$  above the supersingular $j$-invariant $j_i \coloneq j(E)$ of $E$. More precisely, with reference to Figure \ref{figureNSgrossier+}, the following hold. %
\begin{itemize} \item If $j(E)\not\equiv 0,1728\bmod N$, the $(N-3)/2$ orbits of type $\mathbf D$ give $(N-3)/2$ divisors, each specializing to a point on the component $D_i$, of multiplicity $N+1$. The orbit of type $\mathbf E$ gives one divisor specializing to a point on the component $E_i$, which has multiplicity $(N+1)/2$. The fixed point, that is, the orbit of type $\mathbf F$, gives one divisor specializing to the component $F_i$, which has multiplicity $1$. 

\item If $j(E)\equiv 0\bmod N$, the $(N-3)/2$ orbits of type $\mathbf D$ or $\mathbf D_0$ give $(N-3)/2$ divisors, each specializing to a point on the component $D_0$, of multiplicity $(N+1)/3$. The orbit of type $\mathbf E$ or $\mathbf E_0$ gives one divisor specializing to a point on the component $E_0$, which has multiplicity\\  $(N+1)/6$. The fixed point, that is, the orbit of type $\mathbf F$ or $\mathbf F_0$, gives one divisor specializing to the component $F_0$, which has multiplicity $1$. 

\item If $j(E)\equiv 1728\bmod N$, the orbits of type $\mathbf D$ or $\mathbf G_0$ give respectively $(N-3)/2$ or $(N-1)/2$ divisors, each specializing to a point on the component $G_0$, which has multiplicity \\ $(N+1)/2$. The unique orbit of type $\mathbf E$ (when it exists) gives one further divisor specializing to a point on the same component $G_0$. The fixed point, that is, the orbit of type $\mathbf F$ or $\mathbf H_0$, gives one divisor specializing to the component $H_0$, which has multiplicity~$1$. 
\end{itemize}

\end{prop}

\begin{proof}
Since we assume $E$ has good supersingular reduction at $N$, we know from~\cite[Section 1.11]{serre1972proprietes} that 
$I_N \coloneq{\mathrm{Gal}} (\overline{\Q}_N /\Q_N^{\unr} )$ acts on $E[N]$ by a ``fundamental character 
of level $2$'', that is, the image of $I_N$ in $\GL (E[N])$ is precisely a non-split Cartan isomorphic to 
$\F_{N^2}^\times$. Let $\theta$ be a generator of this group and choose a basis $(e_1 , e_2 )$ of $E[N]\otimes 
\F_{N^2}$ of eigenvectors for $\theta$, with eigenvalues $\bar \theta, \theta$, respectively. Note that ${\mathrm{Frob}}_{\F_{N^2} /\F_N}$ (which can be seen as the Fricke involution) maps a line $\langle ae_1 +be_2 
\rangle$ to $\langle a^N e_2 +b^N e_1 \rangle$, i.e., it acts on the corresponding presentation of $\P (E[N]\otimes 
\F_{N^2} )\simeq \P^1 (\F_{N^2} )$ as 
\begin{equation}
\label{FrobNSS}
[x :1] \mapsto [x^{-N} : 1] \hspace{0.3cm} ({\mathrm{and}} \hspace{0.3cm} [1:0]\mapsto [0:1]).
\end{equation}
In particular, $\mathcal{C}_{\ns}^+ (N, E)$  can  be described as the set of unordered pairs
\[
\mathcal{C}_{\ns}^+ (N, E) =\left\{ ([x:1],[x^{-N} :1]), x\in \F_{N^2}^\times , x^{N+1} \neq 1 \right\} \cup \left\{ ([1:0], 
[0:1])\right\} /({\mathrm{Aut}} (E_{\overline \Q_N} )/\{ \pm 1 \} ),
\] 
where we quotient out by the extra automorphisms of $E$, if the curve has any.

   On the other hand, setting $\gamma \coloneq \theta^{N-1}$ (a primitive $(N+1)$th root of unity in $\F_{N^2}^\times$), the generator
$\theta$ acts on $\P (E[N]\otimes \F_{N^2} )\simeq \P^1 (\F_{N^2} )$ as
\begin{equation}
\label{GalP1}
[x : 1]\mapsto [\overline{\theta} x:\theta] = [\theta^{N-1} x:1] =[\gamma x : 1] \hspace{0.3cm} ({\mathrm{and}} 
\hspace{0.3cm} [1:0] \mapsto [1:0] ).
\end{equation}
From this description, it is clear that the Galois orbit of $([x:1],[x^{-N} :1])$ with $x\neq 0$ in $\mathcal{C}_{\ns}^+ (N, E)$ has 
size $(N+1)$, unless $x^{N+1}$ is a $(N+1)$th root of unity in $\F_{N^2}^\times$, so that $x^{N+1} =\pm 1$, and
finally $x^{N+1} =-1$. The latter possibility gives a single Galois orbit of size $(N+1)/2$, given by 
$x=\theta^{(1+2k)(N-1)/2}$ for $k\in \{ 0, 1, \dots , N\}$. This proves the first part of the proposition when $j(E)\neq 0, 1728$.

   If $E$ has exceptional automorphisms, these automorphisms induce elements in $\GL (E[N])$ which span the same non-split Cartan as the one given by the image of the action of the inertia subgroup $I_N$. Identifying Cartan structures via these automorphisms gives the first set of results when $j(E) =0$ or $1728$. 
 
    Now we discuss the specialization locus. Consider again our elliptic curve $E$, whose reduction over $\overline{\F}_N$ is assumed to be supersingular, and assume first
$j(E)\not\equiv 0, 1728\bmod N$. Any 
normalizer-of-nonsplit-Cartan structure defines a point $(E ,C_{\ns}^+ (N))=P$ in $X_{\ns}^+ 
(N)^{\mathrm{st}} ({\cal O}_K )$, for ${\cal O}_K$ the ring of integers of $K\coloneq\Q^{\unr}_N (N^{2/(N^2 -1)} )$ and 
$X_{\ns}^+ (N)^{\mathrm{st}}$ the semi-stable model described in~\cite{EdPar_max}. Recall that our regular model
$X_{\ns}^+ (N)^{\mathrm{reg}}$ is deduced from $X_{\ns}^+ (N)^{\mathrm{st}}$ by taking the Galois 
quotient (and blowing-up where needed). Explicitly, after adapting notations to the present proof, we know from 
Theorem~3.2 of~\cite{{EdPar_regular}} that $P$ specializes to some Drinfeld component with equation 
\begin{eqnarray}
\label{DrinfeldNS+1}
Y^2 =X( X^{\frac{N+1}{2}} +A_{\ns} ) 
\end{eqnarray}
for some $A_{\ns}$ in $\overline{\F}_N^\times$, with Galois action $\theta  \colon X\mapsto \theta^{2(N-1)} X , 
Y\mapsto \theta^{N-1} Y$ for our generator $\theta$ of $\F_{N^2}^\times$ (cf.~the end of the proof of Theorem~3.2 
of~\cite{{EdPar_regular}}, in particular (3.8) and the surrounding discussion). Then we take the Galois quotient of that Drinfeld component, 
blowing up if needed at points where Galois does not act faithfully. Note that the intersection locus of the Drinfeld 
components with other (Igusa) components of our special fiber occurs at points at infinity on the above 
model~(\ref{DrinfeldNS+1}). So here one checks that away from $Y=0$, Galois orbits have size $(N+1)$, and for $Y=0$ there 
is one orbit of size $(N+1)/2$ and one fixed point. Descending to ${\Z}_N^{\unr}$ by the process of the proof of 
Theorem~3.2 of~\cite{{EdPar_regular}} shows precisely that, in the special fiber of the regular model $X_{\ns}^+ 
(N)^{\mathrm{reg}}_{/\Z_N^{\unr}}$, the orbits of size $(N+1)$ reduce to the multiplicity $(N+1)$ component 
(called $D_i$ in Figure~\ref{figureNSgrossier+}), the orbits of size $(N+1)/2$ reduce 
to the multiplicity $(N+1)/2$ component (called $E_i$), and the only point which is Galois-stable specializes to $F_i$. (The cases of $D_i$ and $F_i$ are clear. As for the case of $E_i$, it is enough to consider the Galois quotient of order 
$2$ of the Drinfeld component, where a blow-up makes a component of multiplicity $1$ appear; the points in the 
orbit of size $(N+1)/2$ specialize to this new component by regularity. Taking a further quotient all the way down to ${\Z}_N^{\unr}$ 
shows that they specialize to $E_i$, and away from $D_i$.)

The cases when $j(E)\equiv 0, 1728\bmod N$ are dealt with in a similar way, taking into account the exceptional 
automorphisms of the elliptic curve at $N$ (for a description of Drinfeld components in the cases of non-trivial automorphisms, see the end of proof of Theorem~3.4 of~\cite{{EdPar_regular}}, pp.~1041--1042).
\end{proof}

\begin{prop}
\label{GaloisOrbitsOrd}
Let $E$ be an elliptic curve over $\Z_N^{\unr}$, having good ordinary reduction at $N$. Suppose that the inertia group $\Gal (\overline{\Q}_N /\Q^{\unr}_N )$ acts semisimply on $E[N]$ (for example, this holds if $E$ has CM and $N$ does not divide the discriminant of the CM order).
Then, with the same notation as in Proposition~\ref{GaloisOrbitsSS}, the following hold.
\begin{itemize}
\item If $j(E)\neq 0, 1728$, the set ${\cal C}_{\ns}^+ (N, E)$ has cardinality $\frac{N^2-N}{2}$. The action of $\Gal (\overline{\Q}_N /\Q^{\unr}_N )$ on this set has $(N-1)/2$ large orbits of size $(N-1)$ and one small orbit of size $(N-1)/2$. 

\item If $j(E)=1728$, the set ${\cal C}_{\ns}^+ (N, E)$ has cardinality $\frac{N^2-1}{4}$. The action of $\Gal(\overline{\Q}_N /\Q^{\unr}_N )$ on this set has $(N+1)/2$ orbits of size $(N-1)/2$. 

\item If $j(E)= 0$, the set ${\cal C}_{\ns}^+ (N, E)$ has cardinality $\frac{N^2-N}{6}$. The action of $\Gal (\overline{\Q}_N /\Q^{\unr}_N )$ on this set has $(N-1)/2$ large orbits of size $(N-1)/3$, and one small orbit of size $(N-1)/6$. 
\end{itemize}

All orbits specialize to the $j$-line $A$ of $X_{\ns}^+ (N)^{\mathrm{reg}}_{\ol \F_N}$, provided $j\not\equiv 0\bmod N$. If $j\equiv 0\bmod N$, the small orbit specializes to the component labeled ${\cal E}_0$ (which has multiplicity $(N-1)/6$), and the large ones specialize to ${\cal D}_0$ 
(which has multiplicity $(N-1)/3$).
\end{prop}
\begin{remark}\phantom{~}
\begin{itemize}
\item That $\Gal (\overline{\Q}_N /\Q^{\unr}_N )$ acts semisimply on $E[N](\overline{\Q}_N )$ if $E$ has 
CM and $N$ does not divide the discriminant of the order 
$\mathrm{End} (E)_{\overline{\Q}_N}$ follows from the fact that over 
$\Q_N^{\unr}$ we have $\End(E_{\Q_N^{\unr}}) = \End(E_{\ol \Q_N})$, and the action of Galois commutes with $\mathrm{End} (E)\otimes_{\Z} \F_N \cong \F_N \times \F_N$.

\item If the inertia group does not act semisimply on $E[N]$, one can show that there is a unique Galois orbit. Describing its specialization is more complicated.
\end{itemize}
    
\end{remark}
\begin{proof}
We parallel the arguments of Proposition~\ref{GaloisOrbitsSS} above. As $E$ is now supposed to have good ordinary 
reduction at $N$ and $I_N \coloneq{\mathrm{Gal}} (\overline{\Q}_N /\Q_N^{\unr} )$ acts semisimply on $E[N]$, the action is via some ``fundamental character of level $1$''. That is, there is a choice of some $\F_N$-basis $(e_1 , e_2 )$ of $E[N]$ for which the image of $I_N$ is 
$\left(\begin{smallmatrix}\F_N^\times & 0 \\ 0 & 1\end{smallmatrix}\right)$.
Now ${\mathrm{Frob}}_{\F_{N^2} /\F_N}$ maps a line $\langle ae_1 +be_2 
\rangle$ to $\langle a^N e_1 +b^N e_2 \rangle$, that is, it acts on the corresponding presentation of $\P (E[N]\otimes 
\F_{N^2} )\simeq \P^1 (\F_{N^2} )$ as 
\begin{equation}
\label{FrobNOrd}
[x :1] \mapsto [x^{N} : 1] \hspace{0.3cm} ({\mathrm{and}} \hspace{0.3cm} [1:0]\mapsto [1:0]).
\end{equation}
So now ${\cal C}_{\ns}^+ (N)$ can be described as the set of pairs
\[
{\cal C}_{\ns}^+ (N) =\left\{ \{ [x:1],[x^{N} :1]\} , x\in \F_{N^2}^\times \setminus \F_N^\times  \right\} 
/({\mathrm{Aut}} (E_{\overline \Q_N} )/\{ \pm 1 \} ),
\] 
where we quotient out by the extra automorphisms of $E$, if the curve has any.

   On the other hand, letting $\theta$ be a generator of $\F_{N^2}^\times$ and taking $\gamma \coloneq \theta^{N+1}$ as generator of $\F_N^\times$, we see that $\gamma$ acts on $\P (E[N]\otimes \F_{N^2} )\simeq \P^1 (\F_{N^2} )$ as 
\begin{equation}
\label{GalPOrd}
[x : 1]\mapsto [\gamma x:1]  \hspace{0.3cm} ({\mathrm{and}} 
\hspace{0.3cm} [1:0] \mapsto [1:0] ).
\end{equation}
From this description, it is clear that the Galois orbit of $([x:1],[x^{N} :1])$ with $x\in \F_{N^2}^\times \setminus \F_N^\times$ in 
${\cal C}_{\ns}^+ (N)$ has 
size $N-1$, unless for some nontrivial power $\gamma^k$ the unordered pair~$\{ [\gamma^k x:1],[\gamma^k x^N:1] \}$ 
coincides with $\{[x:1],[x^N:1] \}$. Since $x\notin\F_N^\times$, this can only happen by exchanging the two lines, hence $\gamma^k=x^{N-1}$. Thus $x^{N-1}\in\F_N^\times$ has order dividing $(N+1, N-1) = 2$. The case $x^{N-1}=1$ would imply $x\in\F_N^\times$, so we have a smaller Galois orbit only if $x^{N-1}=-1$, in which case it has size $(N-1)/2$.
This proves the first part of the proposition when $j(E)\neq 0, 1728$.

  When $j(E)=0$, the exceptional automorphism of order $3$ commutes with Galois, and we conclude similarly. If $j(E)=1728$, then all orbits have size $(N-1)/2$.

If $j(E) \not\equiv 0 \bmod N$, the points must all specialize to the $A$ component, since it is the only component containing ordinary $j$-invariants different from $0$.

If $j(E) \equiv 0 \bmod N$, we revisit the geometric arguments around $(3.12)$ in the proof of~\cite{EdPar_regular}, Theorem~3.4. Let $\pi_2 \coloneq N^{1/(N-1)}$.
Consider the quotient of the Igusa component of the semistable model $X_{\ns}^+ (N)^{\mathrm{st}}$ 
over $\Z_N^{\unr} [\pi_2 ]$, and {\em not} over 
the index-$2$
subring $\Z_N^{\unr} [N^{2/(N-1)} ]$ as in loc. cit. (We use notations as in~\cite{EdPar_regular}, remembering that our $N$ here is called $p$ 
there). At points above $j=0$ on those
Igusa components, the inertia subgroup of Galois now acts as
$$
\zeta_3 \colon \pi_2 \mapsto \zeta_3 \, \pi_2 ,\hspace{1cm} \zeta_3 \colon t \mapsto \zeta_3 \, t
$$
so unlike (3.12) of~\cite{EdPar_regular}, the action is
now {\em diagonal}. Take the quotient of $X_{\ns}^+ (N)$ by that order-$3$ subgroup of Galois, and blow-up the quotient singularities over the base 
$\Z_N^{\unr} [\pi_2^3 ]$: that makes 
a projective line with multiplicity $1$ appear at each of our points  with $j=0$ on the Igusa components, which in turn now have multiplicity $3$. All points of $X_{\ns}^+ (N)
(\ol{\Q}_N )$ associated with our elliptic curve $E$ for which $j= 0$ do define sections of $X_{\ns}^+ (N)$
over $\Z_N^{\unr} [\pi_2^3 ]$, so those sections all specialize to the exceptional $\P^1$s above $N$. Now we make a further quotient of our regular model over $\Z_N^{\unr} 
[\pi_2^3 ]$ by the order-$2$ subgroup of Galois, whose action has one single fixed point on each $\P^1$ (in addition to the intersection point with the Igusa component). Because of the moduli interpretation, that fixed point on $\P^1$ is the specialization of an element of the ``small orbit'' associated with $E$, as the latter have values in the subfield 
$\Q_N^{\unr} [\pi_2^6 ] \subset \Q_N^{\unr} [\pi_2^3 ]$. The action of the degree-$2$ element of Galois acts diagonally on the cotangent space of those fixed points, and to obtain a regular model for $X_{\ns}^+ (N)$ over $\Z_N^{\unr} [\pi_2^6 ]$, 
we need to make a last blow-up there.
It follows that the exceptional 
$\P^1$s now have multiplicity $2$ in the special fiber of the regular model of 
$X_{\ns}^+ (N)_{/\Z_N^{\unr} [\pi_2^6 ]}$, and they intersect still another $\P^1$, with multiplicity $1$, which rises up at their former fixed point. Note that at that stage we have recovered the regular model of $X_{\ns}^+ (N)$ over 
$\Z_N^{\unr} [\pi_2^6 ]$
described after (3.12) in~\cite{EdPar_regular}. Again each 
point in the ``small orbit'' of the divisor associated with $E$ 
specializes to the new $\P^1$s with multiplicity $1$, whereas the divisors 
of the ``large orbits'' specialize to the $\P^1$s with multiplicity $2$. We can finally take the quotient of
$X_{\ns}^+ (N)$ all the way down from $\Z_N^{\unr} [\pi_2^6 ]$ to 
$\Z_N^{\unr}$ as in \cite{EdPar_regular}, end of proof of Theorem 3.4. On that final firm ground $\Z_N^{\unr}$, we thus have that divisors associated with large (respectively, small) orbits specialize to ${\cal D}_0$ (respectively, ${\cal E}_0$).  
\end{proof}

\begin{prop}
\label{Jmultiplicity}
Let $j(E)$ be a $j$-invariant in $\Q_N^{\unr}$ with supersingular reduction. The divisor of the function $j-j(E)$ on the base-change to $\Z_N^{\unr}$ of $X_{\ns}^+ (N)^{\mathrm{reg}}$ is given by
\[
\mathrm{div} (j-j(E)) =\left( \sum_{C\in {{\cal C}_{\ns}^+ (N, E)}} m_{C} \cdot (E,C) \right) -({\mathrm{Cusps}}) +\Phi_{j-j(E)}
\]
where $m_{C} = \frac12 \# \Aut(E,C)$  (in particular $m_{C} =1$ unless $j(E)=0,1728$), 
the $(E,C)$ specialize to the supersingular components $D_i$, $E_i$, $F_i$, $G_0$, $H_0$ above $j(E)\bmod N$ according to
Proposition~\ref{GaloisOrbitsSS}, 
the divisor $({\mathrm{Cusps}})$ of cusps, counted with the appropriate multiplicity, specializes to the multiplicity-$(N-1)/2$ 
component $A$, and the vertical divisor $\Phi_{j-j(E)}$ has the following expression:
\begin{itemize}
\item if $j(E)\not\equiv 0, 1728\bmod N$, then $\Phi_{j-j(E)} =\left( N\cdot D_j +(N+1)/2 \cdot E_j + F_j \right)$, where $D_j,$ $E_j,$ $F_j$ are the components in Figure \ref{figureNSgrossier+} above the $j$-invariant $j(E) \bmod N$;
\item if $j(E)\equiv  0\bmod N$, then $\Phi_{j} =\left( N\cdot D_0 +(N+3)/2 \cdot E_0 + 3\cdot F_0 \right)$;
\item if $j(E)\equiv  1728\bmod N$, then $\Phi_{j-1728} =\left( N\cdot G_0 +2 \cdot H_0 \right) $ .
\end{itemize} 
\end{prop}

\begin{proof}
We combine the above propositions
with the results of (the independent!) Section~\ref{sec:V} 
regarding specialization of CM divisors. Indeed, the above expression for $\divisor(j-j(E))$ is clear when restricted to $X_\Q$, since the zeros of this function are exactly the points for which the underlying elliptic curve is $E$ and the poles are cusps. Then $\divisor(j-j(E))$ has to have the above expression for \emph{some} vertical divisor $\Phi_{j-j(E)}$. We know that $\Phi_{j-j(E)}$ is concentrated on the bad reduction fiber, and since $j$ is not constant there,
its support does not contain the $j$-line $A$ in the bad reduction (see Figure~\ref{figureNSgrossier+}). Finally, we can get a unique expression for $\Phi_{j-j(E)}$, by using that, being a principal divisor, the intersection number $\divisor(j-j(E)) \cdot C$ is zero for all components $C$ of $X_{\F_N}$.
\end{proof}

\begin{cor}\label{cor:parameter_F_j}
If $j(E)$ is a supersingular $j$-invariant mod $N$, and $j(E)\not\equiv 0, 1728\bmod N$, then the function $(j-j(E))/N$ gives an isomorphism between $\P^1$ and the irreducible multiplicity-one component $F_j$ above $j(E)$ in the special fiber of our regular model 
$X_{\ns}^+ (N)^{\mathrm{reg}}$ over $\overline{\F}_N$. 

Similarly, if $j(E)\equiv 0\bmod N$ (respectively, $j(E)\equiv 1728 \bmod N$) is supersingular, then $j/N^3$ 
(resp., $(j-1728)/N^2$) is a parameter on the genus~$0$ irreducible multiplicity-one component $F_0$ (resp., $H_0$) 
in $X_{\ns}^+ (N)^{\mathrm{reg}}_{/\overline{\F}_N}$.
\end{cor}

\begin{proof}
By Proposition~\ref{Jmultiplicity}, the coefficient of $F_j$ in the vertical part of $\divisor(j-j(E))$ is $1$ when $j(E)\not\equiv 0, 1728\bmod N$. Since the divisor of $N$ on the regular model is the full special fiber, and $F_j$ has multiplicity $1$, the rational function $(j-j(E))/N$ has order $0$ along $F_j$. Moreover, by Proposition \ref{GaloisOrbitsSS}, only one of the points with $j$ equal to $j(E)$ specializes to $F_j$, hence $(j-j(E))/N$  has only one zero, with multiplicity one, on $F_j$, and hence it is an isomorphism $F_j \to \P^1_{\F_N}$.
A similar argument applies to $j/N^3$ on $F_0$, because the coefficient of $F_0$ in $\Phi_j$ is $3$, and to $(j-1728)/N^2$ on $H_0$, because the coefficient of $H_0$ in $\Phi_{j-1728}$ is $2$. 
\end{proof}

\subsection{Hecke operators and intersection with the diagonal}
\label{sec:heckediag}

Consider the diagonal morphism $\Delta \colon X_{\ns}^+ (N)\to (X_{\ns}^+ (N))^2$. Given a Hecke operator $T$,
seen as a divisor on $X_{\ns}^+ (N)^2$, we want to describe the divisor $ \Delta^* (T )$, which will be used in our method,
see e.g.~\eqref{eq:D_H_D_V}. We fix for all this section 
some prime number $\ell$, and consider the case $T=T_\ell$.
Later, we use this to study an integral linear combination of $T_2$ and $T_3$, having trace~$0$. 

\begin{remark}
    It might help with visualization to recall that, over $\C$, the Hecke operator $T_\ell$ can be seen as the superposition of the graphs of
$\ell+1$ functions on the relevant quotient of the Poincar\'e upper half-plane, of shape $(\tau , g(\tau ))$, for some 
M\"obius transformations $g(\tau )\coloneq\frac{a\tau +b}{c\tau +d}$, where $\left( \begin{smallmatrix}
    a & b \\ c & d
\end{smallmatrix}\right) \in M_2 (\Z )$
has determinant $\ell$ (cf.~\cite[Section~1.3.2]{KohenPacetti} for non-split Cartan curves, \cite[ (7.3.1)]{DiamondIm} 
for $X_0(N)$, or more generally \cite{ShimuraBook}).
\end{remark}

\begin{lemma}
\label{SuppDT}
Let $\ell 
\not\equiv  0,\pm 1 \bmod N$ be a prime number. Then, the support ${\mathrm{Supp}}(\Delta^*T_\ell)$ of $\Delta^*T_\ell$ consists of
non-cuspidal points, associated with complex multiplication (CM) elliptic curves $E$ having an element $\alpha$ of degree $\ell$ in their endomorphism ring. 
Supposing, moreover, that $\ell < N^2 /4$, we have the following description. Assume first $j(E)\neq 0, 1728$. 
\begin{itemize}
    \item 
Suppose $\Tr (\alpha) \neq 0$. If $N$ is inert in $\Q (\alpha )$, then $E$ gives rise to one 
point in ${\mathrm{Supp}}(\Delta^*T_\ell)$, namely the point of type $\mathbf{F}$ (in the sense of Proposition~\ref{GaloisOrbitsSS}) associated with $E$; otherwise, $E$ gives no point in ${\mathrm{Supp}}(\Delta^*T_\ell)$.
    \item  
Suppose ${\mathrm{Tr}} (\alpha )=0$. If $N$ is split in $\Q (\alpha )$, then $E$ gives rise to $(N-1)/2$ points in ${\mathrm{Supp}}(\Delta^*T_\ell)$, which make up the (only) small orbit of Proposition~\ref{GaloisOrbitsOrd}. If $N$ is inert in $\Q (\alpha )$, then $E$ gives rise to $(N+3)/2$ points, namely those of the two orbits of types $\mathbf{E}$ and $\mathbf{F}$ (see Proposition~\ref{GaloisOrbitsSS}).
\end{itemize}

Finally, consider curves with exceptional automorphisms. 
\begin{itemize}
\item Suppose $j(E)=0$ and $\ell\neq 3$. If $N$ is inert in $\Q (\zeta_3 )=\Q (\alpha )$, there is precisely one point in ${\mathrm{Supp}}(\Delta^*T_\ell)$ with 
underlying curve $E$ (the orbit of type $\mathbf{F}_0$); otherwise, $E$ gives no point in ${\mathrm{Supp}}(\Delta^*T_\ell)$.

\item Suppose $j(E)=0$ and 
$\ell=3$. If $N$ is split in 
$\Q (\zeta_3 )$, then there are $(N-1)/6$ points in ${\mathrm{Supp}} (\Delta^*T_\ell )$: the small orbit of Proposition~\ref{GaloisOrbitsOrd}. If $N$ is inert in $\Q (\zeta_3 )$, then there are $(N+7)/6$ points with underlying $j(E)=0$ in ${\mathrm{Supp}} (\Delta^*T_\ell )$ (the two orbits of type $\mathbf{E}_0$ and $\mathbf{F}_0$). 

\item Suppose $j(E)=1728$; then necessarily ${\mathrm{Tr}} (\alpha )\neq 0$. If $N$ is inert in $\Q (i )=\Q (\alpha )$, there is exactly one point with 
underlying curve $E$ in ${\mathrm{Supp}} (\Delta^*T_\ell )$, namely the point of type $\mathbf{H}_0$; otherwise, $E$ gives no point in ${\mathrm{Supp}}(\Delta^*T_\ell)$.
\end{itemize} 
\end{lemma}

\begin{proof} Assume $P$ in ${\mathrm{Supp}} (\Delta^*T_\ell )$ is not a cusp, therefore associated with a pair 
$(E,C_{\ns}^+ (N))$ with $E$ an elliptic curve and $C_{\ns}^+ (N)$ some normalizer-of-nonsplit-Cartan 
structure mod~$N$. Denote by $C_{\ns}(N)$ the corresponding Cartan group. Since there is a cyclic $\ell $-isogeny from $E$ to 
itself, $E$ must be CM, with endomorphism ring $\operatorname{End}(E) \coloneq {\mathrm{End}}_\C (E)$ containing an element $\alpha$ of degree 
$\ell $. (Recall that degree is also the norm of $\alpha$ seen as a complex number; we shall use both terms in the following computations.) There are finitely many such elements in $\C$, belonging to a finite list of (not necessarily maximal) imaginary quadratic orders. 
Moreover, for our pair $(E,C_{\ns}^+ (N))$ to define a point on 
$X_{\ns}^+ (N)$, the $\ell $-isogeny $\alpha \colon E\to E$ must induce an element in $\GL (E[N])$ which belongs to 
$C_{\ns}^+(N)$.

Note that, since $\deg (\alpha) = \ell < N^2/4$ , the trace of $\alpha$ is equal to $0$ if and only if it is divisible by $N$. 

Suppose first that ${\mathrm{Tr}} (\alpha )\not\equiv 0 \pmod N$, so that $\alpha$ acts on $E[N]$ as an element of $C_\ns(N) 
\simeq \F_{N^2}^\times$ (and not its normalizer). If $\alpha$ acts as an element $z\in \Z$ on $E[N]$, then $E[N] \subset 
\ker(\alpha-z)$, i.e.~$\alpha-z  = N\beta$ for some $\beta \in \End(E)$. Let $-D$ be the discriminant of $\End(E)$ 
and write $\beta = \tfrac{a+b\sqrt{-D}}{2}$. We have $\alpha = N\beta+z$, hence $\alpha$ is of the form $\tfrac{c+Nd\sqrt{-
D}}{2}$. Since $\alpha$ has norm $\ell$ we obtain $4\ell = c^2+N^2d^2D>N^2$, contradicting the hypothesis. We deduce that $\alpha$ acts 
on the $N$-torsion as an element of $\F_{N^2}\setminus \F_N$, hence that $\F_N[\alpha] \cong \F_{N^2}$ and, since there is a 
ring homomorphism $\F_N[\alpha] \to \End(E)/N\End(E)$, the prime $N$ is inert in $\End(E)$. In particular, if $N$ is not inert, then there is no point of the form $(E,C_\ns^+(N))$ in $\mathrm{Supp}(\Delta^*T_\ell)$.
Conversely, if $N$ is inert in $\End(E)$, then $\alpha$ modulo $N$ is an element of $\calO/N \mathcal{O} \cong \F_{N^2}$, hence it acts as 
an element of a unique non-split Cartan.

Assume now $\mathrm{Tr} (\alpha )\equiv 0 \bmod N$. By the same argument as above, $\alpha$ mod $N$ belongs to 
exactly one non-split Cartan subgroup if and only if $N$ is inert in $\Q (\alpha )$; however, we must now also count the possible {\it normalizers} of non-split Cartan subgroups that $\alpha$ can belong to. We proceed as in the proofs of Proposition~\ref{GaloisOrbitsSS} and Proposition~\ref{GaloisOrbitsOrd}. Recall that the normalizer of non-split Cartan subgroups $C_{\ns}^+ (N)$ are in one-to-one correspondence with 
(unordered) pairs 
of distinct $\F_{N^2} /\F_N$-conjugate lines in $E[N]\otimes_{\F_N} \F_{N^2}$. So, endowing $E[N]\otimes_{\F_N} \F_{N^2}$ 
with a basis $(Q_1 ,Q_2 )$ of eigenvectors for $\alpha$ mod $N$, we see that $\alpha$ acts on $\P (E[N]\otimes_{\F_N} 
\F_{N^2})\simeq \P^1 (\F_{N^2})$ as $[x\colon 1] \mapsto [-x\colon 1]$ (and $[1\colon 0] \mapsto [1\colon 0]$). 
On the other 
hand, if $N$ is split (respectively, inert) in $\Q (\alpha )$, then $\mathrm{Frob}_N$ stabilizes the lines spanned by 
$Q_1$ and $Q_2$ (resp., $\mathrm{Frob}_N$ exchanges these two lines). 
Notice that $N$ cannot ramify in $\Z[\alpha]$, since, together with the hypothesis $N \mid \mathrm{Tr}(\alpha)$, this would imply that $N$ divides $\ell = \deg(\alpha)$.
Therefore, ${\mathrm{Gal}} (\F_{N^2}/\F_N)$ acts on 
$\P (E[N]\otimes_{\F_N} \F_{N^2})$ via the formulas
(\ref{FrobNSS}) and 
(\ref{FrobNOrd}) above:
\begin{itemize}
	\item $\mathrm{Frob}_N \colon [x\colon 1] \mapsto [\overline{x} \colon 1] =[x^N \colon 1]$ (and $\mathrm{Frob}_N 
	([1\colon 0])=[1\colon 0]$) if $N$ is split in $\Q (\alpha )$; or 
	\item $\mathrm{Frob}_N \colon [x\colon 1] \mapsto [1/\overline{x} \colon 1] =[1/x^N \colon 1]$ for $x\neq 0$ (and 
	$\mathrm{Frob}_N$ switches  $[1\colon 0]$ and $[0\colon 1]$) if $N$ is inert in $\Q (\alpha )$.
\end{itemize}
Therefore, assuming first $E$ has no extra automorphisms (that is, $j(E)\neq 0, 1728$), one sees that
$\alpha$ normalizes precisely $(N-1)/2$ non-split Cartan subgroups if $N$ is split in $\Q (\alpha )$; or $(N+3)/2$ non-split Cartan subgroups: the unique Cartan subgroup containing $\alpha$, together with $(N+1)/2$ further Cartan subgroups normalized by $\alpha$.

In the case of exceptional automorphisms, when $\mathrm{Tr} (\alpha )\equiv 
0$ mod~$N$, the same computations hold, except that we must identify 
the relevant points $(E,C_{\ns} (N)^+)$ when needed (which happens to be necessary only when 
$j(E)=0$ and $\alpha =\pm i\sqrt{3}$, so that $\ell =3$).

Finally let us consider the points at infinity. Using notations as in Deligne--Rapoport \cite[Chapter VII]{DR}, recall that cusps 
for the full modular group $\Gamma (N)$ say over $\C$, can be seen as N\'eron $N$-gons, which are the specializations at 
$q=0$ of the Tate curve 
${\cal G}_{\mathrm m}^{q^{1/N}} /q^\Z$ over $\C [[q^{1/N}]]$ endowed with a fixed basis $(q^{a/N} 
\zeta_N^b , q^{c/N} \zeta_N^d )$ of the $N$-torsion. Two such specializations are identified under any change of choice $q^{1/N} \leftrightarrow 
q^{1/N} \zeta_N^*$ of the $N^{\mathrm{th}}$-root of $q$ and multiplication of the whole $N$-torsion by $\pm 1$, cf.~\cite[p.~158]{DR}. 

   Labeling the $\ell$-isogenies on the Poincar\'e upper-half plane as
\[
\varphi_{\ell} \colon \left\{
\begin{array}{rcl}
\C /(\Z +\tau \Z ) & \to & \C /(\Z +\ell \tau \cdot \Z )\\
z& \mapsto & \ell z  
\end{array} 
\right.
\hspace{0cm} {\mathrm{ and }} \hspace{0.2cm}
{\varphi_{1,\alpha}} \colon \left\{
\begin{array}{rcl}
\C /(\Z +\tau \Z ) & \to & \C /(\Z +\frac{\tau +\alpha}{\ell} \cdot \Z )\\
z& \mapsto & z  
\end{array} 
\right.
\]
for $0\le \alpha \le \ell -1$, we have 
\begin{eqnarray*}
\varphi_{1,\alpha} ({\cal G}_{\mathrm m}^{q_\tau^{1/N}} /q_\tau^\Z ,(q_\tau^{a/N} \zeta_N^b ,q_\tau^{c/N} \zeta_N^d ))
 =({\cal G}_{\mathrm m}^{q_{(\tau +\alpha)/\ell}^{1/N}} /q_{(\tau +\alpha)/\ell}^\Z ,(q_{\tau}^{a/N} 
\zeta_N^b ,q_{\tau}^{c/N} \zeta_N^d )).
\end{eqnarray*}
Specializing to $q=0$, we see that $\varphi_{1,\alpha}$ acts by left multiplication by $\left( \begin{smallmatrix}
\ell & 0 \\
-\alpha & 1
\end{smallmatrix}\right) $
on the $N$-torsion of our $N$-gon (and because of the isomorphisms for $q^{1/N} \leftrightarrow q^{1/N} \zeta_N^*$,  
all $\ell$ images are enhanced $N$-gons defining the same cusp). A similar calculation shows that 
$\varphi_{\ell}$ acts  by left multiplication by 
$
\left( \begin{smallmatrix}
1 & 0\\ 
0 & \ell 
\end{smallmatrix}\right)$
on this $N$-torsion. These matrices can belong to some non-split Cartan subgroup if and only if $\ell \equiv 1\bmod N$, 
and to some normalizer (but not the Cartan itself) if and only if $\ell \equiv -1\bmod N$. 
\end{proof}

\begin{lemma}
\label{MultDT}
Let $\ell$ be a prime number and $P$ be a non-cuspidal point of $\Delta^* (T_\ell )$ 
with underlying elliptic curve 
$E$. The ring $\End_\C(E)$ contains elements of degree $\ell$; let $\alpha \in \End_\C(E)$ be any such element. The multiplicity of $\Delta^* (T_\ell )$ at $P$ is $1$ if  ${\mathrm{Tr}} (\alpha ) \equiv 0 \bmod \ell $ and $2$ otherwise.
\end{lemma}

\begin{proof} Write $E=E_\tau$, having complex multiplication by an order of $\Q(\tau)$, which in turn contains 
$\alpha$ with $\vert \alpha \vert^2 =\ell $.
Then $\Z[\alpha]$ has index coprime to $\ell$ in the quadratic ring $\End_\C(E_\tau)$, and
${\mathrm{End}}_\C (E_\tau ) \otimes \F_\ell $ is isomorphic to $\F_\ell  [\alpha ]\simeq \F_\ell  
[X]/(X^2 -{\mathrm{Tr}} (\alpha )X)$. It follows that $E_\tau$ possesses 
exactly 2 (resp.~1) self-isogenies of degree $\ell$ up to isomorphism, if ${\mathrm{Tr}} (\alpha ) \not
\equiv 0$ mod $\ell $ (resp.~${\mathrm{Tr}} (\alpha ) \equiv 0$ mod $\ell $), namely $\alpha$ and 
$\overline{\alpha}$. 
Moreover, $\alpha$ and $\overline{\alpha}$ normalize the 
same groups $C_{\ns} (N)$, so they induce the same point $(E,C_{\ns}^+ (N) )$ of $\Delta^*T_\ell$. 
This means that locally, in the complex picture,  
there are $2$ (respectively, $1$) components of the divisorial correspondence $T_\ell $ that 
intersects $\Delta (X_{\ns}^+ (N) )$ at $(P,P)$ if ${\mathrm{Tr}} (\alpha ) \not\equiv 0$ 
mod $\ell $ (resp., ${\mathrm{Tr}} (\alpha ) \equiv 0$ mod $\ell $). 

Now we want to prove that all components of $T_\ell $ intersect the diagonal transversely. It is enough to prove this over $\C$.
 Let $\Gamma_{\ns}^+ 
(N)$ be the inverse image in ${\mathrm{PSL}}_2 (\Z )$ of some normalizer of non-split Cartan mod $N$, 
so that we have the complex uniformization $\HH /\Gamma_{\ns}^+ (N) \simeq Y_{\ns}^+ 
(N) (\C )$. Away from the elliptic points (pairs $(E, C_\ns^+(N))$ having extra automorphisms), we can take the $\tau$-invariant in $\HH$ as a parameter. 

Assume first $P\in \Delta^* (T_\ell)$ is non-elliptic, associated with some quadratic imaginary 
$\tau_0$ in the upper half plane and some $\alpha$. There are a homography $g(\tau )=
\frac{\tau +a}{\ell }$, $0\leq a\leq \ell -1$, or $g(\tau )=\ell \tau$, and a 
$\gamma=\left(\begin{smallmatrix} A & B \\ C & D \end{smallmatrix}\right)$ in ${\mathrm{SL}}_2 (\Z )$, such that 
$g(\tau_0 ) =\gamma (\tau_0 )$. Locally around $(P,P)$ on 
 $Y_{\ns}^+ (N)^2$ for the naive complex topology, the $\alpha$-component of the divisor 
$T_\ell $ can be described by $(\gamma (\tau ), g(\tau ))$ around $\gamma (\tau_0 )$. Therefore, the tangent vector to this graph is given by
\[
\left( \gamma (\tau ),g(\tau )\right) ' (\tau_0 ) = \left( \frac{1}{(C\tau_0 +D)^2} ,\frac{1}{\ell } \right) 
\hspace{0.5cm} {\mathrm{or}} \hspace{0.5cm} \left( \frac{1}{(C\tau_0 +D)^2} ,\ell \right),
\] 
which cannot be collinear with $(1,1)$. (Indeed, since $C\tau_0 {+}D$ is a quadratic imaginary number, 
its square is rational only if $C\tau_0 {+}D$ is purely imaginary or rational, hence the square is either 
negative or the square of a rational number. In either case, it cannot be equal to $\ell^{\pm 1}$.) In other words, the curves associated with $T_\ell $ cannot be tangent to the diagonal around $(P,P)$ in $X_{\ns}^+ (N)^2_\C$, and therefore also in $X_{\ns}^+ (N)^2_\Q$. 

Now assume that $P$ is an elliptic point, corresponding to some pair $(E,C_{\ns}^+ (N))/\{ \pm 1\}$ having an 
automorphism of order $2$ 
(resp.~$3$). 
The $\tau$-invariant is not a local parameter, 
so we choose some biholomorphism $f (z)$ from $\HH$ to the unit disc ${\mathbb D}$, mapping $g(\tau_0 )(=\frac{\tau_0 +a}{\ell }$ or $\ell \tau_0$) as above to $0$, as one does when endowing quotients $\HH /\Gamma$ with a structure of Riemann surface (see for example \cite[(1.1.5)]{ShimuraBook}). 
Very explicitly, consider the usual conformal transformation $F(z)=\frac{z-i}{z+i}$ from $\HH$ to ${\mathbb D}$, let $\gamma (z)\coloneq(Az+B)/(Cz+D)$ which maps $\tau_0 =\lambda i+\mu$ to $g(\tau_0 )$ as in the computations above for the non-elliptic cases, and take $f (z)=F((\ell z -a-\mu )/\lambda )$ (in the case $g(\tau )=
\frac{\tau +a}{\ell }$)  or $f(z)=F((z /\ell -\mu )/\lambda )$ (case $g(\tau )=\ell \tau$). 
Now $f^2 (\gamma (\tau ) )$ (respectively $f^3 (\gamma (\tau ) )$) is a parameter on $X_{\ns}^+ (N)(\C )$ around $\tau_0$. 
If $P$ has an  automorphism of order $2$, then considering the first few terms of the Taylor expansion of $T (\tau )\coloneq(f^2 (\gamma (\tau )) , f^2 
(g(\tau ) ))$ shows that the tangent direction to the graph of $T(\tau )$ in $X_\ns^+(N)(\C)^2$ is given by $(\frac{1}{(C\tau_0 +D)^4} ,
\ell^{\pm 2} )$, which is not proportional to $(1,1)$, since $C\tau_0 {+}D$ is an element of $\Q (i)$, which does not contain any 
element of the form $\pm\sqrt{\pm\ell^{\pm 1}}$ for $\ell$ prime.
The same holds for elliptic points having 
automorphisms of order $3$: in this case, what we need to check is that, for $\tau_0 \in \Q (\zeta_3 )$, the vector $(\frac{1}{(C\tau_0 
+D)^6} ,\ell^{\pm 3} )$ is not proportional to $(1,1)$. 
\end{proof}

In computations, it is sometimes easier to use the following result, where we use that for $\ell \not \equiv 0, \pm 1 \bmod N$ the support of $\Delta^* T_\ell$ contains no cuspidal points by Lemma \ref{SuppDT}.

\begin{cor}\label{cor:delta_cap_Tell}
For every prime $\ell \not \equiv 0, \pm 1 \bmod N$ we have 
\begin{equation}\label{eq_cor_giagonal_hecke} 
\Delta^* T_\ell   = \sum_{\substack{E \text{ with an}\\ \text{endomorphism} \\ \text{of deg }\ell}} \sum_{0\subsetneq C \subsetneq E [\ell]} D_{E,C},
\end{equation}
where we denote by $D_{E,C}$ the divisor containing with multiplicity one each point $(E,[\phi]) \in X_\ns^+(N)$ such that 
there exists an endomorphism $\alpha$ of $E$, with kernel $C$, such that $\phi \circ \alpha \circ \phi^{-1}$ is an element of $C_{\ns}^+(N)$.
\end{cor}
\begin{proof}
    By the first part of Lemma \ref{SuppDT}, the hypothesis $\ell \not\equiv 0,\pm 1 \bmod N$ implies that the divisor $\Delta^*T_\ell$ does not contain cuspidal points. Then, by the definition of the Hecke operator, every point in the support of $\Delta^*T_\ell$ is of the form $(E,[\phi])$, with $E$ an elliptic curve having an endomorphism $\lambda$ of degree $\ell$ such that $\phi \circ \lambda \circ \phi^{-1}$ lies in $C_\ns^+(N)$. In particular, the two divisors in \eqref{eq_cor_giagonal_hecke} have the same support. It remains to check that each point $(E,[\phi])$ appears with the same multiplicity. Lemma \ref{MultDT} tells us that, on the left, the multiplicity is equal to $1$ if $\Tr(\lambda)$ is divisible by $\ell$, and $2$ otherwise. 
    
    In the first case, this means that every endomorphism of $E$ of degree $\ell$ lies in the set $\Aut(E) \cdot \lambda$. This implies that there is at most one subgroup $C$ such that $(E,[\phi])$ belongs to $D_{E,C}$, i.e., that the multiplicity of $(E,[\phi])$ in the right hand side of  \eqref{eq_cor_giagonal_hecke} is at most $1$, and hence equal to $1$. 

    If instead $\Tr(\lambda) \not \equiv 0 \bmod \ell$, then $\ell$ splits in $\End(E)$, and the endomorphisms of $E$ of degree $\ell$ are the ones in $\Aut(E) \cdot \lambda$ and the ones in $\Aut(E) \cdot \ol \lambda$, where $\ol \lambda$ is the dual of $\lambda$. It follows that there are $2$ distinct possible kernels, namely $C_1 \coloneq \ker \lambda$ and $C_2 \coloneq \ker \ol\lambda$. To prove that $(E,[\phi])$ has multiplicity $2$  in the right hand side of  \eqref{eq_cor_giagonal_hecke}, we prove that $(E,[\phi])$ appears both in $D_{E,C_1}$ and in $D_{E,C_2}$. For $D_{E,C_1}$, this comes from the definition of $\lambda$. For $C_2$, we note that $\ol \lambda \circ \lambda = \ell$, hence we have $\ol \lambda = [\ell] \lambda^{-1}$  on $E[N]$, which makes sense since $\lambda$ has degree $\ell$ coprime to $N$, hence it is invertible on $E[N]$. We then have $\phi \circ \ol \lambda \circ \phi^{-1} = (\ell \cdot \Id) \cdot (\phi \circ \lambda \circ \phi^{-1})^{-1}$, which lies in $C_{\ns}^+(N)$, since  both factors do.  
\end{proof}

\begin{remark}
By going through the proof of Lemma \ref{SuppDT}, one can check that the support of $\Delta^*T_\ell$ consists of non-cuspidal points for all integers $\ell$ that satisfy the following technical condition: there is no divisor $d$ of $\ell$ such that $d \equiv \pm \ell/d \pmod{N}$. As a consequence, the above corollary also holds for such integers $\ell$ (coprime to $N$).
\end{remark}

Combining the previous results, we now describe how the graph of $T_\ell$ spreads out over $\Z$ in $X_{\ns}^+ 
(N)^2$ and how it specializes to the special fiber $\left(X_{\ns}^+ 
(N)^2\right)_{/\overline \F_N}$.

\begin{cor}
\label{GaloisOnDT_n}
As in Lemma~\ref{SuppDT}, for a prime number $\ell \not \equiv 0,\pm1 \pmod N$, with $\ell<N^2/4$, we write $\Delta^* (T_\ell)$ as a sum of points $(E, C^+)$ with $E=E_\tau$ some CM curves over $\overline{\Q}$ endowed with normalizers of non-split 
Cartan structures $C^+$. For each $E$, we denote by ${\mathrm{End}}(E)$ the endomorphism ring (over $\overline{\Q}$). This ring contains an element $\alpha$ of degree  $\ell$. 
We fix a place ${\mathcal{N}} \vert N$ in~$\Q (j(\tau ))$. 
In the following statement, ``$a\equiv b$'' 
shall always mean ``$a\equiv b\bmod \mathcal{N}$''. 

\begin{enumerate}
\item Assume $N$ is inert in $\Q (\alpha )$. The only point in ${\mathrm{Supp}} (\Delta^*T_\ell )$ with associated  curve $E_\tau$ having values in $\Q (j(\tau ))$ specializes to the smooth part of the only multiplicity-$1$ component of the $j_i \equiv j(\tau )$ supersingular branch of $X_{\ns}^+ (N)_{{\overline{\F}}_N}$ (labeled $F_i$ or $H_0$ in Figure~\ref{figureNSgrossier+}). 
\item Assume $N$ is inert in $\Q (\alpha )$, ${\mathrm{Tr}}(\alpha )= 0$, and $j\not\equiv 0, 1728\bmod 
\mathcal{N}$. The $(N+1)/2$ other points in $\Delta^*T_\ell$ with that same underlying elliptic curve $E_\tau$ (the orbit of type $\mathbf{E}$ at ${\cal N}$, cf.~Lemma~\ref{SuppDT}) have values in some Galois extension of $\Q (j(\tau ))$, totally ramified above $\mathcal{N}$,  of degree $(N+1)/2$ (described in Theorem 3.8 of 
\cite{EdPar_regular}). Galois acts transitively on these $(N+1)/2$ points, and they all specialize to the same geometric point in the (only) component $E_i$, $i>0$, of the $j_i \equiv j(\tau )$ branch of $X_{\ns}^+ (N)_{{\overline{\F}}_{\mathcal{N}}}$ (which has multiplicity $(N+1)/2$).
	
\item Assume $N$ is inert in $\Q (\alpha )$, ${\mathrm{Tr}}(\alpha )= 0$, and $j\equiv 0$ or $1728\bmod 
\mathcal{N}$. The $(N+1)/2$ or $(N+1)/6$ other points in $\Delta^*T_\ell$ with underlying elliptic curve $E_\tau$ (orbits of type ${\mathbf{E}}$ or $\mathbf{E}_0$) have values in some Galois extension of $\Q (j(\tau ))$, totally ramified above $\mathcal{N}$, of degree $(N+1)/2$ or $(N+1)/6$. As in the previous case, Galois acts transitively on these points, and they all specialize to the same geometric point in the (only) component $E_0$ (with multiplicity $(N+1)/6$) if 
$j(\tau )\equiv 0$, or $G_0$ (with multiplicity $(N+1)/2$) if $j(\tau )\equiv 1728$.

\item Assume $N$ is split in $\Q (\alpha )$, so that (up to replacing $\alpha$ by $\alpha u$ for a suitable unit $u$) we have ${\mathrm{Tr}}(\alpha )= 0$. There are $(N-1)/2$ or $(N-1)/6$ points in $\Delta^*T_\ell$ with underlying elliptic curve $E_\tau$ (the ``small orbits'', cf.~Lemma~\ref{SuppDT}). These points all specialize to the same geometric  
point in the central vertical 
component $A$ of $X_{\ns}^+ (N)_{{\overline{\F}}_{\mathcal{N}}}$ (which is a $j$-line, with multiplicity $(N-1)/2$), 
unless $j(\tau )\equiv 0$, in which case they specialize to the component labeled $\mathcal{E}_0$ in Figure~\ref{figureNSgrossier+} (which has multiplicity $(N-1)/6$).

\end{enumerate}
\end{cor}

\begin{proof} 
Recall that $\End(E)$ contains some
$\alpha$ of degree~$\ell$, that acts on $E[N]$ as an element of the normalizer of a non-split Cartan. The 
inequality $\ell < N^2/4$ implies that $\alpha$ acts as a non-scalar element mod $N$. 
In particular, $N$ is not ramified in 
$\Z [\alpha ]$ (otherwise $\alpha$ 
would span a Borel subgroup in 
$\End(E [N])$, not compatible with the non-split Cartan image), and hence in $\Q(j(\tau))$ (recall that $\Q(j(\tau), \alpha)$ is the ring 
class field of $\Q(\alpha)$ with same conductor as $\End(E)$). Moreover, e.g. by \cite{Melistas}, up to twists we can suppose that $E/\Q(j(\tau))$ has good reduction at ${\cal N}$. Now
combine Lemma~\ref{SuppDT} with Propositions~\ref{GaloisOrbitsSS} and \ref{GaloisOrbitsOrd}.
\end{proof}

Writing a closed formula for $\Delta^*T_\ell 
={\Delta}^* (T_\ell  )$ would require listing all quadratic imaginary
integers with norm $\ell$. To get a sense of what this looks like, we give the following explicit statement for $\ell=2, 3$.
Section~\ref{Bfor19} illustrates an application of this calculation.

\begin{example}\label{examplesT_N}
Let $N$ be a prime. For an imaginary quadratic integer $\tau\in \C$, write $E_\tau\coloneq\C/(\Z\oplus \Z\tau)$. If $E/\C$ is a CM elliptic curve and
$\beta\in \End(E)$ has degree  prime to $N$, write $\beta \bmod N$ for its action on $E[N]$ and set
\[
\mathcal C_{\ns}^+(E,\beta)\coloneq
\left\{
C^+\mid \begin{array}{c}
C\subset \GL(E[N]) \text{ is a non-split Cartan} \\
\text{and }(\beta \bmod N)\in C^+\coloneq N_{\GL(E[N])}(C)
\end{array}
\right\}.
\]
When $\F_N[\beta \bmod N]^\times$ is itself a non-split Cartan, we write
$C_{\ns}^+(E,\beta)$ for its normalizer.
\begin{itemize}
    \item For $\ell=2$ and $N\ge 11$, we have
\begin{align*}
{\Delta}^*(T_2)
={}&
\frac{1}{2}\left(1-\Legendre{-1}{N}\right)
\bigl(E_i,C_{\ns}^+(E_i,i)\bigr)
+\sum_{C^+\in \mathcal C_{\ns}^+(E_{i\sqrt{2}},i\sqrt{2})}
\bigl(E_{i\sqrt{2}},C^+\bigr)  \\
&+
\left(1-\Legendre{-7}{N}\right)
\bigl(E_{(1+i\sqrt{7})/2},
C_{\ns}^+(E_{(1+i\sqrt{7})/2},(1+i\sqrt{7})/2)\bigr).
\end{align*}
Here the sum has
$\frac{N-1}{2}+\left(1-\Legendre{-2}{N}\right)$ summands.

\item For $\ell=3$ and $N\ge 13$, we have
\begin{align*}
{\Delta}^*(T_3)
={}&
\left(1-\Legendre{-2}{N}\right)
\bigl(E_{i\sqrt{2}},C_{\ns}^+(E_{i\sqrt{2}},i\sqrt{2})\bigr)
+
\left(1-\Legendre{-11}{N}\right)
\bigl(E_{\tau_{11}},
C_{\ns}^+(E_{\tau_{11}}, \tau_{11}) \bigr) \\
&+
\sum_{C^+\in \mathcal C_{\ns}^+(E_{(1+i\sqrt{3})/2},(1+i\sqrt{3})/2)}
\bigl(E_{(1+i\sqrt{3})/2},C^+\bigr)
+
\sum_{C^+\in \mathcal C_{\ns}^+(E_{i\sqrt{3}},i\sqrt{3})}
\bigl(E_{i\sqrt{3}},C^+\bigr),
\end{align*}
where $\tau_{11} \coloneq \frac{1+i\sqrt{11}}{2}$. If $N$ splits in $\Q(\zeta_3)$, the last two sums have respectively $\frac{N-1}{6}$ and $\frac{N-1}{2}$ summands. If $N$ is inert in $\Q(\zeta_3)$, they have respectively $\frac{N+7}{6}$ and $\frac{N+3}{2}$ summands.
\end{itemize}

\end{example}

\begin{proof}
For $\ell=2$, the imaginary quadratic integers of norm $2$ are, up to sign,
\[
1\pm i,\qquad i\sqrt{2},\qquad \frac{1\pm i\sqrt{7}}{2}.
\]
The corresponding endomorphism rings are $\Z[i]$, $\Z[i\sqrt{2}]$, and $\Z[(1+i\sqrt{7})/2]$, and the $j$-invariants are $1728$, $2^6 5^3$, and $-3^3 5^3$. The formula follows from Lemmas~\ref{SuppDT} and~\ref{MultDT}, noting that $\F_N[1+i]=\F_N[i]$.
The case $\ell=3$ is very similar. The imaginary quadratic integers of norm $3$ are, up to sign,
\[
1\pm i\sqrt{2},\qquad \frac{1\pm i\sqrt{11}}{2},\qquad
\zeta_3^k i\sqrt{3}.
\]
They correspond to four isomorphism classes of elliptic curves, having endomorphism rings
$\Z[i\sqrt{2}]$, $\Z[(1+i\sqrt{11})/2]$,
$\Z[(1+i\sqrt{3})/2]$, and $\Z[i\sqrt{3}]$, and $j$-invariants $2^6 5^3$, $-2^{15}$, $0$, and $2^4 3^3 5^3$, respectively. Note that $i\sqrt{3}$ is contained in the endomorphism ring of two non-isomorphic CM curves.
\end{proof}

We close this section with a proposition that establishes that, in the smooth locus of our modular curve, the graph of relevant Hecke operators $T_\ell$ has no component which is a copy of the diagonal. This can be proved by rather elementary means, roughly speaking by interpreting $X^+_\ns(N)^{\mathrm{sm}}$ as parameterizing first order deformations of supersingular elliptic curves (thanks to Corollary~\ref{cor:parameter_F_j}), and
ultimately reducing to analytic computations establishing that no component of the Hecke divisor $T_\ell$ contains the diagonal {\em in the generic fiber} of the curve. We can however give a more compact proof, using more sophisticated tools provided by Love, Studnia and Vonk in \cite{LSV26}. We thank them for sharing their work with us, and Studnia in particular for suggesting the following argument.

\begin{prop}\label{prop:flatness} 
Let $N \ge 13$ be a prime number and consider an integer $\ell>1$ smaller than $N^2 /4$ and such that $(\ell, N)=1$. Let $X$ be the regular model of $X^+_\ns(N)$ presented in Figure~\ref{figureNSgrossier+} and let $\ol{T_\ell} \subset X \times X$ be the closure of the $\ell$th Hecke correspondence in $X_\Q \times X_\Q$.  
Then the divisor $\Delta^* \ol{T_\ell}$ of $X$ does not contain any of the smooth components  of $X_{\ol \F_N}$.
\end{prop}
\begin{proof}

To prove that a smooth component $F$ of $X_{\ol \F_N}$ is not contained in $\Delta^*\overline{T_\ell}$, it is enough
to show that only finitely many points $\bar P\in F(\overline{\mathbb F}_N)$
satisfy $(\bar P,\bar P)\in\overline{T_\ell}(\ol \F_N)$. 

We first explain why such 
points are reductions modulo the maximal ideal of $\overline{\Z}_N$ of points $(P_1, P_2) \in \overline{T}_\ell(\overline{\Z}_N)$. By \cite[Chapter 10, Proposition 1.36]{Liu02}, the image of $(\ol P, \ol P)$ is a (set-theoretic) closed point in $(\ol{T_\ell})_{\F_N}$ that lies in the closure of some (set-theoretic) closed point $x$ in $T_\ell$. Since $T_\ell$ is of finite type over $\Q_N$, the point $x$ is the image of some point $y\in T_\ell(\ol{\Q_N})$. Since $\ol{T_\ell}$ is a closed subscheme of $X_\ns^+(N)^{\rm reg} \times X_\ns^+(N)^{\rm reg}$, it is proper over $\Z_N$, hence $y$ extends to a point $(P_1, P_2) \in \ol{T_\ell}(\ol{\Z_N})$.

Thus, to prove the proposition it suffices to show that there are only finitely many points $\ol P \in F(\ol \F_N)$ for which there exists a pair $(P_1, P_2) \in \ol{T}_\ell(\ol{\Z}_N)$ that reduces to $(\ol P, \ol P)$ modulo the maximal ideal of $\ol{\Z}_N$. Since $F$ is a smooth component containing only one singular point of $X$,
we may assume without loss of generality that we only work with smooth points of $X$.
Consider a pair $(P_1, P_2)$ that reduces to a point of the form $(\ol P, \ol P)$.
Since we may assume that $(\ol P, \ol P)$ is a smooth point, using the interpretation in \cite[Theorem A and Section 2]{LSV26} of the smooth locus of $X$, for $i=1, 2$ we can write $P_i = (E_i,C_{\ns,i}^+(N))$ for elliptic curves $E_i/\ol{\Z}_N$ together with Cartan structures  $C_{\ns,i}^+(N) \colon \F_{N^2} \hookrightarrow \End_{\ol{\Z}_N}(E_i[N])$ (up to precomposing by the conjugation in $\F_{N^2}$). Since $P_1$ and $P_2$ reduce to the same point $\ol P$ modulo the maximal ideal $\cal N$ of $\ol{\Z}_N$,
both $E_1 $ and $E_2$ reduce to a fixed supersingular elliptic curve $\ol E/\ol\F_N$ with $j$-invariant determined by $F$, 
and the structures $C_{\ns,i}^+(N)$ both reduce to the same structure $\ol C$ on $\ol E[N]$, up to $\Aut(\ol E)$.  
Moreover, by definition of $T_\ell$, there is a cyclic isogeny $\alpha \colon E_1 \to E_2$ of degree~$\ell$ 
sending $C_{\ns,1}^+(N)$ to $C_{\ns,2}^+(N)$, whose reduction is a  degree~$\ell$ endomorphism $\ol \alpha \colon \ol E  \to \ol E$ which preserves $\ol C$ (up to automorphisms).  Notice that in $\End(\ol E)$ there are a finite number of degree~$\ell$ endomorphism $\ol \alpha $, hence it suffices to prove that, for each $\ol \alpha$, there are a finite number of Cartan structures preserved by $\ol \alpha$.

Now to be able to compute with our objects, we need to say more about $\End_{\ol{\F}_N}(\ol E [N])$. 
Recall from \cite[Section 2]{LSV26} that $\End_{\ol{\F}_N}(\ol  E [N])$ sits in an exact 
sequence 
\begin{equation}
\label{ESEnd}    
0\to \End_{\ol{\F}_N}(\alpha_N )\simeq \A^1 ({\ol{\F}_N} ) \to \End_{\ol{\F}_N}(\ol  E [N]) \to O_E /{\cal P}_E \to 0
\end{equation}
where the left-hand $\alpha_N$ is the kernel of Frobenius $\ol E\to \ol E^{(N)}$ and
$O_E =\End_{\ol{\F}_N}(\ol E)$ is a maximal order in the quaternion algebra over $\Q$ ramified at $N$ and $\infty$, 
and ${\cal P}_E$ is its maximal two-sided ideal above $N$ (so $O_E /{\cal P}_E \simeq \F_{N^2}$). Moreover, using 
notations as in Lemma 2.5 of \cite{LSV26}, we know that $\ol E[N]/{\ol{\F}_N}$ can be written as an extension of ${\ol{\F}_N}$-group schemes:
\[
0\to \alpha_N \stackrel{\iota}{\to} \ol E[N] \stackrel{\pi}{\to} \alpha_N \to 0.
\]
Define $\nu \colon \A^1 \to 
\End_{\ol{\F}_N}(\ol  E[N])$ by $\nu (t)=\iota \circ t \circ \pi$. Write $\Z_{N^2} =\Z_N [\beta ]$, and let $\beta_0$ in
$\End_{\ol{\F}_N} (\ol E )$ such that $\Z_N [\beta_0 ]\simeq \Z_{N^2}$. It follows from \cite{LSV26}, Section 2, that elements $t \in \A^1 ({\ol{\F}_N})$ correspond to
closed points in the smooth component of $X_{\ol{\F}_N}$, i.e.~Cartan structures
for $\ol E$, by associating to $t$ the $\F_N$-linear map
induced by
\[
\ol C_t \colon \beta \bmod N \mapsto \beta_t \coloneq\beta_0 +\nu (t) \bmod N \in \End_{\ol{\F}_N} (\ol E[N]).
\]
If $a$ denotes the trace of 
$\beta_0$, we define
$$
\tilde{\beta}_t \coloneq \ol C_t (\beta^N )= \ol C_t (a-\beta )=a-\beta_0 -\nu (t) =\beta_0^N -\nu (t).
$$
Coming back to the proof of the present proposition, recall we can assume there is a  degree~$\ell$ endomorphism $\ol \alpha \colon \ol E  \to \ol E$ which preserves $\ol C$ (up to automorphisms), and we want to show there are only finitely many other $\ol C_t$ which are also preserved by $\ol \alpha$. 
  This condition means
that there exists an automorphism $u$ of $\ol E$ such that either $\ol \alpha \beta_t = u\beta_t u^{-1} \ol\alpha$ or $\ol \alpha \beta_t = u\tilde{\beta}_t u^{-1} \ol\alpha$ in $\End_{\ol{\F}_N}(\ol E [N])$. Setting $\alpha_1\coloneq u^{-1} \ol\alpha \in \End_{\ol{\F}_N}(\ol E)$, we obtain %
\begin{equation}\label{alpha_1NC}
    \alpha_1 \beta_t = \beta_t \alpha_1
\hspace{0.7cm} \mathrm{or} 
\hspace{0.7cm} 
\alpha_1 \beta_t = \tilde{\beta}_t \alpha_1
\end{equation}
as elements of $\End_{\ol{\F}_N}(\ol E [N])$. Therefore $(\ol E, C_t)$ will be normalized by $\alpha_1$ if and only if
\[
\alpha_1 (\beta_0  +\nu (t) )=
(\beta_0 +\nu (t) )\alpha_1
\hspace{0.5cm} \mathrm{or} 
\hspace{0.5cm} 
\alpha_1 (\beta_0  +\nu (t) )=
(\beta_0^N -\nu (t) ) \alpha_1
\]
in $\End_{\ol{\F}_N} (\ol E[N])$, that is, respectively,

\begin{eqnarray}
\label{alpha_1C}
\alpha_1 \nu (t)-\nu (t) \alpha_1 =\beta_0 \alpha_1 -
\alpha_1 \beta_0 
\end{eqnarray}
or 
\begin{eqnarray}
\label{alpha_1CN} 
\alpha_1 \nu (t)+\nu (t ) \alpha_1 =\beta_0^N \alpha_1 -\alpha_1 \beta_0 . 
\end{eqnarray}
Since there are only finitely many automorphisms $u$, it suffices to show that each of these equations has finitely many solutions in $t$, for fixed $\alpha_1$.
We deal with  (\ref{alpha_1CN}) first. 
Projecting it to $O_E/\mathcal{P}_E \cong \F_{N^2}$ 
by using  (\ref{ESEnd}), we see that it implies $\alpha_1 \beta_0 =\beta_0^N \alpha_1$. As $\beta_0 \mod N$ belongs 
to $\F_{N^2} \setminus \F_N$ 
by construction, that would imply $\alpha_1 \equiv 0 \pmod{\mathcal{P}_E}$. Since the norm of every element in $\mathcal{P}_E$ is divisible by $N$, this contradicts our numerical assumption that 
$\alpha$, therefore $\alpha_1$, has prime-to-$N$ norm $\ell <N^2 /4$. 

Now consider (\ref{alpha_1C}). 
Using Lemma~2.5 of \cite{LSV26}, it implies that
\begin{equation}
\label{alpha_1Cbis}    
\nu \left( (\Lie (\alpha_1 )  -\Lie (\alpha_1 )^N )t\right) =
\beta_0 \alpha_1 -\alpha_1 \beta_0 .
\end{equation}

As $\iota \circ \Lie (\alpha_1 ) =
\alpha_1 \circ \iota$, we know that $\Lie (\alpha_1 )\in {\ol{\F}_N}$ is a root of the minimal equation of 
$\alpha_1 \in \Z_{N^2} /N\Z_{N^2}$ over $\F_N$. In particular, $\Lie (\alpha_1 )$  cannot be $0$ in ${\ol{\F}_N}$ by our condition $\mathrm{Norm} 
(\alpha_1 )=\ell \not \equiv 0 \bmod N$. Recalling that $\nu$ is injective \cite[Proposition 2.2]{LSV26},  (\ref{alpha_1Cbis}) is linear in $t$, so it can only have infinitely many solutions if both the coefficient of $t$ and the constant coefficient vanish, that is, $\Lie(\alpha_1)=\Lie(\alpha_1)^N$ and $\beta_0\alpha_1=\alpha_1\beta_0 \bmod N$. 
The second equality shows that the class of $\alpha_1$ modulo $N$ centralizes $\beta_0$. The centralizer of
$\beta_0$ in $O_E/NO_E$ is $\F_N[\beta_0]\simeq\F_{N^2}$ (an explicit description of $O_E/NO_E$ can be given using the fact that $\calO_E \otimes \Z_N$  is the maximal $\Z_N$-order inside the unique ramified quaternion algebra over $\Z_N$, explicitly described e.g.~in \cite[Theorem 13.1.6]{VoightQuaternion}), so we obtain $\alpha_1\bmod N\in\F_N[\beta_0]$. On the other hand, $\Lie(\alpha_1)=\Lie(\alpha_1)^N$ implies that $\alpha_1 \bmod N$ is fixed by the Frobenius automorphism of $\F_N[\beta_0] / \F_N$, so $\alpha_1 \bmod N$ is in $\F_N$. 
In other words, there exists $a \in \mathbb{Z}$ and $\tau \in O_E$ such that $\alpha_1=a + N\tau$. Now note that $\tau$ is a quadratic imaginary integer, because it is contained in a quaternion algebra that is non-split at infinity, and that $\tau \neq 0$, because $\alpha_1$ corresponds to a cyclic isogeny.
It therefore follows 
that the norm $\ell$ of $\alpha_1$ is at least $N^2/4$, contradicting our
assumption.
\end{proof}

\subsection{Computing the vertical divisor \texorpdfstring{$B$}{B} of \texorpdfstring{\eqref{eq:L_is_divisorial}}{(4.2.5)}} \label{sec:V}
Our objective in this section is to compute the divisor $B$ appearing in \eqref{eq:L_is_divisorial}, but this requires some preparations.
It will be useful to employ the following notation. For a divisor $Z$ on $X_{\ns}^+(N)^{\mathrm{reg}}$, write $\operatorname{mdeg}_{N}(Z)$ for the multidegree of its specialization at $N$, viewed as an element of the free abelian group generated by the irreducible components of the special fiber, that is: \[
\operatorname{mdeg}_N(Z)
   \coloneq
   \sum_C  \delta_C  \, [Z, C] \cdot C \in \bigoplus_C (\mathbb{Z} \cdot C)
\]
where $C$ runs through the set of irreducible components of 
$X_{\ns}^+(N)^{\mathrm{reg}}_{/{\ol \F}_N}$ and $\delta_C$ denotes the multiplicity of $C$. See e.g.~\cite{Chin86} for generalities about divisors intersections. Here and below, a divisor on the generic fiber of our curve shall be identified with its Zariski closure in the regular model. As usual we extend those constructions by linearity to divisors with coefficients in $\Q$.

We start by solving the following problem. Given a degree $0$ divisor ${\cal D}$ on $X_{\ns}^+ (N)_{/{\Q_N^{\unr}}}$  
with a prescribed multidegree $\operatorname{mdeg}_N(\calD)$, we wish to compute a vertical divisor 
\[
\Phi_{\cal D}=\sum_C a_C({\cal D})C,
\]
where $C$ runs through the irreducible components of the special fiber of $X$, such that ${\cal D}+\Phi_{\cal D}$ has multidegree $0$ (that is, $({\cal D}+\Phi_{\cal D})\cdot C=0$ for every irreducible component $C$ of the special fiber). We keep labeling all components of the special fiber as in Figure~\ref{figureNSgrossier+}.

In order to compute the $\Phi_{\cal D}$, we could display the intersection matrix of our regular models for all four possibilities of $N\bmod 12$. The necessary information can be found in~\cite{EdPar_regular}. However, it is faster to proceed in the following way. Given the shape of special fibers for our arithmetic surfaces, we can reduce to the cases where 
\[ 
\calD = {\cal D}_\Gamma \coloneq \frac{1}{\delta_\Gamma}\Delta_\Gamma -\frac{2}{N-1}\Delta_{\mathrm{cusp}},
\]
where $\Gamma$ is an irreducible component of multiplicity $\delta_\Gamma$, $\Delta_\Gamma$ is a horizontal divisor of degree $\delta_\Gamma$ specializing to $\Gamma$, and $\Delta_{\mathrm{cusp}}$ is a divisor of degree $(N-1)/2$ specializing to the central Igusa component $A$ of multiplicity $(N-1)/2$ (and away from $j\equiv 0\bmod N$, if this case may occur; we may take for instance a cuspidal divisor). 

One checks that, in the special fiber above $N$ of our regular models, the horizontal divisor $\Delta_\Gamma$ satisfies $[\Delta_\Gamma, \Gamma]=1$.
Indeed, the intersection number can be computed as the length of the quotient of the completed local ring $\left(\Z_N^{\unr}[Y]/(Y^{\delta_\Gamma}-N)\right)[[X]]$ by the ideal generated by the two equations $X-Y$ and $Y$, which is easily seen to be $1$. By taking integral linear combinations of the $\calD_\Gamma$ we can therefore obtain horizontal divisors of degree $0$ whose specializations realize any given multidegree of total degree $0$, as we have $\operatorname{mdeg}_N(\calD_\Gamma) = \Gamma - A$, and these multidegrees clearly generate all multidegrees of total degree $0$.

Therefore, all cases boil down to the following possibilities for $\Gamma$:
\begin{itemize} 
\item $\Gamma=D_i$, $E_i$, or $F_i$ for some $i>0$, with $\delta_\Gamma=N+1$, $(N+1)/2$, or $1$ respectively; these are the components above a generic supersingular $j$-invariant in characteristic $N$; 
\item $\Gamma=D_0$, $E_0$, or $F_0$, with $\delta_\Gamma=(N+1)/3$, $(N+1)/6$, or $1$ respectively; these are the components above the supersingular $j$-invariant $0$ in characteristic $N\equiv -1\bmod 3$; 
\item $\Gamma=G_0$ or $H_0$, with $\delta_\Gamma=(N+1)/2$ or $1$ respectively; these are the components above the supersingular $j$-invariant $1728$ in characteristic $N\equiv -1\bmod 4$; 
\item $\Gamma={\cal D}_0$ or ${\cal E}_0$, with $\delta_\Gamma=(N-1)/3$ or $(N-1)/6$ respectively; these are the components above the ordinary $j$-invariant $0$ in characteristic $N\equiv 1\bmod 3$.
\end{itemize} 
By Propositions~\ref{GaloisOrbitsSS} and~\ref{GaloisOrbitsOrd}, each $\Gamma$ listed above is the specialization of (a rational multiple of) a suitable CM divisor on $X(\overline{\Q}_N)$. In each case, the integer $\delta_\Gamma$ is the multiplicity of the component $\Gamma$, namely one of $1$, $(N-1)/6$, $(N+1)/6$, $(N-1)/3$, $(N+1)/3$, $(N+1)/2$, or $N+1$.

It also follows from Corollary~\ref{GaloisOnDT_n} that when we consider the specialization of the divisors $\Delta^*(T_\ell)$, namely the intersections of the diagonal with Hecke correspondences, the supersingular components $D_i$ do not occur. However, for completeness, we consider all components $\Gamma$ of the special fiber. The results are listed in Table~\ref{table:phi_Delta}: for each irreducible component $\Gamma$, we give a vertical $\Q$-divisor $\Phi_\Gamma$ such that $\calD_\Gamma + \Phi_\Gamma$ is of multidegree $0$.

\begin{table}
\centering
\begin{tabular}{|c|c||c |c |}
\hline
${\Gamma} $ & $\Phi_{\Gamma} $ & ${\Gamma} $ & $\Phi_{\Gamma} $   \\
\hline
 & & &\\ 
$D_i, i>0$  & $\frac{2}{N-1} \, D_i +\frac{1}{N-1} \, E_i + \frac{2}{N^2 -1}\,  F_i$  &
$D_0$  & $\frac{2}{N-1} \, D_0 +\frac{1}{N-1} \, E_0 + \frac{6}{N^2 -1} \, F_0$     \\
 & & & \\
$E_i, i>0$  & $\frac{2}{N-1} \, D_i +\frac{2N}{N^2 -1} \, E_i + \frac{2}{N^2 -1} \, F_i$  &
$E_0$  & $\frac{2}{N-1} \, D_0 +\frac{4N-2}{N^2 -1} \, E_0 + \frac{6}{N^2 -1} \, F_0$     \\
 & & &\\
$F_i, i>0$  & $\frac{2}{N-1} \, D_i +\frac{1}{N -1} \, E_i + \frac{1}{N -1} \, F_i$    &
$F_0$  & $\frac{2}{N-1} \, D_0 +\frac{1}{N -1} \, E_0 + \frac{3}{N -1} \, F_0$    \\
 & & & \\
\hline
 & & & \\ 
$G_0$  & $\frac{2}{N-1} \, G_0 + \frac{4}{N^2 -1} \, H_0$ &
${\cal D}_0$  & $\frac{2}{N-1} \, {\cal D}_0 + \frac{1}{N -1} \, {\cal E}_0$   \\
 & & & \\
$H_0$  & $\frac{2}{N-1} \, G_0 +\frac{2}{N -1} \, H_0$  &
${\cal E}_0$  & $\frac{2}{N-1} \, {\cal D}_0 +\frac{4}{N -1} \, {\cal E}_0$ \\
 & & & \\
\hline
\end{tabular}
\caption{Values for $\Phi_{\Gamma}$}
\label{table:phi_Delta}
\end{table}

\begin{example}
For the comfort of the reader, we detail the computations for the first case of Table~\ref{table:phi_Delta}. We use the intersection numbers described in Remark~\ref{rmk: intersection matrix}. 

Let ${\cal D}_{D_i}\coloneq \frac{1}{N+1}\Delta_{D_i} -\frac{2}{N-1}\Delta_{\mathrm{cusp}}$, where $\Delta_{D_i}$ is a horizontal divisor of degree $N+1$ specializing to $D_i$. 
We want to compute a vertical divisor $\Phi_{D_i}$ such that $[{\cal D}_{D_i}+\Phi_{D_i},C]=0$ for every irreducible component $C$ of the special fiber of $X_{\ns}^+(N)^{\mathrm{reg}}$. The divisor $\Phi_{D_i}$ is defined up to addition of a multiple of the full special fiber, and it shall ease the computation to fix the coefficient of $A$ to be $a_A=0$. One then checks that all coefficients $a_C$ are $0$ except possibly for $C=D_i,E_i,F_i$, so that \[ \Phi_{D_i}=a_{D_i}D_i+a_{E_i}E_i+a_{F_i}F_i. \] We compute $[D_i,D_i]=-1$, $[E_i,E_i]=-2$, and $[F_i,F_i]=-(N+1)$. Intersecting with $A$ gives $[{\cal D}_{D_i}+\Phi_{D_i},A]=a_{D_i}-\frac{2}{N-1}=0$; similarly, intersecting with $D_i$, $E_i$, and $F_i$ we get
\[
-a_{D_i} +a_{E_i}  +a_{F_i} +\frac{1}{N+1} =0, \ a_{D_i} -2a_{E_i} =0, \hspace{0.4cm} {\mathrm{and}} 
\hspace{0.4cm} a_{D_i} -(N+1) a_{F_i} =0.  
\]
We finally have $a_{D_i}=\frac{2}{N-1}$, $a_{E_i}=\frac{1}{N-1}$, and $a_{F_i}=\frac{2}{N^2-1}$. 
\end{example}

Doing the math in the remaining cases, with the normalization $a_A=0$ throughout, we obtain Table~\ref{table:phi_Delta}. 
Note that in all cases the coefficients of $\Phi_{\Gamma}$ satisfy $a_*=O(1/N)$. 

Using Table~\ref{table:phi_Delta}, we can finally compute the divisor $B$ in \eqref{eq:L_is_divisorial}. 
Recall that this divisor arises in the following context: $X$ is our regular model $X_{\ns}^+(N)^{\mathrm{reg}}$ over $\Z$, endowed with a $\Z$-section $b$; we have fixed a simple open $U$ and a trace zero endomorphism $f$ of $J$, given by a correspondence $T$. We assume that $T$ is of the form $T=m(a_1T_\ell+a_2T_n)$ for certain Hecke operators $T_\ell$ and $T_n$ with $\ell, n$ prime to $N$. Then $B$ is the unique divisor satisfying the following three conditions: \begin{enumerate} 
\item \label{item:B1} it is supported on the special fiber $X_{\F_N}$ of the regular model of $X$ in Figure~\ref{figureNSgrossier+}; 
\item \label{item:B2} the coefficient of the component meeting the point $b\in X(\Z)$ is $0$;
\item \label{item:B3} for each $u\in U(\Z_N^{\unr})$, the divisor \begin{equation}\label{eq:condition_B_modular} 
T\cap (X\times u)+\ol{D_V}+B 
\end{equation} 
has multidegree $0$. 
\end{enumerate}
We can compute $B$ with the tools developed in this section, as follows. First, we determine the specialization of all points appearing in
\begin{equation}\label{eq:Dsec5} 
D\coloneq T\cap (X\times u)+\ol{D_V} = m\cdot \bigl(a_1T_\ell(u)+a_2T_n(u)+a_1T_\ell(b)+a_2T_n(b) -a_1\Delta^*T_\ell-a_2\Delta^*T_n \bigr).
\end{equation}
We explain how this can be done. 
As $b$ and $u$ are defined over $\Z_N^{\unr}$ and $\ell,n$ are prime to $N$, every point in $T_\ell(u)$, $T_n(u)$, $T_\ell(b)$, and $T_n(b)$ is defined over $\Z_N^{\unr}$. We can decide which component each point $(E,C^+)$ specializes to from the $j$-invariant of $E$ modulo $N$. Since $U$ is a simple open, the points $u\in U(\Z_N)$ must belong to a single $F_i$ or $H_0$, corresponding to a supersingular $j$-invariant $\ol j\in \ol{\F}_N$. Hence the components where the points in $T_\ell(u)$ and $T_n(u)$ specialize only depend on the curves which are $\ell$-isogenous or $n$-isogenous to the curve with $j$-invariant $\ol j$, which may be computed for example using the classical modular polynomials. To determine the points in $\Delta^*T_\ell, \Delta^*T_n$ we use Corollary~\ref{cor:delta_cap_Tell}; their specializations are then described by Propositions~\ref{GaloisOrbitsSS} and~\ref{GaloisOrbitsOrd}. 

Moreover, by Proposition~\ref{prop:flatness}, the special fiber $D_{\F_N}$ of $D$ at $N$ consists precisely of the specializations of the points appearing in \eqref{eq:Dsec5}. Thus, from the previous calculation we obtain an expression of $\operatorname{mdeg}_N(D)$ as a linear combination of the specialization types $\Gamma$ appearing in Table~\ref{table:phi_Delta}. We may then compute the corresponding linear combination of the $\Phi_\Gamma$. If we take $B$ to be this linear combination, conditions~\ref{item:B1} and \ref{item:B3} are satisfied. Finally, condition~\ref{item:B2} is also satisfied after adding a suitable multiple of the full fiber $X_{\F_N} = \operatorname{div}(N)$.

\subsection{An example: Vertical corrections for \texorpdfstring{$N=19$}{N=19}}
\label{Bfor19}

We conclude this section by making the above constructions explicit for $N=19$.
The curve $X_{\ns}^+ (19)$ has two simple opens $U_{H_0}, U_{F_1}$, corresponding to the components $H_0$ and $F_1$ of the special fiber at $19$. They are associated with the two supersingular invariants $j\equiv -1 \bmod 19$ and $j\equiv 7\bmod 19$, respectively.

Let $b\coloneq\left(E_i,C_{\ns}^+(E_i,i)\right)\in X_{\ns}^+(19)(\Z)$ be the Heegner point associated with the elliptic curve with complex multiplication by $\Z[i]$, endowed with the normalizer of the Cartan subgroup given by $(\Z[i]\otimes \F_{19})^\times\hookrightarrow \GL(E_i[19])$. Since $1728 \equiv -1 \bmod 19$, the point $b$ specializes to the component $H_0$. Since both Hecke operators $T_2$ and $T_3$ have trace $-3$, we can take the trace zero endomorphism $f=T_2-T_3$. The exponent of $(J_{\F_{19}} / J^0_{\F_{19}})(\ol \F_{19})$ is $m=19-1=18$. We consider the map $\jbtilde$ with base point $b$, on the simple open $U \coloneq U_{H_0}$, and corresponding to the map $f$.

Let
\[
D\coloneq18\bigl(T_2(u)-T_3(u)+T_2(b)-T_3(b)-\Delta^*(T_2)+\Delta^*(T_3)\bigr),
\]
as in \eqref{eq:Dsec5}. We now compute the points appearing in this divisor and the multidegrees of their specializations.
Set $\tau_d\coloneq(1+i\sqrt d)/2$ for $d=3,7,11$. Example~\ref{examplesT_N} gives
\begin{align*}
\Delta^*(T_2)
={}&
\bigl(E_i,C_{\ns}^+(E_i,i)\bigr)
+
\sum_{C^+\in \mathcal C_{\ns}^+(E_{i\sqrt 2},i\sqrt 2)}
\bigl(E_{i\sqrt 2},C^+\bigr)
+
2\bigl(E_{\tau_7},C_{\ns}^+(E_{\tau_7},\tau_7)\bigr),
\\
\Delta^*(T_3)
={}&
2\bigl(E_{\tau_{11}},C_{\ns}^+(E_{\tau_{11}},\tau_{11})\bigr)
+
\sum_{C^+\in \mathcal C_{\ns}^+(E_{\tau_3},\tau_3)}
\bigl(E_{\tau_3},C^+\bigr)
+
\sum_{C^+\in \mathcal C_{\ns}^+(E_{i\sqrt 3},i\sqrt 3)}
\bigl(E_{i\sqrt 3},C^+\bigr).
\end{align*}
The sum in $\Delta^*(T_2)$ has $9$ terms, while the two sums in $\Delta^*(T_3)$ have respectively $3$ and $9$ terms. The six CM $j$-invariants of the elliptic curves occurring in $\Delta^*(T_2), \Delta^*(T_3)$ reduce modulo $19$ as follows. The supersingular ones are
\[
j(i)=1728\equiv -1,\qquad
j(\tau_{11})=-2^{15}\equiv 7,\qquad
j(\tau_7)=-3^3\cdot 5^3\equiv 7
\]
and the ordinary ones are
\[
j(i\sqrt 2)=2^6\cdot 5^3\equiv 1,\qquad
j(\tau_3)=0,\qquad
j(i\sqrt 3)=2^4\cdot 3^3\cdot 5^3\equiv 2.
\]
All congruences are modulo $19$. Using the description of how CM points specialize (Proposition~\ref{GaloisOrbitsSS} and Corollary~\ref{GaloisOnDT_n}), we get
\[
\operatorname{mdeg}_{19}\bigl(\Delta^*(T_2)\bigr)
= H_0 +9 \cdot A +2 \cdot F_1
\qquad \text{and} \qquad
\operatorname{mdeg}_{19}\bigl(\Delta^*(T_3)\bigr)
=2 \cdot F_1 +3 \cdot \mathcal E_0 +9 \cdot A.
\]

We also need the multidegrees of $T_2(u)$ and $T_3(u)$ when $u$ specializes to
$H_0$ or to $F_1$. Since the relevant points are defined over $\Z_{19}^{\unr}$, they
specialize to components of multiplicity one; the component where they specialize is determined by the reduction modulo $19$ of the (supersingular) $j$-invariant of the underlying elliptic curve. The action of $T_2, T_3$ on supersingular $j$-invariants modulo $19$ can be computed using the classical modular polynomials $\Phi_2 (X,Y)$ and $\Phi_3 (X,Y)$. Representing by $(j=j_0)$ the class of the $j$-invariant $j_0$, we check that
\[
T_2 (j= -1) =(j= -1)+2(j=7), \qquad 
T_2 (j= 7) =(j= -1)+2(j= 7)
\]
and
\[
T_3 (j= -1) =4(j= 7), \qquad
T_3 (j= 7) =2(j= -1)+2(j= 7).
\]
Recalling that the points with $j$-invariant $-1$ specialize to $H_0$ and the points with $j$-invariant $7$ specialize to $F_1$, we obtain the following table:
\[
\begin{array}{c|cc}
 & u\leadsto H_0 & u\leadsto F_1\\
\hline
\operatorname{mdeg}_{19}\bigl(T_2(u)\bigr)
& H_0+2 \cdot F_1 & H_0+2 \cdot F_1\\
\operatorname{mdeg}_{19}\bigl(T_3(u)\bigr)
& 4 \cdot F_1 & 2 \cdot H_0 + 2 \cdot F_1
\end{array}
\]
In particular, as $b$ specializes to $H_0$, the first column also gives
$\operatorname{mdeg}_{19}(T_2(b))$ and $\operatorname{mdeg}_{19}(T_3(b))$.
By additivity of $\operatorname{mdeg}_{19}$,
\[
\operatorname{mdeg}_{19}(D)
=
18\,\operatorname{mdeg}_{19}\bigl(T_2(u)-T_3(u)\bigr)
+18\bigl(-2 \cdot F_1 + 3 \cdot \mathcal E_0\bigr).
\]
Therefore, if $u$ specializes to $H_0$, then
\[
\operatorname{mdeg}_{19}(D)
=
18\bigl(H_0-4 \cdot F_1+3 \cdot \mathcal E_0 \bigr),
\]
and if $u$ specializes to $F_1$, then
\[
\operatorname{mdeg}_{19}(D)
=
18\bigl(-H_0-2 \cdot F_1+3 \cdot \mathcal E_0\bigr).
\]

Recall that $\Phi_\Gamma$ denotes the vertical divisor corresponding to a divisor specializing to the component $\Gamma$, as in Table~\ref{table:phi_Delta}. For the simple open $U_{H_0}$, the non-normalized vertical correction is
\[
\widetilde B_{H_0}
=
18\bigl(\Phi_{H_0}-4 \cdot \Phi_{F_1}+3 \cdot  \Phi_{\mathcal E_0}\bigr)
=
2\cdot (G_0+H_0)-4\cdot (2\cdot D_1+E_1+F_1)+3\cdot (2\cdot \mathcal D_0+4\cdot \mathcal E_0).
\]
The full special fiber is $\operatorname{div}(19)
=
9 \cdot A+6 \cdot  \mathcal D_0+3 \cdot \mathcal E_0+20 \cdot D_1+10 \cdot E_1+F_1+10 \cdot G_0+H_0$.
Since $b$ specializes to $H_0$ and we want our divisor $B$ not to meet the point $b$, we normalize it by making the coefficient of $H_0$ equal to $0$. We obtain $B_{H_0}\coloneq\widetilde B_{H_0}-2\operatorname{div}(19)$,
that is,
\[
B_{H_0}=-18 \cdot A +6 \cdot 
{\cal E}_0 -6  \cdot {\cal D}_0 -18  \cdot G_0 -48 \cdot  D_1 
-24 \cdot  E_1  -6  \cdot F_1.
\]
The same calculation for the simple open $U_{F_1}$, with the same base point $b$ and the same endomorphism $f=T_2-T_3$, gives
\[
B_{F_1} = 18 \cdot A +18 \cdot 
{\cal E}_0 +18  \cdot {\cal D}_0 +18  \cdot G_0 +36\cdot  D_1 
+18 \cdot  E_1,
\]
which has no component along $F_1$ (nor, by construction, $H_0$).

\section{Model-free computations for non-split Cartan curves}\label{sect: model-free computations for Xnsp}

In the next sections, we explain how we compute the maps $\jbtilde$ and $\kappa$ introduced in Section \ref{sec:review-geometric-quadratic-chabauty} for non-split Cartan curves of prime level using the representation of Section \ref{sec:mumford_makdisi} and the model in Section \ref{sec:some_modular_theory}. 
We then explain how to intersect their images in analytic coordinates.

\begin{remark}\label{rmk:no_base_point_Xns} As in previous sections, we will assume that we know a rational point $b \in X_{\ns}^+(N)(\Q)$. For computationally accessible values of $N$, this is not a serious restriction, because CM points usually provide rational points on $X_{\ns}^+(N)$. There will be no such point only when $N$ is not inert in any of the $13$ class-number-one quadratic imaginary orders. This is the case only for the small (logarithmic) density $2^{-9}$ of $N$. The smallest such $N$ is $\numprint{15073}$.
\end{remark}

\begin{assumption}\label{assump:properties-of-p}
In the case of modular curves $X = X_\ns^+(N)$, we will restrict to primes  $p$ satisfying additional conditions. When these conditions are not met, or are too hard to check, we try again with a different auxiliary prime $p$. 
\begin{enumerate}
    \item We require that we can compute a subgroup $\Gamma$ of the Mordell--Weil group of $J$ as in Section \ref{subsec:computing-MW}, in particular,  that   $J^0(\Z)_{\ol 0}$ surjects onto $\ker(J(\Z/p^2\Z) \to J(\Z/p\Z))$, see  \eqref{eq: J0Z0 to JZp2 onto}.
    \item We require that $p >3$, and $p \nmid \deg\big( X(N) \longrightarrow  X_\ns^+(N)\big) = N+1$, and moreover that Makdisi's modular forms generate $M_2(\Gamma(N))$ modulo~$p$  
   so as to be able to construct a $p$-adic Makdisi model of the Jacobian (see Section \ref{sec:makdisi_modular_P}, especially  \eqref{eqn:p_ndiv_Gamma} and Remark \ref{rmk: generation Mak} there). 
    \item We also require that $p \not \equiv \pm 1 \pmod N$, so that the cusps are not $\F_p$-rational. This is not a theoretical obstruction, but in the current implementation we work with non-cuspidal points on the modular curves, represented as elliptic curves  with good reduction, with level~$N$ structure. For the same reason, we also assume that all the (finitely many) elliptic curves that we encounter have good reduction over $\Z_p^{\unr}$ (see also Remark \ref{remark representing points}).
\end{enumerate}

\end{assumption}

\subsection{Representing points on non-split Cartan curves}
\label{sec:represent_points}
 We begin by briefly discussing the computational problem of representing a point on $X_{\ns}^+(N)(\Z_q)$, where $\Z_q$ is the ring of integers of the finite unramified extension $\Q_q/\Q_p$ with residue field $\F_q$. We assume that $\F_q$ contains primitive $N$th roots of unity, and we fix a primitive $N$th root of unity $\mu$. We also identify $\mu$ with the unique $N$th root of unity in $\Z_q$ reducing to $\mu$ modulo the maximal ideal.

Recall from Section \ref{subsect: modular curves} that a point in $X_{\ns}^+(N)(\Q_q)$ is a pair $(E, [\phi])$, where $E/\Q_q$ is an elliptic curve and $[\phi]$ is an equivalence class of $C_{\ns}^+(N)$-level structures defined over $\Q_q$. 
We will only be interested in points corresponding to elliptic curves $E/\Z_q$ with good reduction.
We represent this concretely as a pair of pairs: $((a_4, a_6), (P,Q))$ in $\Z_q^2$ and $E[N]^2$ is described as follows. 
\begin{itemize}
\item Since $p > 3$, the curve $E$ can be represented in short Weierstrass form $y^2 = x^3 + a_4x + a_6$ with $a_4, a_6 \in \Z_q$ given to finite precision.
\item Giving the data of an isomorphism $\phi: E[N] \to (\Z/N\Z)^2$ is equivalent to giving the data of a pair of points $P, Q$ that generate $E[N]$.
Since $E/\Q_q$ has good reduction and hence $E[N]$ is finite étale over $\Z_q$, it is enough to choose $P, Q$ in $E(\F_q)$. 
This choice lifts uniquely to $E(\Z_q) = E(\Q_q)$. Note that this requires all points in $E[N]$ to be defined over $\F_q$. 
We may achieve this at the cost of a finite extension $\F_{q^r}$ of the ground field $\F_q$ (and we then replace $\Q_q$ with the corresponding unramified extension $\Q_{q^r}$). 
This representation essentially amounts to giving a point on $X(N)(\Q_q)$ lifting our given point $(E, [\phi]) \in X_{\ns}^+(N)(\Q_q)$. To ensure that we always work on the same connected component of $X(N)$ and use the techniques in \cite{NicMod}, we require that the Weil pairing of $P$ and $Q$ equals $\mu$, cf.\ Remark~\ref{rk:fixed_Weil}.
\end{itemize}

\begin{remark}\label{remark representing points}
    \phantom{~}
    \begin{enumerate}
        \item Note that we represent a $C_{\ns}^+(N)$-level structure by a full level structure. In practice, this means that our representation of a $\Q_p$-rational point of $X_{\ns}^+(N)$ is encoded by an elliptic curve with torsion basis over an extension $\Q_q$ of $\Q_p$ (we need the full $N$-torsion to be defined over $\Q_q$). 

        \item The requirement to work only with elliptic curves having good reduction over $\Z_p^{\unr}$ is merely a limitation of our current implementation, not a theoretical obstruction.
        
        \item It is not true that every elliptic curve $E/\Q_p$ with $j(E) \in \Z_p$ may be encoded in the above form, even up to twists. Namely, there exist elliptic curves $E/\Q_p$ with integral $j$-invariant none of whose twists have good reduction over $\Q_p^{\unr}$; this happens when  $j(E) \equiv 1728\pmod p$ and the valuation of $j(E)-1728$ is odd, or when $j(E) \equiv 0\pmod p$ and the valuation of $j(E)$ is not divisible by $3$.
    \end{enumerate}
\end{remark}

We will consider two kinds of points in $X_{\ns}^+(N)(\Q_q)$: lifts of $\F_p$-rational points, and images of Heegner points in $X_{\ns}^+(N)(F)$ where $F$ is a number field that embeds in $\Q_q$. In the next two sections, we describe how to compute the relevant equations for the elliptic curves and how to choose the points $P,Q \in E[N](\F_q)$ in these cases.

\subsection{Enumerating residue discs}
\label{sec:enumerate}
We now explain how to list all of the points in $X_{\ns}^+(N)(\F_p)$ without using equations. By the modular interpretation, a point in $X_{\ns}^+(N)(\F_p)$ corresponds to a pair $(E, [\phi])$, where $E$ is an elliptic curve over $\F_p$ and $\phi : E[N] \to  (\Z/N\Z)^2$ is an isomorphism, up to the equivalence relation \eqref{eq:rationalptsYH} (note that our assumption $p \not \equiv \pm 1 \pmod N$ implies that $X_{\ns}^+(N)(\F_p)$ contains no cuspidal points). Thus, to enumerate all points in $X_{\ns}^+(N)(\F_p)$ we may proceed as follows. 
\begin{enumerate}
    \item For each candidate $j$-invariant $j_0 \in \mathbb{F}_p$, construct $E_{j_0}/\F_p$ with $j(E_{j_0}) = j_0$ and compute its geometric automorphism group $\Aut(E_{j_0, \overline{ \F}_p})$ (which depends only on $j_0$).
    \item Loop over all $\Z/N\Z$-bases $(P, Q)$ of $E_{j_0}[N]$ satisfying the condition on the Weil pairing (that is, $e_N(P, Q)=\mu$, where $\mu$ is our fixed $N$th root of unity). Each such basis gives an isomorphism $\phi : E_{j_0}[N] \to (\Z/N\Z)^2$, and the point $(E_{j_0}, [\phi]) \in X_{\ns}^+(N)(\overline{\F}_p)$ is rational if and only if the points $(E_{j_0}, [\phi \circ \operatorname{Frob}_p])$ and $(E_{j_0}, [\phi])$ are equal
    on the modular curve. To check equality, we loop over all automorphisms $\alpha$ of $(E_{j_0})_{\overline\F_p}$ and test whether there exists $C \in C_{\ns}^+(N)$ such that
    \[
    \phi \circ \operatorname{Frob}_p = C \circ \phi \circ \alpha.
    \]
    \item For every pair of candidate structures $(E_{j_0}, [\phi_1])$, $(E_{j_0}, [\phi_2])$, we test whether they represent the same $\F_p$-point by checking equivalence of the structures $\phi_1, \phi_2$ up to automorphisms and $C_\ns^+(N)$, and remove duplicates. Note that pairs $(E, [\phi]), (E', [\phi'])$ with $j(E) \neq j(E')$ cannot represent the same point.
\end{enumerate}

Finally, for each residue disc $(E, [\phi])$, we (try to) find a lift $(\widetilde{E}, [\widetilde{\phi}]) \in X_{\ns}^+(N)(\Q_p)$ with the additional property that $\widetilde{E}$ has good reduction over $\Q_p$. Smoothness of $X_{\ns}^+(N)$ over $\Z_p$ ensures that every $(E, [\phi]) \in X_{\ns}^+(N)(\F_p)$ lifts to \textit{some} point $(\widetilde{E}, [\widetilde{\phi}]) \in X_{\ns}^+(N)(\Q_p)$, but there is no guarantee that the representative elliptic curve $\tilde{E}$ has good reduction. We will show that there always is such a lift, and that in most cases there are infinitely many.

We distinguish cases according to whether $j(E) \in \{0,1728\}$ or not.
\begin{enumerate}
    \item If $j(E) \not \in \{0, 1728\}$, then lifting with good reduction over $\Q_p$ is always possible. Fix any $\widetilde{E}/\Q_p$ with good reduction and special fiber $E$. 
Reduction mod $p$ establishes a Galois-equivariant bijection $\widetilde{E}[N] \cong E[N]$. Via this isomorphism, we may interpret $\phi$ as an isomorphism $\tilde\phi\colon \tilde{E}[N] \cong E[N] \to (\Z/N\Z)^2$. The pair $(\widetilde{E}, [\tilde{\phi}])$ is readily checked to be $\Q_p$-rational and a lift of $(E, [\phi])$. Since $\tilde{E}$ has good reduction, $\widetilde{E}[N]$ is an unramified representation of $\operatorname{Gal}(\overline{\Q}_p/\Q_p)$. In particular, $\operatorname{Gal}(\overline{\Q}_p/\Q_p)$ acts on $\phi$ via its quotient $\operatorname{Gal}(\Q_p^{\unr}/\Q_p)$, which is canonically isomorphic to $\operatorname{Gal}(\overline{\F}_p/\F_p)$ and hence (topologically) generated by the unramified Frobenius $\widetilde{\operatorname{Frob}}_p$. Thus, $(\widetilde{E}, [\widetilde{\phi}])$ is $\Q_p$-rational if and only if $(\widetilde{E}, [\widetilde{\phi}]) = (\widetilde{E}, [\widetilde\phi \circ \widetilde{\operatorname{Frob}}_p])$ as points in $X_{\ns}^+(N)(\overline{\Q}_p)$. 

    By $\F_p$-rationality of $(E, [\phi])$, there is an automorphism $\alpha$ of $E$ and a matrix $C \in C_{\ns}^+(N)$ such that
    \begin{equation}\label{eq: rationality over Fp}
            \phi \circ \operatorname{Frob}_p = C \circ \phi \circ \alpha.
    \end{equation}
    The crucial point is that both $E$ and $\widetilde{E}$ have the same geometric automorphism group, namely $\{ \pm \operatorname{Id} \}$. 
    Letting $\widetilde{\alpha}$ be the automorphism of $\widetilde{E}$ that lifts $\alpha$, we claim that we have
    \begin{equation}\label{eq: rationality over Qp}
    \widetilde{\phi} \circ \widetilde{\operatorname{Frob}}_p = C \circ \widetilde{\phi} \circ \widetilde{\alpha},
    \end{equation}
    which gives $\Q_p$-rationality of $(\widetilde{E}, [\widetilde{\phi}])$. Equation
     \eqref{eq: rationality over Qp} may be shown by reduction modulo $p$, which gives a bijection $\widetilde{E}[N] \to E[N]$ and identifies $\widetilde{\operatorname{Frob}}_p$ with $\operatorname{Frob}_p$, so that \eqref{eq: rationality over Qp} reduces to \eqref{eq: rationality over Fp}.

    \item If $j(E) \in \{0, 1728\}$, the same argument shows that there is a lift $(\widetilde{E}, [\widetilde{\phi}]) \in X_{\ns}^+(N)(\Q_p)$ where $\widetilde{E}$ has good reduction. Namely, it suffices to take $\widetilde{E}$ having $j$-invariant exactly equal to $0$ or $1728$. Again, the key point is that reduction modulo $p$ identifies $\operatorname{Aut}(\widetilde{E}_{\overline{\Q}_p})$ with $\operatorname{Aut}(E_{\overline{\F}_p})$. 
    Note that here we use $p>3$.
    
    However, if we merely take $j(\widetilde{E})$ to be congruent to $0$ (resp.~$1728$) modulo $p$, then the rationality condition for $(\widetilde{E}, [\widetilde{\phi}])$ is not necessarily satisfied. Indeed, the $\F_p$-rationality of $(E, [\phi])$ only guarantees that \eqref{eq: rationality over Fp} holds for some automorphism $\alpha \in \operatorname{Aut}(E_{\overline{\F}_p})$.
    If $\alpha$ is not $\pm \operatorname{Id}$ and $j(\tilde{E}) \not \in \{0, 1728\}$, then there is no automorphism $\widetilde{\alpha}$ lifting $\alpha$, so we cannot deduce $\Q_p$-rationality of $(\widetilde{E}, [\widetilde{\phi}])$. For each  $\beta \in \operatorname{Aut}(E_{\overline{\F}_p})$ we may test whether the normalizer of Cartan structure $\phi \circ \beta$ lifts to a $\Q_p$-rational one, as explained above.

    It is possible that no  structure $\phi \circ \beta$ lifts to $\Q_p$ if $j(\widetilde{E}) \not \in \{0, 1728\}$. However, if $\phi \circ \beta$ does lift, we may easily find infinitely many different lifts $(\widetilde{E}, [\widetilde{\phi}])$ of $(E, [\phi])$ with $\widetilde{E}$ having good reduction over $\Q_p$: it suffices to deform $p$-adically any given lift with these properties.  
\end{enumerate}

\begin{remark}
    Recall from Remark~\ref{remark representing points}(2) that not every point $(E,[\phi])\in X_{\ns}^+(N)(\Q_p)$ can be represented with $E$ defined over $\Z_p^{\unr}$ and having good reduction. This obstruction occurs when $v_p(j(E))>0$ is not divisible by $3$, or when $v_p(j(E)-1728)>0$ is odd.

Consider the residue disc of a point $\ol u=(\ol E,\ol{[\phi]})\in X(\F_p)$. If $j$ is a local parameter at $\ol u$, then this disc contains points with $v_p(j(E))>0$ not divisible by $3$; similarly, if $j-1728$ is a local parameter at $\ol u$, the residue disc contains points with $v_p(j(E)-1728)>0$ odd. These are precisely the points for which the above obstruction can occur. This occurs at the reductions of the elliptic points of order $3$, respectively order $2$.

By contrast, if $j$ has order of vanishing $3$ at $\ol u$, respectively if $j-1728$ has order of vanishing $2$ at $\ol u$, then every point in the residue disc has $v_p(j(E))$ divisible by $3$, respectively $v_p(j(E)-1728)$ even, and the obstruction does not arise. 

There are several possible ways to handle the residue discs in which $j$ or $j-1728$ is a local parameter: one can for example allow representatives over ramified extensions of $\Q_p$. Alternatively, provided that the computations are carried out with sufficiently high $p$-adic precision, one can restrict to the  subsets defined by $v_p(j(E))\equiv 0\pmod 3$, respectively
$v_p(j(E)-1728)\equiv 0\pmod 2$. 
\end{remark}

\begin{example}\label{ex: curve that doesn't lift}
When $p=7$ and $E$ is the elliptic curve $y^2=x^3+x$ of $j$-invariant $1728$, there are $\F_p$-rational points $(E, [\phi])$ of $X_{\ns}^+(13)(\F_7)$ for which all lifts $(\widetilde{E}, [\widetilde{\phi}]) \in X_{\ns}^+(13)(\Q_7)$ with $j(\widetilde{E}) \neq 1728$ have bad reduction. 
\end{example}

\subsection{Heegner points}
\label{sec:Heegner}

\subsubsection{Heegner points on \texorpdfstring{$X_\ns^+(N)$}{Xns+N}}

In this section we discuss Heegner points on the curve $X_{\ns}^+(N)$. Our definition is a special case of \cite[Definition 3.1]{KohenPacetti} when the integer $m$ in loc.~cit.~is equal to $1$ (and hence the auxiliary cyclic ideal $\mathfrak{m}$ is trivial).
\begin{defi}
Let $\calO $ be an  imaginary quadratic order of conductor prime to $N$ and in which $N$ is inert (that is, $\calO/N\calO$ is isomorphic to the finite field $\F_{N^2}$). 
Let $E/F$ be an elliptic curve with potential complex multiplication by $\calO$, where $F=\Q(j(E))$ is a number field.

After choosing an $\calO$-module generator of $E[N]$, we may identify $E[N]$ with $\calO/N\calO$ as an $\calO$-module. 
Any choice of $\Z/N\Z$-basis of $\calO/N\calO$ yields an isomorphism $\phi \colon E[N] \to (\Z/N\Z)^2$ and an embedding $(\calO/N\calO)^\times \hookrightarrow \GL_2(\Z/N\Z)$, whose image is a non-split Cartan subgroup. 
There is a  choice of $\phi \colon E[N] \to (\Z/N\Z)^2$  such that the image of $(\calO/N\calO)^\times$ in $\GL(E[N])$ is equal to $C_\ns(N)$ when represented with respect to the basis $\phi^{-1}$. Such a $\phi$ is unique up to postcomposition by $C_{\ns}^+(N)$, and
the pair $(E, [\phi])$ gives an $F$-point of $X_{\ns}^+(N)$, which depends only on $E$ (see \cite[Appendix 5, top of page 195]{MR1757192}). 
We will call $(E, [\phi])$ a \emph{Heegner point} of order $\calO$.
\end{defi}

\begin{defi}\label{def Heegner}
    Let $\calO$ be an imaginary quadratic order of conductor prime to $N$ in which $N$ is inert. The associated \emph{Heegner divisor} on $X_{\ns}^+(N)$ is 
    \[
    D_\calO = \sum_{\End(E) \cong \calO} (E, [\phi]) ,
    \]
    where the sum is over all (isomorphism classes of) elliptic curves $E/\ol \Q$ with potential complex multiplication by $\calO$ and for each such $E$, the point $(E, [\phi])$ is the associated Heegner point. 
\end{defi}

\begin{remark}
    Note that, for every Heegner point $(E, [\phi])$ and every $\sigma \in \operatorname{Gal}(\overline{\Q}/\Q)$, the point ${}^\sigma (E, [\phi])$ is also a Heegner point associated to the same quadratic order. 
    In particular, fixing an elliptic curve $E/F$ with $\End(E) \cong \calO$, where as above $F=\Q(j(E))$,  we have 
    \[
    D_\calO = \sum_{\sigma \in \operatorname{Gal}(\overline{\Q}/\Q) / \operatorname{Gal}(\overline{\Q}/F)} {}^\sigma (E, [\phi]),
    \]
where the sum is taken over any set of coset representatives.
    In particular, $D_\calO$ is a divisor defined over $\Q$.
\end{remark}

The study of Heegner points on non-split Cartan curves was initiated in work of Kohen and Pacetti \cite{KohenPacetti} following work on the curves $X_0(N)$ by Gross, Zagier, and Zhang.
\begin{remark}\label{rmk: Heegner points want to generate}
    For each $\mathcal{O}$, let $n_\mathcal{O}$ be the degree of $D_\mathcal{O}$. Let $\xi$ be the cuspidal divisor of $X_{\ns}^+(N)$, of degree $n_\xi = \frac{N-1}{2}$. The degree 0 divisor $n_\xi D_\mathcal{O} - n_\mathcal{O} \xi$ is rational and therefore represents a point in $\operatorname{Jac} X_{\ns}^+(N)(\Q)$. Kohen and Pacetti \cite[Section~3]{KohenPacetti} give evidence that, when $\operatorname{Jac} X_{\ns}^+(N)$ has analytic rank equal to its dimension,
    these divisors and their Hecke images generate $J(\Q)$ up to finite index. Since we want to avoid cusps, we will never work with the divisors $n_\xi D_\mathcal{O} - n_\mathcal{O} \xi$ themselves, but rather take degree 0 linear combinations of the $D_{\mathcal{O}}$.
\end{remark}

\subsubsection{Computational representation of Heegner points}

Given an imaginary quadratic order $\calO$ in which $N$ is inert, we explain how we represent all Heegner points associated to $\calO$. The sum of such Heegner points gives the divisor $D_\calO$.
We list the $j$-invariants corresponding to elliptic curves with CM by $\calO$ as roots of the corresponding Hilbert class polynomial. 
We fix a number field $H$ and a prime $\pp$ of $H$ over $p$ such that, for each such $j$, there exists an elliptic curve $E/H$  with that $j$-invariant, having good reduction at~$\pp$. 
Up to enlarging the base field $H$, we can also assume that $\End_H(E) = \calO$. 

It remains to describe the Cartan structure $\phi$ on $E$ such that $(E, [\phi])$ is a Heegner point and it is enough to describe the reduction of $\phi$ modulo $\pp$. Denote by $\overline E$ the reduction of $E$ modulo $\pp$. Depending on whether $\ol E$ is ordinary or supersingular, we compute the reduction of $\phi$ as follows.

\begin{itemize}
    \item 
In the ordinary case, since $N$ is inert in $ \calO = \End(E)$, it is also inert in the possibly larger  imaginary quadratic ring $\ol \calO \coloneq \End(\ol E)$ and the embedding $\calO \hookrightarrow \ol\calO$ gives an isomorphism $\calO/N\calO \cong \ol\calO/N\ol\calO$. In particular, if  the Frobenius morphism $\Frob_\pp \colon \ol E \to \ol E$
is not in $\Z + N \ol\calO$, we also have $\ol \calO/N\ol\calO  \cong \F_N[\Frob_\pp]$.
By choosing a basis $P,Q$ of $\overline{E}[N]$ such that the matrix of $\Frob_\pp$ with respect to $P,Q$ belongs to $C_{\ns}(N)$, we obtain the desired $C_{\ns}^+(N)$-structure.

Note that, since by hypothesis $N$ does not ramify in $\calO$, the lift of the Frobenius lies in $\Z + N \ol\calO$ if and only if the discriminant of the characteristic polynomial of Frobenius is divisible by $N^2$. When this happens, 
we use the technique described below for the supersingular case.

\item In the supersingular case, $\End(\overline{E})$ is a quaternion algebra, so $\Frob_\pp$ no longer determines the Cartan structure. 
Instead, we must determine how $(\calO/N \calO)^\times$ sits inside the quaternion algebra. 

Let $\ell$ be a prime that is the norm of some element $\alpha \in \calO \setminus (\Z+N\calO)$, so that we have $\calO/N\calO \cong \F_N[\ol\alpha] \cong \F_{N^2}$.  
Then $E/H$ has an endomorphism $\alpha$ of degree $\ell$ in characteristic~$0$, unique up to taking the dual and composing with elements of $\Aut(E)$. Our goal is to identify the possible reductions $\overline{\alpha} \in \End(\overline{E})$ of this endomorphism among the degree-$\ell$ endomorphisms of $\overline{E}$.
We first compute, using V\'{e}lu's formulas, all endomorphisms of degree $\ell$ on $\overline{E}$.
Let $E_n$ be the reduction of $E$ modulo $\pp^n$. Using V\'{e}lu's formulas, we can check which endomorphisms of $E_{n-1}$ lift to endomorphisms of $E_n$. 
We iterate this process for $n\ge 2$, until the number of candidates for $\overline{\alpha}$ is equal to the number of elements of $\calO$ having norm equal to $\ell$.
\begin{remark}
    An endomorphism of $\overline E$ comes from the CM order if and only if it lifts to every infinitesimal thickening of $\overline E$ induced by $E/H$. Since there are only finitely many degree-$\ell$ endomorphisms, finite precision is enough to distinguish those that lift to $E$.
\end{remark}

By choosing a basis $P,Q$ of $\overline{E}[N]$  such that the matrix of $\overline{\alpha}$  with respect to $P,Q$ 
lies in $C_{\ns}(N)$, we obtain the desired $\phi$.
\end{itemize}

\subsection{Mordell--Weil group}\label{subsec:computing-MW}

Recall that we assume that the rank $r$ of $J(\Z)$ is equal to the genus $g$ of $X_{\ns}^+(N)$. As explained in  Remark \ref{remark: GQC works}, this can be checked by looking at the analytic rank. For  $N<100$, this information is also available on the L-functions and Modular Forms Database \cite{lmfdb}. 
Our choice of coordinates on $\Mx(\Z_p)$ in Section \ref{sec:kappa_general} is made so as to simplify the expression of the $\kappa$ map.
Computing this coordinate system requires a finite index subgroup $\Gamma < J(\Q)$  such that:
\begin{enumerate}
    \item \label{item:Gamma mod p} $\Gamma$ surjects onto $J(\F_p)$; 
    
    \item \label{item:Gamma_0 dense} 
    $\Gamma^0 \coloneq \Gamma \cap J(\Z)_0 = \ker\big(\Gamma \to J(\F_p)\big)$ surjects onto $J(\Z/p^2\Z)_0 = \ker\big(J(\Z/p^2\Z) \to J(\F_p)\big)$; 

\item \label{item:Gamma-in-J0}
    $\Gamma  \subset J^0(\Z)$.

\end{enumerate}

\begin{remark}
    For our application, \ref{item:Gamma mod p} can be replaced with the assumption that the image of $\Gamma \to J(\F_p)$ coincides with the image of $J(\Z) \to J(\F_p)$. Proving that this holds for a given $\Gamma$ amounts to showing saturation of $\Gamma$ in $J(\Z)$ at a finite number of primes (namely, the divisors of $\#J(\F_p)$). This is a problem that often comes up in the context of the Mordell--Weil sieve, see e.g.~\cite[\S 6.3]{furio20267adicgaloisrepresentationselliptic}.
\end{remark}

We look for $\Gamma$ by first constructing Heegner divisors $D_1, \ldots, D_s$ (see Definition \ref{def Heegner}) and then defining $\Gamma$ to be the subgroup generated by degree-0 combinations of the $D_i$ that lie in $J^0(\Z)$.
\begin{remark}\label{rmk:Gamma-in-J0}
    Assume that $N$ does not divide the discriminant associated with any of our Heegner divisors. Then each point in the support of a $D_i$ reduces modulo $N$ to the smooth locus of $X_{\ns}^+(N)^{\rm reg}$. Moreover, the component where each point specializes is determined by (the reduction modulo $N$ of) its $j$-invariant.
    We can then use the techniques of \cite{EdPar_regular} and Section \ref{NewRegularModel} to find linear combinations of the $D_i$ that lie in $J^0(\Z)$.
\end{remark}
Using the methods outlined in Section~\ref{sec:makdisi_modular_P}, we can construct a Makdisi model of $J$ over $\Z_q/p^e\Z_q$ for some moderate power $q$ of $p$ and some small integer $e \geq 2$. We can then check conditions \ref{item:Gamma mod p}--\ref{item:Gamma-in-J0} as follows:
\begin{enumerate}
    \item It is enough to check that $\Gamma$ generates the $\ell$-Sylow $J_\ell$ of $J(\F_p)$ for every prime $\ell \mid \#J(\F_p)$. 
    Since $X_{\ns}^+(N)$ is a modular curve, $\# J(\F_{p^d})$ can be computed for any $d \in \N$ from the action of Hecke. Write $n \coloneq \#J(\F_p)$. 
    Let $\ell$ be a prime factor of $n$, and let $\ell^v$ be the highest power of $\ell$ that divides $n$, i.e.\ the order of $J_\ell$.

    We can determine the exact order of any point of $J(\F_p)$, since we know $n$ and since Makdisi's algorithms allow us to apply the group law and test for $0$ in $J(\F_p)$. In particular, if $J_\ell$ is cyclic, then we can easily prove it, since a random element of $J(\F_p)$ will have order divisible by $\ell^v$ with probability $1-1/\ell \gg 0$. In this situation, to prove that $\Gamma$ generates $J_\ell$, we check that at least one of the generators of $\Gamma$ has order divisible by $\ell^v$. This situation is rather frequent since it occurs in particular whenever $v = 1$.

    When we are not able to prove that $J_\ell$ is cyclic (so we strongly suspect that it is not), we use a variant of the techniques of \cite[\S 6.2]{NicJac}. Suppose first that $\ell \mid q-1$, so that $\mu_\ell \subset \F_q$ and we can use the Frey--R\"uck pairing~\cite[Algorithm 6]{NicJac}
    \[ [ \cdot , \cdot ]_\ell : J(\F_q)[\ell] \times J(\F_q)/\ell J(\F_q) \longrightarrow \F_q^\times / \F_q^{\times \ell} \underset{z \mapsto z^{\frac{q-1}\ell}}{\simeq} \mu_\ell, \]
    which is non-degenerate since $\mu_\ell \subset \F_q$. By picking random points $x \in J(\F_q)$, we can thus construct linear forms $[\cdot,x]_\ell$ on $J(\F_q)[\ell]$, through which we can detect linear dependencies on $J(\F_q)[\ell]$ or prove that none exist. This in turn allows us to elucidate the structure of the $\ell$-Sylow of $\Gamma \bmod p$ (and, in particular, to determine whether its order is $\ell^v$) by a variant of \cite[Algorithm 13]{NicJac}.

    If $\ell \nmid q-1$ and $\ell \neq p$, we can still extend scalars to $\F_q(\mu_\ell)$ if this extension is not too large. Otherwise, we are reduced to determining the structure of the $\ell$-Sylow of $\Gamma \bmod p$ by checking ``by hand'' which relations hold between the generators of $\Gamma \bmod p$. This is of course very inefficient, but fortunately we expect this situation to be quite rare, since for $\ell \neq p$ it requires $v > 1$ and $[\F_q(\mu_\ell):\F_q] \gg 1$ (so this cannot happen unless $\ell$ is both large and a repeated factor of $n$). Since the case $\ell=p$ can be annoying, it is a good idea to choose $p$ so that $J_p$ is cyclic (or even trivial).

A byproduct of this computation is a description of the finite abelian group $\Gamma \bmod p$ and of the relations satisfied by the reductions modulo $p$ of the generators of $\Gamma$. This gives in particular a set of generators for $\Gamma^0 \coloneq \ker(\Gamma \longrightarrow J(\F_p))$. Note that this group is torsion-free, because the same holds for $J(\Z)_0 \subset J(\Z_p)_0$.

    \item Mascot~\cite[\S 2.2.3]{NicJac} introduces ``local  charts'' $J(\Z_q/p^e\Z_q)_0 \longrightarrow (\Z/p^{e-1}\Z)^{[\F_q:\F_p]g}$, which restrict to maps
    \[ c : J(\Z/p^e\Z)_0 \longrightarrow (\Z/p^{e-1}\Z)^g, \]
    where $e \in \N$ is any fixed integer.
    These maps are likely to be bijective (although this is not guaranteed), and when they are, 
    they yield (the reduction modulo $p^{e-1}$ of)
    a system of analytic coordinates, in the sense of Section~\ref{sec:parameters}. Note that for $e=2$, the reduction modulo $p$ of any system of analytic coordinates is automatically a group homomorphism. This applies in particular to the map $c$ above (when $e=2$).
    
    Using the system of generators of $\Gamma^0$ computed in part (1), we choose $e=2$ and test whether $\dim_{\F_p} c(\Gamma^0)$ is $g$ or less.
    If it is $g$, this proves that $\Gamma^0$ surjects to $J(\Z/p^2\Z)_0$, and in particular --- using our standing assumption $\rank \, J(\Z)=g$ --- that $\Gamma$ has finite index in $J(\Z)$ (and also that $c$ is bijective). If it less than $g$, then either $c$ is not bijective (unlikely, but we can try again with another chart $c$), or $\Gamma^0$ fails to surject onto $J(\Z/p^2\Z)_0$.

    \item This is satisfied by construction.

\end{enumerate}

\subsection{Lifting from \texorpdfstring{$J(\Z) \times J^0(\Z)$}{J(Z)xJ0(Z)} to \texorpdfstring{$\Mx(\Z)$}{Mx(Z)}}\label{subsect: lifting to MxZ}

Let $x \in J(\Z)$ and $y \in J^0(\Z)$. We describe how to find a computational representation of a lift of $(x,y)$ to $\Mx(\Z)$. As with all of our computations, points are represented over $\Z_p$, up to precision $p^e$, so we will always view $\Mx(\Z)$ as a subset of $\Mx(\Z_p)$ and represent the latter via points in $\Mx(\Z_p/p^e\Z_p)$. We will only be able to find such a representation up to $(p-1)$th 
roots of unity. Because of our coordinate system (defined below, see Definition \ref{defi:psi-prime}), which is based on the choice of $p$-adic logarithm in Section \ref{ref: p-adic log}, this ambiguity is not relevant in calculations.

Suppose first that $x, y$ are represented by Galois-stable divisors $D = \sum r_i P_i$, $E = \sum s_j Q_j$ of degree 0 whose supports have disjoint reductions modulo all places of $\ol\Q$ above $p$ (in particular, the points $P_i$ are pairwise distinct from the points $Q_j$). Choose a number field $K$ such that all the points $P_i,Q_j$ are in  $X(\calO_K)$. For each finite  place $\mathfrak{q}$ of $\calO_K$ we denote by $m_\mathfrak{q}(P_i,Q_j)$ the arithmetic intersection multiplicity of $P_i, Q_j$ at $\mathfrak{q}$. Since $y = [E]$ lies in $J^0(\Z)$, there exists a divisor $B_E$ supported on $X_{\F_N}$ such that $E+B_E$ is of multidegree $0$ as a divisor on the arithmetic surface $X \to \Spec \Z$. We denote by $e_i$ the arithmetic intersection multiplicity of $B_E$ and $P_i$ at $N$. 
For each $i$, let $\lambda_i\coloneq
N^{-e_i}\prod_{\mathfrak q}
\prod_j \mathfrak q^{-s_j m_{\mathfrak q}(P_i,Q_j)}$, where the product is over the nonzero prime ideals $\mathfrak{q}$ of $\calO_K$. By Galois invariance, we may view $\lambda \coloneq \prod_{i} \lambda_i^{r_i}$ as a fractional ideal of $\Q$, hence as a non-zero rational number well-defined up to $\pm 1$. The point
\begin{equation}\label{eq:easy Z-points on M}
\lambda \bigotimes_i ( P_i^*1)^{\otimes r_i}
\in \Mx_{x,y}(\Q)
\end{equation}
is then a lift of $(x, y)$ to $\Mx(\Q)$. By \cite[Section 6.9]{Edil}, this $\Q$-point extends to a $\Z$-point of $\Mx$ lifting $(x, y)$.

In practice, we are faced with three computational tasks:
\begin{enumerate}
    \item Avoid self-intersections, that is, choose representative divisors so that no $P_i$ is equal to a $Q_j$. To do this, we use the methods of Section~\ref{subsec:computing-MW} to look for congruences 
    $x \equiv \sum_i r_i D_i $ and $y \equiv \sum_j s_j E_j$ modulo~$p^e$ with $D_i, E_j$ being pairwise distinct Heegner divisors and with $ \sum_j s_j E_j \in J^0(\Z)$, for a chosen precision level $e$. Lemma \ref{lem:onlycareaboutlastcoord} ensures that the formula \eqref{eq:easy Z-points on M}, applied to $\sum_i r_i D_i $ and $ \sum_j s_j E_j$, gives a $\Z$-point on $\Mx$ which, up to $(p-1)$th roots of unity, has the same reduction modulo~$p^e$ as a lift of $(x,y)$. 

    Note that we can also impose additional restrictions on our Heegner divisors: for example, we can consider only Heegner divisors corresponding to maximal orders, and we may take the discriminant of each $D_i$ to be coprime to the discriminant of each $E_j$. 

    \begin{remark}
        While we do not have a theoretical guarantee that one can generate $J^0(\Z)$ with (linear combinations of) Heegner divisors as above, Remark \ref{rmk: Heegner points want to generate} gives a strong indication in this sense.
        In the case of $X_{\ns}^+(13)$, we were  able to carry out this step using only rational CM points (see Example \ref{eg:mwgroup13}).
    \end{remark}
    
    \item Compute the intersection multiplicities $m_{\mathfrak{q}}(P_i, Q_j)$.  Arguing as in the previous step, we may assume that $P_i$ and $Q_j$ are Heegner points, corresponding to coprime fundamental discriminants. 
    The intersection numbers can then be read off Theorem B in \cite{LSV26}. In the lucky situation where rational CM points suffice, one can also use \cite[Table 1]{LSV26}. See also \cite[Proposition 15]{rebolledo2026computingnecklacesellipticcurves} for related results.

    \item Finally,  \eqref{eq:easy Z-points on M} describes a point on the Mumford torsor in naive form, which we can convert to Makdisi form as in Section \ref{sec:naive2Makdisi}.
\end{enumerate}

\begin{example}
\label{eg:mwgroup13}

Let $N=13$ (so that $m=1$ and thus $J(\Z)=J^0(\Z)$, and $g=r=3$), $p=5$, let $P_1,\cdots,P_7$ be the rational Heegner points of respective discriminants $-11$, $-19$, $-67$, $-7$, $-163$, $-28$, $-8$,
let $\Gamma$ be the subgroup of $J(\Z)$ spanned by the divisors of degree 0 supported by these 7 points $P_i$, and let \[ \gamma_1 \coloneq [P_1-P_7] \in \Gamma, \quad \gamma_2 \coloneq [P_2-P_4] \in \Gamma, \quad \gamma_3 \coloneq [P_5-P_6] \in \Gamma \]
be candidates for generators of $\Gamma$.
We check that $\Gamma$ satisfies the hypotheses listed in Section \ref{subsec:computing-MW} by applying the methods outlined there. 
More specifically, we find that 
\begin{itemize}
\item $J(\F_p)$ has order $n=377=13 \cdot 29$ and is therefore cyclic,
\item $\gamma_1$ generates $J(\F_p)$ since neither $13 \gamma_1$ nor $29 \gamma_1$ is $0$ in $J(\F_p)$, 
\item we have the relations $\gamma_2 \equiv 55 \gamma_1$ and $\gamma_3 \equiv -65 \gamma_1$ in $J(\F_p)$,
\item and the elements
\[ \gamma_1^0 \coloneq 377 \gamma_1 \in \Gamma^0, \quad \gamma_2^0 \coloneq \gamma_2- 55 \gamma_1 \in \Gamma^0, \quad \gamma_3^0 \coloneq \gamma_3+65 \gamma_1 \in \Gamma^0 \]
of $\Gamma^0 \coloneq \ker\big(\Gamma \longrightarrow J(\F_p)\big)$ form an $\F_p$-basis of $J(\Z/p^2\Z)_0 = \ker\big(J(\Z/p^2\Z) \longrightarrow J(\F_p)\big)$.
\end{itemize}

We also observe that the relations
\[ [P_j-P_7] \equiv \sum_{i=1}^3 R_{i,j} \gamma_i \]
hold in $J(\Z_p/p^e\Z_p)$ at least all the way up to $e=4$ (we actually checked them up to $e=10$), where
\[ R = \begin{pmatrix} 1&2&7&2&-4&-4\\0&1&1&0&1&1\\0&1&4&1&-2&-3\end{pmatrix}.\]
The simplicity of these relations makes us suspect that they actually hold in $J(\Z)$ so that, in particular, $\Gamma$ would be generated by $\gamma_1, \gamma_2, \gamma_3$; but we do not need to prove this as long as we work mod $p^e$ with $e \leq 4$ (and $e=4$ will indeed be the highest accuracy that we will have to use, cf.~Table \ref{tab:example} below).

For example, in order to compute the image in $\Mx(\Z_p/p^e\Z_p)$ of a lift to $\Mx(\Z)$ of $(\gamma_1,\gamma_1)$, we use the relation matrix $R$ and a bit of linear algebra to find the alternative expression
\[ \gamma_1 \equiv [2 P_2 - P_3-P_5] \bmod p^e \]
which represents $\gamma_1 \coloneq [P_1-P_7]$ modulo $p^e$ but involves neither $P_1$ nor $P_7$; so we can use the right partial group law on lifts of $(P_1-P_7,P_2-P_3)$ and of $(P_1-P_7,P_2-P_5)$ to obtain a mod $p^e$ approximation of a lift of $(\gamma_1,\gamma_1)$ to $\Mx(\Q)$, thus circumventing the intersection of the divisor $P_1-P_7$ with itself. We then evaluate the intersection multiplicities $m_q(P,Q)$ for $P \in \{ P_1, P_7 \}$, $Q \in \{P_2,P_3,P_5\}$, and $q \in \{2,3,13\}$ so as to find the unique (up to sign) scalar in $\Q^\times$ by which this point must be multiplied so as to land in $\Mx(\Z)$.
\end{example}

\subsection{Coordinates on \texorpdfstring{$\Mx(\Z_p)$}{Mx(Zp)}} \label{sec:computing coordinates on mumford}

In this section, we explain the system of coordinates we use in actual computations. We work in $\Mx(\Z/p^e \Z)$, for some choice $e \in \N$ of $p$-adic accuracy. 
Given $u \in \Mx(\Z/p^e \Z)$, we write $(x(u),y(u))$ $\in J(\Z/p^e \Z) \times J^0(\Z/p^e \Z)$ for the point above which $u$ lies.

Let $\Gamma \subset J^0(\Z)$ be a subgroup satisfying the conditions in Section \ref{subsec:computing-MW} with generators $\gamma_i$.
Let  $\Gamma^0 = \Gamma \cap J(\Z)_0 = \ker ( \Gamma \to J(\F_p) )$ be the kernel of reduction, with generators
\begin{equation*}
    \gamma^0_j = \sum_{1 \le i \le h} M_{i,j} \gamma_i , 
    \quad j=1, \cdots, g,
\end{equation*}
for suitable $M_{i,j} \in \Z$.
Property \ref{item:Gamma_0 dense} implies that
  we can compute the coordinates in $(\Z/p^{e-1}\Z)^g$ of any point of $J(\Z/p^e \Z)_0$ on the $\gamma^0_i$.

Let $\gamma_{i,j}$ be lifts of $(\gamma_i, \gamma_j) \in J(\Z) \times J^0(\Z)$ to $\Mx(\Z)$. 
These exist by property \ref{item:Gamma-in-J0}, that is, the hypothesis that $\Gamma$ is contained in $J^0(\Z)$, and are unique up to $\pm 1$; we choose them arbitrarily. They can be computed to  precision $p^e$, up to $(p-1)$th roots of unity, as in Section~\ref{subsect: lifting to MxZ} (cf.~Example~\ref{eg:mwgroup13} in particular). 

Let $\ol t \in \Mx(\F_p)$.  By properties \ref{item:Gamma mod p} and \ref{item:Gamma_0 dense} of $\Gamma$, we can find integers $s_i$ and $r_i$ such that
\begin{equation}\label{eqn:fix_decomp_t}
    x\coloneq \sum_{i=1}^h s_i \gamma_i  \equiv  x( \ol t)  \bmod p, \quad y\coloneq \sum_{i=1}^h r_i \gamma_i \equiv  y(\ol t) \bmod p.
\end{equation}
These integers $r_i$ and $s_i$ are not unique, but we fix a choice of $r_i$ and $s_i$  once and for all. In particular, $x,y$ are well defined points in $J(\Z)$ and $J^0(\Z)$, respectively.

We now describe a computable system of coordinate for points $u \in \Mx(\Z/p^e \Z)_{\ol t}$. We will then show that they agree with the reduction mod $p^e$ of the coordinates defined in Section~\ref{sec:kappa_general}.

Fix a point $u \in \Mx(\Z/p^e \Z)_{\ol t}$. Property \ref{item:Gamma_0 dense} allows us to write 
\begin{equation}    
x(u)-x = \sum_{j=1}^g m^0_j \gamma_j^0, \quad y(u)- y = \sum_{j=1}^g n^0_j \gamma_j^0   \quad \text{in } J(\Z/p^e \Z)_0
\end{equation}
 for some coefficients $m^0_j, n^0_j \in \Z$, uniquely determined modulo $p^{e-1}$. We then use \eqref{eqn:fix_decomp_t} to deduce the expression
\[ x(u) \equiv \sum_{i=1}^h m_i \gamma_i \bmod p^e, \quad m_i = s_i + \sum_{j=1}^g M_{i,j} m^0_j \in \Z. \]
Applying the same process, we express $y(u)$ as a combination of the $\gamma_i$:
\[ y(u) \equiv \sum_{i=1}^h n_i \gamma_i \bmod p^e, \quad n_i = r_i + \sum_{j=1}^g M_{i,j} n^0_j \in \Z.
\]
We let $m, n$ (resp.~$m^0, n^0$) be the vectors in $\Z^h$ with coordinates $(m_i)_{i=1,\ldots,h}, (n_i)_{i=1,\ldots,h}$ (resp.~the vectors in $\Z^g$ with coordinates $(m^0_i)_{i=1,\ldots,g}, (n_i^0)_{i=1,\ldots,g}$).
The bilinear combination ${}^\top m \gamma n$ gives a point of $\Mx(\Z)$. We define 
\begin{equation}\label{eq:definition_uprime}
u' \coloneq {}^\top m \gamma n \bmod p^e, \quad u' \in \Mx(\Z/p^e\Z).
\end{equation}
By construction, $u'$ is a point in $\Mx_{x(u) \bmod p^e, \, y(u) \bmod p^e}(\Z/p^e\Z)$, so there exists a unique\\ $z \in \Gm(\Z/p^e \Z)$ such that $u \equiv z \cdot u' \bmod p^e$.  
We are finally ready to define the coordinate system we use in explicit computations.
\begin{defi}\label{defi:psi-prime}
We define 
\begin{equation}
\begin{array}{rcl} \Psi' : \Mx(\Z/p^e \Z)_{\ol t} & \longrightarrow & (\Z/p^{e-1}\Z)^g \times (\Z/p^{e-1}\Z)^g \times \Z/p^{e-1}\Z \\ u & \longmapsto & 
\left( m^0, n^0,\frac1p \log(z) \right).
\end{array} 
\end{equation} 
Here we use the $p$-adic logarithm fixed in Section~\ref{ref: p-adic log}, so that the logarithm of every $(p-1)$th root of unity vanishes.
\end{defi}

We connect this construction to our coordinate system in Section \ref{sec:kappa_general}. Let $\gamma$ be the matrix of the points $\gamma_{i, j}$. Take the points $x_j  = y_j=\gamma_j^0$ for $1 \leq j \leq g$ as generators of $J(\Z_p)_0$ and let
\[
R^\Z \coloneq {}^\top M \gamma r \in \Mx(\Z)^g, \quad S^\Z \coloneq {}^\top s \gamma M \in \Mx(\Z)^g, \quad P^\Z\coloneq {}^\top M \gamma M ,
\]
where the coordinates $r_i, s_i$ of the vectors $r, s$ are defined in \eqref{eqn:fix_decomp_t}.
This is compatible with the notation in Section \ref{sec:kappa_general}, in the sense that the $i$th element of the column vector $R^\Z$ is a lift of $(x_i, y)$, the $j$th coordinate of the row vector $S^\Z$ is a lift of $(x, y_j)$, and the $(i,j)$-entry of $P^\Z$ is a lift of $(x_i, y_j)$. Moreover we have $x_i = y_i$. As in Section \ref{sec:kappa_general}, up to multiplying by suitable $(p-1)$th roots of unity, we obtain a matrix $P$, respectively, vectors $R$ and $S$, whose entries reduce modulo~$p$ to the identity point $1 \in \Mx(\F_p)$, respectively $1_y$ and $1_x$.
Moreover, we can consider the lift $\BLC{s}{\gamma}{r} \in \Mx(\Z)$ of $(x,y)$ and define $t \coloneq \xi \cdot \BLC{s}{\gamma}{r}$ for a suitable $(p-1)$th root of unity $\xi$, chosen so that $t$ reduces to our fixed $\ol t \in \Mx(\F_p)$.

These choices allow us to use the construction of Section \ref{sec:kappa_general} to obtain an analytic system of coordinates $\Psi$ on $\Mx(\Z_p)_{\ol t}$. 
The next proposition shows that $\Psi'$ coincides with $\Psi$ to precision~$p^{e-1}$.

\begin{prop}\label{prop:psi-vs-psidash}
\phantom{~}
    For every $u \in \Mx(\Z/p^e\Z)_{\ol t}$ we have
    \[
    \Psi(u) \equiv \Psi'(u) 
    \bmod p^{e-1}.
    \]  
\end{prop}

\begin{remark}\label{rmk:choice-coordinate-system}
    In particular, $\Psi'$ is the reduction modulo~$p^{e-1}$ of analytic coordinates on $\Mx(\Z_p)$ at $\ol t$. In these coordinates, $\jbtilde$ is quadratic modulo~$p$ (see Corollary \ref{cor:jbtilde quadratic}). 
\end{remark}

\begin{proof}
   Let $A^\Z, B^\Z, C^\Z$ and $D^\Z$ be the maps defined in \eqref{eq: AZ BZ CZ DZ}. Recall that, for every $u \in \Mx(\Z/p^e\Z)_{\ol t}$, the coordinates $\Psi'(u)$ are defined in terms of $m^0, n^0, z$, and $u'$, as in Definition \ref{defi:psi-prime} and \eqref{eq:definition_uprime}.
   A short calculation gives
\begin{align*}
{}^{\top} m \gamma n 
&= \left( {}^{\top} m^0 ({}^\top M\gamma M) n^0 \Rtimes {}^{\top} m^0 ({}^{\top} M\gamma) r \right) \Ltimes \left( {}^{\top}  s (\gamma M) n^0  \Rtimes {}^{\top} s \gamma r \right) 
\\ & = \left( C^\Z(m^0, n^0) \Rtimes B^\Z(m^0) \right) \Ltimes \left( A^\Z(n^0) \Rtimes t \right) = D^\Z(m^0, n^0).
\end{align*}
Recall that $D^\Z$ is the same as the map $D$ in \eqref{eq: A B C D}, up to $(p-1)$th roots of unity and that by Lemma \ref{lem: coords mod pe} the map $\Psi$ gives a bijection $\Mx(\Z/p^e\Z)_{\ol t} \to (\Z/p^{e-1}\Z)^{2g+1}$. Then there exists a $(p-1)$th root of unity $\zeta$ such that the point $u'$ of \eqref{eq:definition_uprime} satisfies
\[
u'  = {}^{\top} m \gamma n = \zeta \cdot  D(m^0, n^0)=  \zeta \cdot \Psi^{-1} \left( m^0, n^0, 0 \right),
\]
where all the equalities take place in $\Mx(\Z/p^e\Z)$.
Hence, we have
\[
\Psi^{-1}( m^0, n^0, \tfrac 1p \log(z)) =  \exp 
\log(z) \cdot \Psi^{-1}( m^0, n^0, 0)  = \exp 
\log(z) \cdot\zeta^{-1} \cdot  u' =  z \zeta' \cdot u' = \zeta'  \cdot u
\]
for a suitable $(p-1)$th root of unity $\zeta'$. Since $\Psi^{-1}( m^0, n^0, \tfrac 1p \log(z)) $ and $u$ lie in the same residue disc, we deduce that $\zeta'=1$ and $\Psi^{-1}( m^0, n^0, \tfrac 1p \log(z)) = u \bmod p^e$, hence
\[
\Psi'(u) = (m^0, n^0, \frac{1}{p} \log z) = \Psi(u) \quad \text {in} \quad (\Z/p^{e-1}\Z)^{2g+1}.\qedhere
\]
\end{proof}

\subsection{Hecke operators}\label{sec:hecke_computational}
In this section, we explain how to compute the image of a point $(E, [\phi])$ in $X_{\ns}^+(N)(\Z_q)$ under the action of a Hecke operator $T_\ell$, where
$\ell$ is a prime not equal to $p$ or $N$.
\begin{enumerate}
\item (Field extension) 
Compute a field extension of $\Q_q$ over which all $\ell$-isogenies of $E$ are defined: this is the unramified extension whose degree $d$ is the order of the image of the $q$-power Frobenius in ${\rm PGL}(E[\ell])$, so it is enough to compute a multiple of this $d$. This Frobenius has characteristic polynomial $\chi \coloneq x^2-a_q x + q$, where $a_q$, the trace of this Frobenius, can be computed as $a_q = q+1-\# E(\F_q)$. If $\operatorname{disc} \chi = a_q^2-4q \neq 0 \in \F_\ell$, then $\chi$ has distinct roots $\lambda_1, \lambda_2 \in \F_{\ell^2}^\times$, and $d = \ord(\lambda_1 / \lambda_2)$; otherwise, $d \in \{ 1, \ell \}$.

\item (Create kernel polynomials) Factor the $\ell$-division polynomial of $E/\F_q$. Then, if needed, recombine the factorization until one obtains the kernel polynomials of degree--$\ell$ isogenies:
let $m \in (\Z/\ell\Z)^\times$ be a generator. Compute $[m]$, the multiplication-by-$m$ map on the $x$-coordinate of $E$. Take the closure of the orbits of the irreducible factors under the action of $[m]$ to obtain the  kernel polynomials. To compute the orbit closures in practice, we write $[m](x,y) = \left( \frac{\phi_m(x)}{\psi_m(x)^2}, \frac{\omega_m(x)}{\psi_m(x)^3}  \right)$, take an irreducible factor $g(x)$ (which factors over some extension as $\prod_{i=1}^d (x-x_i)$) of the $\ell$-division polynomial $\psi_\ell(x)$, and apply standard resultant techniques to compute $\prod_{i=1}^d \left(x - \frac{\phi_m(x_i)}{\psi_m(x_i)^2}\right)$. This gives either $g(x)$, or another irreducible factor of $\psi_\ell(x)$ corresponding to the same $\F_\ell$-line inside $E[\ell]$.

\item (V\'elu's formula) Hensel lift the kernel polynomials to polynomials over $\Q_q$. Using V\'elu's formulas, compute the $\ell$-isogenies associated with these kernel polynomials.

\item (Push forward Cartan structures) For each of the $\ell$-isogenies $\psi: E \to E'$ over $K$, construct the point $(E', [\phi \circ (\psi|_{E[N]})^{-1}]  ) \in X_{\ns}^+(N)(K)$. The image of $(E, [\phi])$ under $T_\ell$ is the sum of these $\ell+1$ points.
\end{enumerate}

We also need to compute $\Delta^* T_\ell$, which we do as follows. We first compute the $j$-invariants of elliptic curves with a self $\ell$-isogeny as the roots of $\Phi_\ell(x,x)$, where $\Phi_\ell$ is the classical modular polynomial for isogenies of degree $\ell$. For each such $j$-invariant, corresponding to an elliptic curve $E_j$, we construct all self-isogenies of degree $\ell$ by iterating over $\ell$-torsion points and using Vélu's formulas to check if the codomain is isomorphic to $E_j$. For each self-isogeny $\varphi$, we enumerate all the Cartan structures that are stable under $\varphi$, thereby obtaining a complete description of all the points in $\Delta^*T_\ell$. The multiplicities of the points appearing in this divisor are determined by Lemma \ref{MultDT} (or by Corollary \ref{cor:delta_cap_Tell}). For $\ell\in \{2,3\}$, see also Example \ref{examplesT_N}.

\subsection{The \texorpdfstring{$\jbtilde$}{jbtilde} map on a point}\label{sec:explicit_jbtilde}

 Fix a simple open $U$ of $X = X_{\ns}^+(N)$, a base point $b \in X(\Q)$, and a trace zero endomorphism $f: J \to J$ of the form $a_1 T_\ell + a_2 T_n$, where $T_\ell, T_n$ are Hecke operators and $a_1, a_2 \in \Z$. These choices determine a line bundle $\calL$ as in Section \ref{sec:jbtilde}, and hence a map $\jbtilde : U \to \Mx$. In this section, we explain how to evaluate $\jbtilde$ on a point of $U(\Z_p)$  to any prescribed finite $p$-adic precision.

We choose $b = (E_b, [\phi_b]) \in X(\Q)$ given by a CM point (see also Remark \ref{rmk:no_base_point}).
We assume that $\ell,n$ are coprime to $pN$ and that $(b_\Q, b_\Q)$ lies on neither $T_\ell$ nor $T_n$.

\begin{remark}
To apply Corollary \ref{cor:delta_cap_Tell}, it is convenient to take $\ell, n$ to be primes not congruent to $\pm1$ modulo~$N$. 
To compute the divisor $B$ in \eqref{eq:L_is_divisorial}, we apply the method of Section \ref{sec:V}. Since this method depends on Proposition \ref{prop:flatness}, we then require $\ell, n < N^2/4$.
\end{remark}

\begin{remark}
    The assumption that $(b_\Q, b_\Q)$ does not lie on $T_\ell$ nor on $T_n$ is a technical assumption that may be removed by proving a more general version of Lemma \ref{lem:practical_formula_j_btilde}. We do not pursue this here, since in practice the assumption is not very restrictive.
\end{remark}

Now we have the data to compute $\jbtilde$ on a single point $u \in X(\Z_p)$ such that $(u_{\Q_p}, u_{\Q_p})$ and $(u_{\Q_p}, b_{\Q_p})$ lie in neither $T_\ell$ nor $T_n$. In particular, for every $u_e\in U(\Z/p^e\Z)$ in the residue disc of $\bar u$, we can choose a lift $u\in U(\Z_p)$ satisfying these conditions and use it to compute $\jbtilde(u_e)$.  Since $\jbtilde$ is a morphism of schemes, the value $\jbtilde(u) \bmod p^e = \jbtilde(u_e)$ is independent of the choice of lift $u$.
We compute $\jbtilde(u_e)$ as follows:
\begin{enumerate}
    \item Compute the divisors
\begin{equation}\label{eq:computation_some_divisors_for_jbtilde}    
\Delta^* T_\ell, \quad  \Delta^*T_n, \quad T_\ell \cap (b\times X_{\Z_p}) =  T_\ell(b), \quad T_n \cap (b\times X_{\Z_p}) =  T_n (b), \quad  T_\ell(u), \quad T_n(u)
\end{equation}
  on $X_{\Z_p}$, using Section \ref{sec:hecke_computational}.
\item For each divisor $D$ in \eqref{eq:computation_some_divisors_for_jbtilde}, compute local equations for $D$ around the points $u$ and $b$, denoted $t_{D,u}$ and $t_{D,b}$, as in Remark~\ref{rmk:multiplicative_local_eqs_almost}. 
When a divisor is given by a linear or multiplicative combination of the divisors in \eqref{eq:computation_some_divisors_for_jbtilde}, compute its local equation by the corresponding multiplicative combination, as in Definition~\ref{def:multiplicative_local_eqs}.

    \item Determine the numbers $V_{U,N}$ and $V_{U_b,N}$, 
    which we recall are the multiplicities with which the component $U_{\F_N}$ appears in $B$ (resp.~the component of $X_{\F_N}$ containing $b_{\F_N}$ appears in $B_b$), using Section \ref{sec:V}.
    \item Compute $ \sigma_\Delta  \cdot \frac{\eqoverZ}{t_{D_V,b}}(b)$ using \eqref{eq:NOT_computing_sigma_Delta_2} and replace it in \eqref{eq:tildej_b_quasi_easy_Q_p} to compute $\jbtilde(u)$ in naive form.
    \item Represent $\jbtilde(u)$ in Makdisi's form using the method presented in Section~\ref{sec:naive2Makdisi}.
\end{enumerate}

\begin{remark}\label{rmk:precision_for_jbtilde}
In practice, the points and the rational functions are known with finite precision. In particular, when computing the ratios appearing in \eqref{eq:NOT_computing_sigma_Delta_2} and \eqref{eq:tildej_b_quasi_easy_Q_p}, there may be precision loss. If the loss is too great, then we start over with higher precision. By our assumption on $(u_{\Q_p}, u_{\Q_p})$ and $(u_{\Q_p}, b_{\Q_p})$, to compute the ratios $t_1(x)/t_2(x)$ to a given precision $O(p^e)$, we need to compute the parameters to precision $e+v$, where $v$ is the valuation of $t_2(x)$. On the other hand, as explained in Section~\ref{sec:naive2Makdisi}, no loss of precision will occur in the conversion from naive to Makdisi form, which is fortunate as this is the most expensive part of the computation.
\end{remark}

\subsection{Interpolating \texorpdfstring{$\jbtilde$}{jbtilde}}
\label{sec:interpolate}

Fix a residue disc corresponding to a point $\ol u \in U(\F_p)$. We are now in a position to explain how to compute with diagram \eqref{diag:gqc_coordinates} from the introduction, which we reproduce here.
\begin{equation}
	\begin{tikzcd}
	& U(\Z)_{\ol u} \arrow[rr,"\wt{j_b}"]
     \arrow[d]
    &&\ol{\Mx(\Z)}_{\wt{j_b}(\ol u)} \arrow[d] 
    & \Z_p^{2g}  \arrow[ddr] \arrow[l,"\kappa",twoheadrightarrow]  \arrow[dd]
    \\
	& U(\Z_p)_{\ol u} \arrow[rr,"\wt{j_b}"] \arrow[ld,"\beta"]&& \Mx(\Z_p)_{\wt{j_b}(\ol u)} \arrow[rd," \Psi"] \\
\Z_p \arrow[rrrr,"{\theta = (\theta_1, \ldots, \theta_{2g+1})}"] \arrow[rrrrr,"\theta_{2g+1}", bend right=20] &&&&
    \Z_p^{2g+1} \arrow[r,twoheadrightarrow] & \Z_p
	\end{tikzcd}  
    \label{diag:extended_Zp}
\end{equation}

We define the coordinate system $\Psi$ to be the one introduced in Section \ref{subsec:analytic-coordinates}. It depends on the choice of points $t, P_{i,j}, R_{i,y}, S_{x,j}$, which we take to be $\Z$-points multiplied by $(p-1)$th roots of unity (see Section \ref{sec:kappa_general}). 
Theorem~\ref{thm:Mx(Z) bar} thus ensures that the rightmost diagonal arrow $\Z_p^{2g} \to \Z_p$ on \eqref{diag:extended_Zp} is the 0 map whenever $\Mx(\Z)_{\wt{j_b}(\ol u)} \neq \emptyset$. 
When $\Mx(\Z)_{\wt{j_b}(\ol u)}$ is empty the above diagram does not apply, since the map denoted $\kappa$ in the diagram cannot exist, but we still have a map $\kappa \colon \Z_p^{2g} \to \Mx(\Z_p)_{\jbtilde(\ol u )}$, see Definition \ref{definition kappa}, that parametrizes the points in the residue disk of $\jbtilde(\ol u)$ that, up to multiplying by a $(p{-}1)$th root of unity, are in $\ol{\Mx(\Z)}$.

 We choose a local parameter $w$ at $\ol u$ (see Remark \ref{rmk: local parameter}), and we take $\beta \coloneq\tilde w$ as in \eqref{eq:alg_param}.
As in the diagram, we then let $\theta$ be the composition
\[
\theta \coloneq \Psi \circ \jbtilde \circ \beta^{-1} : \Z_p \to \Z_p^{2g+1}
\]
and $\theta_{2g+1}$ be the composition of $\theta$ with the projection onto the last coordinate of $\Z_p^{2g+1}$. The goal of this section is to explain how to compute the map $\theta_{2g+1}$, irrespective of whether $\Mx(\Z)_{\jbtilde(\ol u)}$ is empty or not.

The coordinate system $\Psi$ is somewhat abstract, and only accessible computationally if we work at finite $p$-adic precision (namely, projecting from $\Z_p^{2g+1}$ to $(\Z/p^{e-1}\Z)^{2g+1}$ for some fixed $e \geq 2$). Proposition \ref{prop:psi-vs-psidash} shows that the reduction modulo $p^{e-1}$ of $\Psi$ can be computed as the map $\Psi'$ of Section \ref{sec:computing coordinates on mumford}. For this reason, we define
\[ \beta' \coloneq \beta \bmod p^e : U(\Z/p^e\Z)_{\ol u} \to \Z/p^{e-1}\Z, \]
\[
\theta' =(\theta_1', \ldots, \theta_{2g+1}') \coloneq \Psi' \circ \jbtilde \circ \beta'^{-1} : \Z/p^e\Z \to (\Z/p^{e-1}\Z)^{2g+1},
\]
and note that, by the aforementioned Proposition \ref{prop:psi-vs-psidash}, the  reduction modulo $p^{e-1}$ of the power series $\theta_{2g+1}$ agrees with $\theta'_{2g+1}$. We may compute $\theta'_{2g+1}(x)$ on each given $x \in \Z/p^{e-1}\Z$ by forming the corresponding point $u = \beta^{-1}(x) \in U(\Z/p^e\Z)_{\ol u}$, computing $\jbtilde(u)$ as explained in the previous section, and then computing $\Psi'$ as explained in Section \ref{sec:computing coordinates on mumford}.

By Corollary \ref{cor:jbtilde quadratic}, $\theta_{2g+1}$ reduces modulo $p$ to an at most quadratic polynomial. Thus, to determine $\theta_{2g+1} \bmod p =\theta_{2g+1}' \bmod p$, it suffices to interpolate its values over three points $u_1,u_2,u_3 \in U(\Z_p)_{\ol u}$ with different reductions modulo~$p^2$, so that the numbers $x_i \coloneq\beta(u_i)$ are all distinct modulo~$p$. We pick the $u_i$ so that all the points $(u_{i,\Q_p}, u_{i,\Q_p})$ and $(u_{i,\Q_p}, b_{\Q_p})$ are not in the Hecke correspondences $T_\ell, T_n$: we can do this since there are only finitely many $u\in U(\Z_p)$ violating this last condition. 
Then, we can compute $\jbtilde(u_i)$ as in Section \ref{sec:explicit_jbtilde}, hence the values $\theta'(\beta_i) = \Psi'(\jbtilde(u_i))$ modulo~$p$. We then interpolate the quadratic polynomial $\theta'_{2g+1}$ modulo~$p$.

\begin{remark}\label{rmk:computing_theta_arbitrary_precision}
    The method of this section allows us to compute $\theta_{2g+1}(x) \bmod p^e$ for every $x \in \Z_p$ and every fixed $e \geq 1$. Based on this, we could (in principle)  compute an arbitrary number of coefficients in the power series development of $\theta_{2g+1}(x) = \sum_{n \geq 0} a_n x^n$, up to any fixed precision. As explained, we can compute $a_0 = \theta_{2g+1}(0)$ up to arbitrary precision. Now suppose we want to determine $a_{n+1}$ up to some precision $p^e$, and we have algorithms to compute $a_0, \ldots, a_n$ up to arbitrary precision. Compute $a_0, \ldots, a_n$ up to respective precisions $p^{(n+2)e}, p^{(n+1)e}, \ldots, p^{2e}$. Recalling that the coefficients $a_i$ are in $\Z_p$, the congruence
    \[
    \theta_{2g+1}(p^e) \equiv \sum_{i=0}^n a_i p^{ei} + a_{n+1}p^{(n+1)e} \bmod p^{(n+2)e}
    \]
    shows that we can determine
    \[
    a_{n+1} \equiv \frac{\theta_{2g+1}(p^e)-\sum_{i=0}^n a_i p^{ei}}{p^{(n+1)e}} \bmod p^e.
    \]
    Finally, note that we can write $\Z_p$ as the union $\bigsqcup_{x_0=0}^{p-1} (x_0+p\Z_p)$; on each residue disc $x_0+p\Z_p$, the power series $\theta_{2g+1}(x_0+px)$ in the variable $x$ reduces modulo $p^e$ to a polynomial of degree at most $e-1$. Combined with the above remarks, this allows us to compute the polynomial $\theta_{2g+1}(x_0+px) \bmod p^e$ for every $x_0 = 0, \ldots, p-1$ and every fixed precision $e \geq 1$.
\end{remark}

\subsection{Intersecting the images of \texorpdfstring{$\jbtilde$}{jbtilde}  and \texorpdfstring{$\kappa$}{kappa}}\label{sec:intersection}

In the previous section, we have explained how to compute, to finite $p$-adic precision, a power series $\theta_{2g+1}(x)$ representing the map $\jbtilde$ on a residue disc.
Our objective in this section is to intersect the image of $\jbtilde$ with the $p$-adic closure of the integral points $\Mx(\Z)$ in a given residue disc $\Mx(\Z_p)_{\ol t}$. Recall from the previous section that we have fixed an analytic coordinate $\beta$ on $U$ at the point $\ol u$, and that $\ol t = \jbtilde(\ol u)$.
By Theorem \ref{thm:Mx(Z) bar}, any integral point in $\Mx(\Z)_{\ol t}$ lies in $\Psi^{-1}\left( \Z_p^g \times \Z_p^g \times \{0\} \right)$: this is trivially true if $\Mx(\Z)_{\ol t}$ is empty, and follows from the quoted theorem otherwise. 
In particular, if $P$ is a point lying in $\overline{\Mx(\Z)_{\ol t}} \cap  \jbtilde (U(\Z))$, we may write $P = \jbtilde(\beta^{-1}(x))$ for some $x \in \Z_p$ and get the equation $\theta_{2g+1}(x)=0$.
Now recall the approximate coordinate system $\Psi'$ introduced in Definition \ref{defi:psi-prime} and the fact that it agrees with $\Psi$ modulo $p^{e-1}$, see Proposition \ref{prop:psi-vs-psidash}. As in the previous section, the power series $\theta_{2g+1}(x)$ and $\theta_{2g+1}'(x)$ coincide modulo $p^{e-1}$, and we know how to compute the latter.
In principle, Remark \ref{rmk:computing_theta_arbitrary_precision} allows us to compute $\theta_{2g+1}(x_0+px)$ to precision $p^e$, deduce information about its Newton polygon and the reduction of $\theta_{2g+1}(x_0+px) \bmod p^e$, and therefore bound the number of zeros of $\theta_{2g+1}(x_0+px)$ in $\Z_p$. Doing this for every $x_0=0,\ldots,p-1$ gives a bound on the number of zeros of $\theta_{2g+1}(x)$ in $\Z_p$.

In the rest of this section, we explain a different and  more efficient method that works under the assumption that $\theta_{2g+1}(x) \bmod p$ is nonzero and squarefree \textit{for at least one choice} of trace zero endomorphism $f$. We expect that this will be satisfied with very high probability.

The power series $\theta_{2g+1} \in \Z_p[[x]]$ is at most quadratic mod $p$ by Corollary \ref{cor:jbtilde quadratic}. 
By Strassmann's theorem, the number $d$ of solutions in $\Z_p$ to the equation $\theta_{2g+1}(x) = 0$ is at most the number of solutions to $(\theta_{2g+1} \bmod p)(x)  = 0$ in $\F_p$, counted with multiplicity. This gives a first upper bound on the number of points in $U(\Z)_{\ol u}$, and a modulo-$p$ congruence for the coordinate $x=\beta(P)$ of any such point.

\begin{remark}
When $U(\Z)_{\ol u}$ contains precisely one integral point, we generically expect the number of solutions to be $d=2$: indeed, $\theta_{2g+1}(x) \bmod p$ will usually be of degree $2$, and since it has a root (corresponding to the integral point), it has two (counted with multiplicity).
\end{remark}

If the N\'eron--Severi rank $\rho \coloneq \rho(J_\Q)$ is strictly greater than $2$, the bound on the number of integral points in $U(\Z)_{\ol u}$ can be improved,  by repeating the method with another choice of trace zero endomorphism $f$. After this process, we have several (at most quadratic mod $p$) power series $\theta^{f_1}_{2g+1}, \cdots, \theta^{f_{\rho-1}}_{2g+1}$. The number of solutions of $\gcd( \theta^{f_1}_{2g+1} \bmod p, \cdots, \theta^{f_{\rho-1}}_{2g+1} \bmod p)  = 0$ is an upper bound for the number of integral points in the residue disc we are considering. Note that this $\gcd$ is squarefree under our assumption that some $\theta^{f_i}_{2g+1}(x) \bmod p$ is squarefree. If this upper bound is the same as the number of known integral points in $U(\Z)_{\ol u}$, we have provably determined $U(\Z)_{\ol u}$. Otherwise, we proceed as follows. 

Suppose that a common root $x_0 \in \F_p$ of the equations $ \theta^{f_1}_{2g+1} \bmod p = \dots = \theta^{f_{\rho-1}}_{2g+1} \bmod p  = 0$ does not correspond to a known integer point. Up to renumbering, we assume that the power series $\theta^{f_1}_{2g+1}(x)$ does not have a double root modulo $p$, hence in particular that $x_0$ is a simple root of $\theta^{f_1}_{2g+1} \bmod p$. This implies that $\frac{d\theta^{f_1}_{2g+1}}{dx}(x_0) \not \equiv 0 \pmod p$.
We can try to rule out the existence of integral points in the `higher residue disc' (as in \cite[Remark~9.9]{DRHS})
\begin{equation}\label{eq:higher-residue-disc-1}
 \{ P \in U(\Z_p)_{\ol u} : \beta(P) \equiv x_0 \bmod{p} \}
\end{equation}
consisting of points reducing to a specified $\Z/p^2\Z$-point in the residue disc of $\ol u$. 

Via the coordinate $\beta$, the above higher residue disc corresponds to the subset $x_0 + p\Z_p$ of $\Z_p$. For any trace zero endomorphism $f$, when restricted to $x_0 + p\Z_p$, the integral power series $\theta_{2g+1}^f(x)$ takes the form
\begin{equation}
\theta_{2g+1}^f (x_0 + px_1) = \theta_{2g+1}^f(x_0)  + \frac{d\theta_{2g+1}^f(x)}{dx} (x_0) \cdot px_1 + O(p^2); 
\end{equation}
by assumption, $\theta_{2g+1}^f(x_0) \equiv 0 \pmod{p}$, so we may divide by $p$ to obtain the (linear) congruence
\begin{equation}\label{eq:higher-congruence-1}
\frac{\theta_{2g+1}^f(x_0)}{p}  + \frac{d \theta^f_{2g+1}}{dx} (x_0) \cdot x_1 \equiv 0 \pmod{p}.
\end{equation}
We now observe that for $f=f_1$, the derivative $\frac{d \theta^f_{2g+1}}{dx} (x_0)$ is nonzero modulo $p$ by assumption, so this gives $x_1 \equiv - \frac{\theta_{2g+1}^{f_1}(x_0)}{p}  \cdot \left(\frac{d \theta^{f_1}_{2g+1}}{dx} (x_0)\right)^{-1} \bmod p$.

We then check whether, for $f=f_2, \ldots, f_{\rho-1}$, the value of $x_1 \bmod p$ thus determined satisfies \eqref{eq:higher-congruence-1}. If this fails for even one $f$, we have proven that the higher residue disc \eqref{eq:higher-residue-disc-1} does not contain any integral points. Otherwise, we have determined the value of $x_1 \bmod p$, and  hence of $x_0+p x_1 \bmod p^2$, and we can keep going, by considering an (even) higher residue disc.

We explain more precisely how we may iterate this procedure. Suppose that for some $e \geq 1$ and $x_0 \in \Z/p^e\Z$ the higher residue disc
\begin{equation}\label{eq:higher-residue-disc-e}
 \{ P \in U(\Z_p)_{\ol u} : \beta(P) \equiv x_0 \bmod{p^e} \}
\end{equation}
is suspected not to contain any integral points, but we haven't yet ruled this out. This means in particular that all the functions $\theta_{2g+1}$ produced so far vanish at $x_0$ modulo $p^e$.
We then fix an arbitrary lift of $x_0$ to $\Z_p$ and consider the equation
\[
0 = \theta_{2g+1}^f(x_0 + p^e x_1) = \theta_{2g+1}^f(x_0) + \frac{d\theta^f_{2g+1}(x)}{dx}(x_0) p^e x_1 + O(p^{2e}).
\]
By assumption, we have $\theta_{2g+1}^f(x_0) \equiv 0 \pmod{p^e}$, so we may divide by $p^e$ to obtain the linear congruence
\begin{equation}\label{eq:linear-congruence-higher-residue-disc}
\frac{\theta_{2g+1}^f(x_0)}{p^e} + \frac{d\theta_{2g+1}^f}{dx}(x_0) \cdot x_1 \equiv 0 \pmod p.
\end{equation}
Note that we only need $\frac{d\theta_{2g+1}^f}{dx}(x_0) \bmod p$, which we have already computed, and $\frac{\theta_{2g+1}^f(x_0)}{p^e} \bmod p$, which we can compute as in the previous section.
As before, taking $f=f_1$ determines $x_1 \bmod p$. If this value does not satisfy one of the equations \eqref{eq:linear-congruence-higher-residue-disc} corresponding to $f_2, \ldots, f_{\rho-1}$, we have shown that the higher residue disc \eqref{eq:higher-residue-disc-e} does not contain any integral points. Otherwise, we have determined one more $p$-adic digit of $\beta(P)$, where $P$ is our putative integral point.

In particular, if the common vanishing locus of the $\theta_{2g+1}^{f_1}, \dots, \theta_{2g+1}^{f_{\rho-1}} \in \Z_p[[x]]$ is the set of integral points $U(\Z)_{\ol u}$, then this describes a finite algorithm to determine this set of integral points.
If a zero corresponds to an unknown integer point, we can describe it to arbitrary $p$-adic precision and then eventually recognize it as an integer point (even though we do not know a priori the required precision for the computation to terminate, since that would require for instance some bound on the height of that unknown rational point).

\begin{remark}
The recognition of an unknown integral point $P$ may be done as follows. We may assume that $j(P) \neq 0, 1728$, because the CM points are easy to determine. The procedure described above computes increasingly more precise approximations of $\beta(P)$, hence of $j(P)$. Computing the rational numbers of smallest height compatible with these approximations gives some candidate $j' $ for the value of $j(P) \in \Q$. 
Given $j'$, we can construct an elliptic curve $E' /\Q$ with that $j$-invariant, and use known techniques for the computation of the images of Galois representations (see e.g.~\cite{ZywinaExplicitOpenImagesElliptic}) to determine 
if $E' [N]$ is indeed endowed with some
$\operatorname{Gal}(\ol \Q/\Q)$-stable Cartan structures. 
If it is not the case, we can compute $\beta (P)$ with higher precision, whence some new candidate $j''$ for $j(P)$, and so on.
This will describe the corresponding integral point(s) on $X_{\ns}^+(\Q)$. In fact, using known results on the mod-$N$ images of elliptic curves, it is not hard to show that an elliptic curve $E/\Q$ carries at most one normalizer of non-split Cartan structure of level $N \geq 13$ defined over $\Q$ (indeed, the existence of two different such structures forces the Galois action on $E[N]$ to have projective image of size at most $4$, contradicting \cite{zywina2015possibleimagesmodell}.  Alternatively, the  uniqueness follows from Proposition~\ref{GaloisOrbitsSS} by embedding $\Q$ in $\Q_N$.) 

Note however that
if the above algorithm does not succeed in finding a rational point after a certain number of steps, we are unable to  decide if that is because we do not have enough precision, or  because no such rational point exists.
\end{remark}

\section{Summary of the equationless method}
\label{sec: summary}

We now give an overview of the method for computing rational points using the equationless geometric Chabauty method  on  $X = X_{\ns}^+(N)$, under the hypothesis that $r=g$. Recall from~\eqref{eq:m} that $m$ is the lcm of the exponents of $(J_{\F_\ell}/J^0_{\F_\ell})(\overline{\F}_\ell)$ for $\ell$ prime. We pick a prime $p$ as in Assumption \ref{assump:properties-of-p}.
Let  $U \subset X$ be a simple open and $\ol u \in U(\F_p)$. 
The following steps describe the method to bound $U(\Z)_{\ol u}$. One then iterates Steps \ref{step:heckebasis} and \ref{step:higher-residue-discs}
over all choices of $U$ and $\ol u$ to bound $X_{\Q}(\Q) = X(\Z) = \bigcup_{U,\ol{u}} U(\Z)_{\ol u}$.

\begin{enumerate}
\item Precompute information about the Mordell--Weil group: As in Section~\ref{subsec:computing-MW}, find generators $\gamma_i$ for a suitable subgroup $\Gamma \subset J(\Q)$ (satisfying properties \ref{item:Gamma mod p}--\ref{item:Gamma-in-J0} in that section) and find a basis  of $\ker(\Gamma \to J(\F_p))$. Also find points in $\Mx(\Z)_{\gamma_i,\gamma_j}$ for all $i$ and $j$ as explained in Section~\ref{subsect: lifting to MxZ}.

\item \label{step:heckebasis} 
Let $f_1, \dots, f_{\rho-1}$ be linearly independent trace-zero endomorphisms of $J$ that are combinations of two Hecke operators. For each $f_j= f= a_1 T_\ell + a_2 T_n$, repeat the following steps:

\begin{enumerate}
 \item \label{item:computejbtilde} 
Choose three lifts $u_1, u_2, u_3 \in U(\Z_p)_{\ol u}$ with distinct images in $U(\Z/p^2\Z)_{\ol u}$ as in Section \ref{sec:enumerate} and compute $\jbtilde(u_i) \in \Mx(\Z/p^2 \Z)$ in Makdisi form as in Section~\ref{sec:explicit_jbtilde}.

\item \label{step5b}
\label{item:describecoords}  Describe coordinates on the residue disc $\Mx(\Z_p)_{\jbtilde(\ol u)}$: Apply the method of Section \ref{sec:computing coordinates on mumford} to obtain a map $\Psi'$ defining coordinates on the residue disc of $\tilde j_b(\ol u)$ (which is equal to the reduction modulo $p$ of $\jbtilde(u_1)$).

\item  This residue disc also contains $\tilde j_b(u_i)$ for all $i=1,\cdots,3$, so apply $\Psi'$ (with precision $e=2$) to obtain
\[ (a^i_1, \ldots, a^i_{2g+1}) = \Psi'(\tilde j_b(u_i)) \in \F_p^{2g+1}. \]

\item \label{item:intersectkapppajbtilde} 
    Discard the $a_1^i, \cdots, a_{2g}^i$, and keep $a_{2g+1}^i$. Compute the unique polynomial $\theta^f_{2g+1}  \in \F_p[x]$ of degree at most $2$ that interpolates $\beta(u_i) \mapsto a^i_{2g+1}$, $i=1,2,3$, for $\beta$ a parameter at $\ol u$ as in  Section \ref{sec:interpolate}.
\end{enumerate}

\item\label{step:higher-residue-discs} After iterating through all of the endomorphisms $f = f_1, \dots, f_{\rho-1}$ from Step \ref{step:heckebasis}, the number of roots in $\F_p$ of $\gcd(\theta^{f_1}_{2g+1}, \cdots, \theta^{f_{\rho-1}}_{2g+1})$, counted with multiplicity, gives an upper bound on the number of integral points in $U(\Z)_{\ol u}$. If this is not equal to the number of known integral points in the residue disc, proceed to the higher residue discs (with higher $p$-adic precision) as described in Section~\ref{sec:intersection}.
\end{enumerate}

\section{The example of level 13}\label{sec: example}
We now explain the method applied to the case $N = 13$. We choose $p = 5$. We take the non-square $\varepsilon = 2 \in \F_{13}^\times$ in the definition of $C_\ns^+(13)$, cf.~Section~\ref{subsect: modular curves}. 
In this case, there is only one irreducible component in the fiber above $N = 13$ and so $m = 1$, which simplifies the method described in Section \ref{sec:V}. In particular, we have $V_{13} = 0$ (see \eqref{eq:sigma_Delta_sigma_b}) 
and only one simple open which consists of the smooth locus of a regular model of $X = X_{\ns}^+(13)$.

First, we note that there are $7$ rational Heegner points on $X_{\ns}^+(13)$. We verify that their pairwise differences generate $J_\Q(\Q)$ up to finite index which is coprime to $p$ using the methods described in Section~\ref{subsec:computing-MW}. See Example \ref{eg:mwgroup13} for more details on these computations.

Since $\rho=3$, we have two independent trace zero endomorphisms, which we pick to be $f_1 = T_3 - T_2$ and  $f_2 = 2T_7 - 3T_2$. We apply the method of Section \ref{sec:intersection} to each endomorphism and each residue disc separately, and then combine the information to find the rational points.

In Table \ref{tab:example} we list the results of the computation. The top part of the table lists the 10 residue discs of $X$ along with their $j$-invariants and whether the residue disc contains a Heegner point $P$, along with the value of $x=\beta(P)$ mod $p$ when it does. For each disc, we then describe $\theta_{2g+1}$ modulo $p$ for $f_1$ and $f_2$ and their gcd.  From these polynomials we make the conclusion in the right hand column: if the gcd has the same number of roots as the expected number (0 if there are no Heegner points, 1 if there is one), then we are done. 
Otherwise, we proceed to the higher residue disc to rule out the existence of the extraneous root. For the higher residue disc corresponding to $x_0+p^ex_1$, the table gives the polynomial $\frac{\theta_{2g+1}^f(x_0+p^ex_1)}{p^e} \bmod p$. The calculation takes about 2 minutes per disc per trace zero endomorphism on one of the authors' laptops. 

\begin{table}[ht]
\footnotesize
\setlength{\tabcolsep}{3pt}
\renewcommand{\arraystretch}{1.3}

\begin{tabularx}{\textwidth}{|llllllX|}
\hline
\multicolumn{7}{|c|}{\textbf{Working in $\Mx$ mod $p^2$}} \\
\hline
\textbf{Residue disc} & $j \bmod p$ & \textbf{Heegner?} &
\textbf{$T_3-T_2$} & \textbf{$2T_7-3T_2$} &
\textbf{gcd} & \textbf{Conclusion} \\
\hline

1  & 0 & No & $3x^2$ & $4x$ & $x$ & Inspect $x=0+O(p)$ \\
2  & 0 & $\sqrt{-67}$, at $x \equiv 3$ & $x(x-3)$ & $3x(x-3)$ & $x(x-3)$ &Inspect $x=0+O(p)$ \\
3  & 0 & No & $4(x-1)^2$ & $3(x-1)$ & $x-1$ & Inspect $x=1+O(p)$ \\
4  & 0 & $\sqrt{-28}$, at $x \equiv 0$ & $2x$ & $x(x-4)$ & $x$ & \checkmark \\
5  & 0 & $\sqrt{-163}$, at $x \equiv 0$ & $3x(x-4)$ & $2x(x-1)$ & $x$ & \checkmark \\
6  & 0 & No & $4(x^2+x+2)$ & $x^2+x+2$ & $x^2+x+2$ & Irreducible $\Rightarrow$ \checkmark \\
7  & 0 & $\sqrt{-7}$, at $x \equiv 1$ & $2(x-1)(x-4)$ & $4(x-1)$ & $x-1$ & \checkmark \\
8  & 0 & $\sqrt{-8}$, at $x \equiv 3$ & $3(x-3)$ & $4(x-3)^2$ & $x-3$ & \checkmark \\
9  & 2 & $\sqrt{-11}$, at $x \equiv 0$ & $2x(x-3)$ & $2x(x-4)$ & $x$ & \checkmark \\
10 & 4 & $\sqrt{-19}$, at $x \equiv 0$ & $2x(x-2)$ & $x(x-1)$ & $x$ & \checkmark \\
\hline

\multicolumn{7}{|c|}{\textbf{Working in $\Mx$ mod $p^3$}} \\
\hline
\textbf{Higher disc} & $j \bmod p$ & \textbf{Heegner?} &
\textbf{$T_3-T_2$} & \textbf{$2T_7-3T_2$} &
\textbf{gcd} & \textbf{Conclusion} \\
\hline

$1$, $x = 0 + O(p)$ & 0 & No & $0$ & $4(x-4)$ & $x-4$ & Inspect $x=4p+O(p^2)$ \\
$2$, $x = 0+O(p)$ & 0 & No & $2(x-3)$ & $x-3$ & $x-3$ & Inspect $x=3p+O(p^2)$ \\
$3$, $x = 1 +O(p)$ & 0 & No & $1$ & $3(x-3)$ & $1$ & \checkmark \\
\hline

\multicolumn{7}{|c|}{\textbf{Working in $\Mx$ mod $p^4$}} \\
\hline
\textbf{Higher disc} & $j \bmod p$ & \textbf{Heegner?} &
\textbf{$T_3-T_2$} & \textbf{$2T_7-3T_2$} &
\textbf{gcd} & \textbf{Conclusion} \\
\hline

$1$, $x = 4p+O( p^2)$ & 0 & No & $4$ & $4(x-3)$ & $1$ & \checkmark \\
$2$, $x = 3p +O( p^2)$ & 0 & No & $2(x-3)$ & $x$ & $1$ & \checkmark \\
\hline

\end{tabularx}

\caption{Computation for $X_{\ns}^{+}(13)$ at $p=5$.}
\label{tab:example}
\end{table}

As an example of how to interpret the table, we describe in detail the computation for the first residue disc, which contains no known integral points (first line of the first section of Table~\ref{tab:example}). The polynomials $\theta_{2g+1}^{f_1}(x) \bmod p = 3x^2$ and $\theta_{2g+1}^{f_2}(x) \bmod p = 4x$ are not coprime, but their gcd is $x$, of degree $1$. This implies on the one hand that this residue disc contains at most one integral point, and on the other, that any integral point $P$ in this residue disc would have to lie in the subdisc where the coordinate $\beta$ is congruent to $0$ modulo $p$. We thus consider the higher residue disc $\beta(P)=px_1$ (first line of the second section in Table~\ref{tab:example}). Since the derivative of $\theta_{2g+1}^{f_1}(x) \bmod p$ vanishes at $x_0=0$, the polynomial $\frac{\theta_{2g+1}^{f_1}(px_1)}{p} \bmod p$ has degree $0$, as predicted by \eqref{eq:linear-congruence-higher-residue-disc}. The constant value of this polynomial is $0$, so congruence \eqref{eq:linear-congruence-higher-residue-disc} is automatically satisfied and we do not obtain any information about $x_1$. On the other hand, the derivative of $\theta_{2g+1}^{f_2}(x)$ at $x_0=0$ is equal to $4 \not \equiv 0 \pmod p$, so considering $\frac{\theta_{2g+1}^{f_2}(px_1)}{p} \bmod p = 4(x-4)$ determines $x_1 \equiv 4 \pmod p$. Thus, if the residue disc in question contains an integral point $P$, this point satisfies $\beta(P) = 4p + O(p^2)$. We then continue with the first line of the third section of Table \ref{tab:example}. We consider $\frac{\theta_{2g+1}^{f_1}(4p+p^2x_1)}{p^2} \bmod p$, which -- again as predicted by \eqref{eq:linear-congruence-higher-residue-disc} -- has degree $0$. However, this time the constant value of this polynomial is $4 \not \equiv 0 \pmod p$, so the congruence \eqref{eq:linear-congruence-higher-residue-disc} has no solutions, and we deduce that the higher residue disc $\beta(P) \equiv 4p \bmod p^2$, and hence the original residue disc, contain no integral points.

\bibliographystyle{myalpha}
\bibliography{bibliography.bib}
\end{document}